\documentclass[10pt]{amsart}

\usepackage{graphicx} 
\usepackage{amsmath}

\usepackage{amsthm}
\usepackage{fancyhdr,amssymb}
\usepackage{amsfonts}
\usepackage{hyperref}
\usepackage{tikz}
\usepackage{tikz-cd}
\usepackage[mathscr]{euscript}
\usepackage{enumerate}
\usepackage{bm}
\usepackage{hyperref}

\newcommand{\Z}{\mathbb{Z}}

\renewcommand{\O}{\mathcal{O}}
\newcommand{\Q}{\mathbb{Q}}

\renewcommand{\S}{\mathfrak{S}}
\newcommand{\id}{\text{id}}
\renewcommand{\P}{\mathbb{P}}
\newcommand{\A}{\mathbb{A}}
\newcommand{\C}{\mathbb{C}}

\DeclareMathOperator{\GL}{GL}
\newcommand{\ind}{\text{ind}}
\newcommand{\codim}{\text{codim}\,}
\newcommand{\Spec}{\text{Spec}}
\DeclareMathOperator{\Proj}{Proj}
\DeclareMathOperator{\T}{T}
\DeclareMathOperator{\TS}{TS}
\DeclareMathOperator{\mult}{mult}
\newcommand{\Pic}{\text{Pic}}
\newcommand{\Aut}{\text{Aut}}
\newcommand{\End}{\text{End}}
\newcommand{\NS}{\text{NS}}
\newcommand{\Hom}{\text{Hom}}
\newcommand{\Alb}{\text{Alb}}

\newcommand{\comment}[1]{}
\newtheorem{theorem}{Theorem}
\newtheorem {lemma}{Lemma}

\newtheorem{question}{Question}
\newtheorem {corollary}{Corollary}
\newtheorem {proposition}{Proposition}
\theoremstyle{definition}
 
\newtheorem{formula}{Formula}
\theoremstyle {definition}
\newtheorem*{remark}{Remark}
\begin{document}

\baselineskip=16.75pt
\title{Automorphisms of Hilbert Schemes of Points on Abelian Surfaces}
\author{Patrick Girardet}
\address{Department of Mathematics, University of California San Diego}
\email{pgirarde@ucsd.edu}

\begin{abstract}
Belmans, Oberdieck, and Rennemo asked whether natural automorphisms of Hilbert schemes of points on surfaces can be characterized by the fact that they preserve 
the exceptional divisor of non-reduced subschemes. Sasaki recently published examples, independently discovered by the author, of automorphisms on 
the Hilbert scheme of two points of certain abelian surfaces which preserve the exceptional divisor but are nevertheless unnatural, giving a negative answer to the question. We construct additional examples for abelian surfaces of unnatural automorphisms which preserve the exceptional divisor of the Hilbert scheme of an arbitrary number of points. The underlying abelian surfaces in these examples have Picard rank at least 2, and hence are not 
generic. We prove the converse statement that all automorphisms are natural on the Hilbert scheme of two points for a principally polarized abelian surface of Picard rank 1. Additionally, we prove the same if the polarization has self-intersection a perfect square. 
\end{abstract}
\maketitle

\tableofcontents{} 

\section{Introduction}
\noindent
Let $X$ be a smooth complex projective surface and let $X^{[n]}$ denote the Hilbert scheme of $n$ points on $X$. Any automorphism 
$g: X\rightarrow X$ induces an automorphism 
\[
g^{[n]}: X^{[n]} \rightarrow X^{[n]}
\]
defined by the action of the pullback $(g^{-1})^*$, where we view the points of $X^{[n]}$ as ideal sheaves. We call an automorphism 
of $X^{[n]}$ \textit{natural} if it arises as $g^{[n]}$ for some automorphism $g$ of $X$. 

This procedure induces an injective homomorphism 
\[
(-)^{[n]}:\Aut(X)\rightarrow \Aut(X^{[n]})
\]
of groups. The injectivity can be seen by considering the action of a natural automorphism on the locus of subschemes in 
$X^{[n]}$ supported at a single point with multiplicity $n$. The converse question of determining for which $X$ this map is surjective, 
i.e. when all automorphisms of $X^{[n]}$ are natural, has been studied for some time and is still open in general. 

A classic result in the subject is that of Beauville \cite[Section 6]{Beauville_unnatK3}, who showed that if 
$X$ is a general smooth quartic K3 hypersurface in $\P^3$ then one can construct an involution on $X^{[2]}$ which is not natural. 
The construction and proof of unnaturality are geometric in nature. Two points on $X$ define a line in $\P^3$ which will then intersect $X$ in a 
complementary length 2 subscheme, yielding an involution on $X^{[2]}$. This automorphism is not natural because any natural automorphism would need 
to preserve the exceptional divisor $E$ on $X^{[2]}$ of nonreduced subschemes, but this involution does not do so. 
Subsequent works have often focused on studying automorphisms of $X^{[n]}$ for K3 surfaces $X$ and generalized Kummer varieties as 
these are examples of hyperkähler varieties. Some of this work will be reviewed below. One may also study naturality of automorphisms for symmetric products of curves - this is largely settled, and we will review this work as well. 

More recently, Belmans, Oberdieck, and Rennemo studied this question of naturality of automorphisms of Hilbert schemes of points 
of surfaces in \cite{BOR}, proving that if $X$ is a weak Fano or general type smooth projective surface, then all automorphisms of $X^{[n]}$ 
are natural except for the case $X = C_1\times C_2$ where $C_1, C_2$ are smooth curves which either are both genus 0 or are both 
of genus at least 2. 
Thus, the question of naturality of automorphisms is mostly settled for surfaces to one end or another of the positivity spectrum, 
whereas the Calabi-Yau region in the middle can allow for interesting unnatural automorphisms. It has been shown that unnatural automorphisms necessarily do not preserve the exceptional divisor with $X$ either 
a K3 or Enriques surface \cite{boissiere_sarti}, \cite{hayashi_2018}.

The following question thus appears in \cite{BOR}:
\begin{question}\label{BOR_question}
Suppose $X$ is a smooth projective surface and $g:X^{[n]}\rightarrow X^{[n]}$ is an automorphism preserving the exceptional divisor. Unless 
$X=C_1\times C_2$ and $n=2$, does it follow that $g$ is natural? 
\end{question}
An affirmative answer to this question would imply that while unnatural automorphisms can exist, there is some measure of control 
on the groups $\Aut(X^{[n]})$ in that those elements which don't come from $\Aut(X)$ can be detected simply by considering their action on multiplicities of points. 

Recently, Sasaki gave examples in \cite{sasaki2023nonnatural} of automorphisms on the Hilbert scheme of two points of certain 
abelian surfaces which preserve the exceptional divisor but which can be shown to be unnatural, giving an answer in the negative to the previous 
question. These examples were independently discovered by the author and presented in a talk but not published at the time. 
Here we construct counterexamples that work for the Hilbert scheme of $n$ points of certain abelian surfaces for all $n$ (Subsection \ref{nilpotent_end_counterex}): 
\begin{theorem}
For all $n\ge 2$, there exist abelian surfaces $A$ and unnatural automorphisms $g:A^{[n]}\rightarrow A^{[n]}$ such that $g(E) = E$.
\end{theorem}
To do this, we use work of Ekedahl and Skjelnes \cite{ekedahl2014recovering} and Rydh and 
Skjelnes \cite{rydh2010intrinsic} describing the (smoothable locus of the) Hilbert scheme of points of an arbitrary scheme as a particular blowup of the 
symmetric product. However, our examples are not for generic abelian surfaces, but rather for specific abelian surfaces of Picard rank at least 2. 
These constructions rely on proving that $\S_n$-equivariant automorphisms on the Cartesian product $X^n$ induce 
automorphisms of the smoothable locus of $X^{[n]}$ for certain schemes $X$ (Proposition \ref{induced_aut_lemma}). We believe that this opens up the 
possibility of constructing many counterexamples to the question of Belmans-Oberdieck-Rennemo for higher dimensional abelian varieties using algebraic number theory, which we will demonstrate in one instance (Subsection \ref{higher_deg_counterex}). 

Furthermore, we address the natural converse question of whether automorphisms of $A^{[2]}$ are natural for abelian surfaces $A$ of 
Picard rank 1. We find an answer in the affirmative for many polarization types: 
\begin{theorem}\label{main_theorem}
If $A$ is a complex abelian surface of Picard rank 1 admitting a polarization $\Theta$ such that either:
\begin{itemize}
\item[(i)] $\Theta^2 = 2$, or
\item[(ii)] $\Theta^2$ is a perfect square
\end{itemize}
then all automorphisms of $A^{[2]}$ are natural. 
\end{theorem}
We suspect this result should hold for all polarization types, but we are unable to fully apply our techniques at present to all 
polarizations: 

\begin{question}\label{open_question}
If $(A,\Theta)$ is any polarized complex abelian surface of Picard rank 1, does it follow that all automorphisms of $A^{[2]}$ are natural? 
\end{question}

One may ask the same question for all $n$, and we present some evidence in Subsection \ref{higher_deg_counterex} consistent with an affirmative answer (Proposition \ref{no_mat_int_sol}). 
We discuss the obstructions to extending our proof to all polarization types or to higher values of $n$ in Section \ref{open_directions}.

Recall that if $A$ is a complex abelian surface of Picard rank 1 then its group automorphisms are simply multiplication by 
$\pm 1$ \cite[Lemma 1.2]{yoshioka2023}, and every automorphism of $A$ as a variety is the composition of a group automorphism 
followed by a translation \cite[Proposition 1.2.1]{BL_complex_abvar}. Thus, Theorem \ref{main_theorem} gives an explicit description of all 
automorphisms of $A^{[2]}$ for $A$ of Picard rank 1 and appropriate polarization. 

To put our result in perspective, other recent work studying automorphisms of Hilbert schemes of points for K3 surfaces includes calculations of the full automorphism 
groups for generic K3 surfaces in \cite{boissiere2016automorphism} for $n=2$ and 
a subsequent generalization to all $n$ \cite{cattaneo2019automorphisms}, and a recent determination of the automorphism 
group of the Hilbert square of Cayley's K3 surfaces \cite{lee2024automorphisms}. This list is hardly exhaustive of recent 
work on the general theme of automorphisms of Hilbert schemes and related ``punctual'' moduli spaces (e.g. punctual Quot schemes 
\cite{biswas2015automorphisms}). 

One can ask the analogous question for naturality of automorphisms of symmetric products of smooth projective curves. This question has been settled entirely except for the case $g = 2$. One can find a good review of the problem in \cite{ciliberto_sernesi_high_deg} 
which proves naturality of automorphisms of the symmetric product $C^{[d]}$ where $C$ is a smooth projective curve of genus $g$ 
for $g > 2, d \ge 2g-2$. The other cases can be found in \cite{ciliberto_sernesi_middle_deg} for $g > 2, g\le d\le 2g-3$, 
\cite{martens_torelli}, \cite{ran_martens} for $g > 2, 1\le d \le g-2$, \cite{catanese_ciliberto_symmell} for the case 
$g=1$ and all degrees $d$, and \cite{weil1957beweis} for the case $g > 2, d = g-1$ (in which case there is an additional involution 
on $C^{[g-1]}$ if $C$ is not hyperelliptic). For the case $g = d = 2$ it is possible to construct unnatural automorphisms 
by taking a curve $C$ whose Jacobian decomposes as a product $J(C) \cong E\times E$ of isomorphic elliptic curves - see \cite{hayashida_nishi} for examples of such curves. There are infinitely 
many automorphisms of $J(C)$ given by $\text{GL}_2(\Z)$ which fix the canonical divisor $K_C\in J(C)$, and hence lift to the blowup 
$C^{[2]}$ of the Jacobian at the point $K_C$ \cite{munoz_porras}. Though the author has not 
seen this example explicitly written, Ciliberto and Sernesi seem aware of the possibility in \cite{ciliberto_sernesi_high_deg}. 
Some of these naturality results were rediscovered more recently by Biswas and Gómez \cite{biswas_gomez} in the case $g > 2, d > 2g-2$. 

\subsection{Conventions}\hfill\\
Throughout this paper, ``automorphism'' will denote an automorphism as a variety, 
and any group automorphisms will be specifically designated. We will write $\End(A)$ for the group endomorphisms of $A$, and 
$\Aut(A), \Aut^0(A)$ for respectively the full automorphism group scheme (i.e. including translations) and the connected component of the identity. 
\subsection{Plan of the paper}\hfill\\
We will first state some necessary facts regarding Hilbert schemes of points (Subsection \ref{hilb_prelim}) and abelian surfaces (Subsection \ref{abelian_surface_prelim}). We will then outline the 
construction of some unnatural automorphisms of abelian surfaces and higher dimensional abelian varieties which nevertheless preserve the exceptional divisor (Section \ref{counterexample_section}). We calculate 
top intersection numbers of divisor classes in $A^{[2]}$ for $A$ an abelian surface of Picard rank 1 (Subsection \ref{intersection_number_subsection}) and the dimensions of various 
spaces of global sections of lines bundles on $A^{[2]}$ (Subsections \ref{global_sections_calcs}, \ref{global_sections_hilb2}). We then use this to show that the exceptional divisor is fixed by any given  
automorphism of $A^{[2]}$ for appropriate polarization types (Subsections \ref{descent_strategy}, \ref{fix_E_NS}, \ref{fix_E_Pic_subvar}). To do 
so, we will first show that any automorphism of $A^{[n]}$ descends to an automorphism of $A$ under the summation morphism $\Sigma:A^{[n]}\rightarrow A$ (Proposition \ref{summation_descent}), though this alone will not give the theorem as it holds regardless of Picard rank. The intersection numbers we calculate are valid for any polarization 
type while the dimensions of global sections of line bundles are specifically for the principally polarized case. This will yield an automorphism on the symmetric product $A^{(2)}$ which we can then lift to $A^2$ by results in \cite{BOR} and show naturality (Subsection \ref{A2_lift_naturality}).

\subsection{Acknowledgements}\hfill\\
The author wishes to thank Dragos Oprea and Shubham Sinha for many helpful discussions regarding automorphisms of Hilbert schemes of points on surfaces and of other moduli spaces. This work has been supported by NSF grant DMS-1502651. 

\section{Preliminaries}
\subsection{Hilbert schemes of points on surfaces}\label{hilb_prelim}\hfill\\
Let $X$ be a smooth complex projective surface. The Hilbert scheme of $n$ points on $X$, written $X^{[n]}$, parametrizes length 
$n$ subschemes of $X$. It is a smooth complex projective variety of dimension $2n$. The symmetric product $X^{(n)}$ is the 
quotient of the Cartesian product $X^n$ by the action of the symmetric group $\mathfrak{S}_n$ interchanging the factors, and is 
a complex projective variety of dimension $2n$ with singularities along the big diagonal $\Delta\subset X^{(n)}$ of points 
$\sum a_i x_i$ with some multiplicity $a_i \ge 2$. The Hilbert-Chow morphism 
\[
\pi: X^{[n]}\rightarrow X^{(n)}
\]
which sends a subscheme $\mathcal{Z}\in X^{[n]}$ to its support with multiplicities is an isomorphism on the locus of 
reduced subschemes where all points in the support have multiplicity 1, and hence is birational. It is well-known that $\pi$ is crepant 
(i.e. $\pi^*\omega_{X^{(n)}} = \omega_{X^{[n]}}$) and that for $n=2$ the Hilbert scheme $X^{[2]}$ is simply the blowup of 
$X^{(2)}$ along the ideal sheaf of the diagonal, so that $\pi$ is the blowdown map 
\[
\pi: Bl_{\Delta}X^{(2)} \rightarrow X^{(2)}.
\]
For all $n$, we write $E$ for the exceptional divisor of $\pi$, which is the preimage $\pi^{-1}(\Delta)$ of the big diagonal 
$\Delta\subset X^{(n)}$. 
We have a natural projection map 
\[
p: X^n \rightarrow X^{(n)},
\]
such that $p^*\Pic(X^{(n)}) = \Pic(X^n)^{\mathfrak{S}_n}$. Given a line bundle $L\rightarrow X$, we may form the box product 
\[
L^{\boxplus n} = p_1^*L\otimes ... \otimes p_n^*L
\]
on $X^n$, where each $p_i$ is a projection map to a factor of $X$. This line bundle is clearly symmetric, inducing a line 
bundle $L_{(n)}$ on $\Pic(X^{(n)})$ and thus gives a homomorphism 
\[
(-)_{(n)}: \Pic(X) \rightarrow \Pic(X^{(n)}),\, L \mapsto L_{(n)}.
\]
We can pull back $L_{(n)}$ under the Hilbert-Chow morphism to $X^{[n]}$, and will write $L_{[n]}$ for this line bundle $\pi^*L_{(n)}$ 
on $X^{[n]}$. Since $\pi$ is crepant, $\omega_{X^{[n]}} = (\omega_X)_{[n]}$ \cite[Section 2]{BOR}, so that if $X$ is $K$-trivial 
then so is $X^{[n]}$. In particular, if $X$ is an abelian surface then $\omega_{X^{[n]}}$ is trivial. 

If $X$ is a scheme of dimension at least 3, then some of the preceding discussion breaks down. There still exists a Hilbert-Chow 
morphism from $X^{[n]}$ to the symmetric product $X^{(n)}$. However, there can exist points in $X^{[n]}$ which are not limits 
of reduced subschemes, and these points can constitute multiple distinct irreducible components even when $X$ is irreducible 
(see \cite{iarrobino1972reducibility} for examples of this phenomenon). We will call the closure in $X^{[n]}$ of the locus of 
reduced subschemes the \textit{smoothable locus}, and write 
$X_{sm}^{[n]}$ for this component. For $\dim X \le 2$ the smoothable locus is just $X^{[n]}$. For $\dim X \ge 3$, it 
is not fully known for which values of $n$ the smoothable locus differs from $X^{[n]}$, only that there are examples where it does. 
Even more pathologically, in these examples the smoothable locus can be a component of positive codimension in $X^{[n]}$, violating 
our intuition for surfaces that $X^{[n]}$ consists ``mostly'' of collections of $n$ distinct points along with some scheme-theoretic 
behavior when these points run into each other. 
\newline\newline 
In \cite[Section 6]{fogarty2}, Fogarty proves for $X$ a surface that 
\begin{align*}
\Pic(X^n) &\cong \Pic(X)^n \times \Hom(\Alb(X), \Pic^0(X))^{\binom{n}{2}} \\
\Pic(X^{(n)}) &\cong \Pic(X^n)^{\mathfrak{S}_n} \cong \Pic(X) \times \Hom(\Alb(X), \Pic^0(X))^{\S_2}, 
\end{align*}
where the $\Pic(X)$ factor in $\Pic(X^n)^{\mathfrak{S}_n}$ injects into $\Pic(X)^n$ as the $\mathfrak{S}_n$-invariant divisors, $\Alb(X), \Pic^0(X)$ are 
respectively the Albanese and Picard varieties (connected component of the identity in $\Pic(X)$) of $X$. 
The term $\Hom(\Alb(X), \Pic^0(X))$ comes from classes of ``diagonal'' line bundles on $X\times X$, as a map $X\rightarrow \Pic^0(X)$ 
naturally gives a line bundle on $X\times X$ and such a map necessarily factors through $\Alb(X)$ since $\Pic^0(X)$ is an abelian 
variety. Note that $\Hom$ here refers to group homomorphisms rather than morphisms of varieties. In $\Pic(X^n)$ we get one such factor for each of the $\binom{n}{2}$ pairs in $X^n$. 
Since elements of $\Hom(\Alb(X), \Pic^0(X))$ correspond to classes of line bundles on $X\times X$, $\S_2$ acts on 
$\Hom(\Alb(X), \Pic^0(X))$ by swapping the factors of $X\times X$. In $\Pic(X^{(n)})$ the $\S_n$ action 
leaves but a single $\Hom$ factor, from which we take only the $\S_2$-invariant bundles.   

In a subsequent paper \cite{fogarty_sym}, Fogarty shows that the fixed part $\Hom(\Alb(X), \Pic^0(X))^{\S_2}$ can be identified with the 
Neron-Severi group of the Albanese variety of $X$ (see Theorem 3.8 in said paper and the subsequent table), so that 
\[
    \Pic(X^{(n)}) \cong \Pic(X)\times\NS(\Alb(X)).
\]
Additionally, if we set $B = \dfrac{E}{2}$ it is well-known that
\[
\Pic(X^{[n]}) \cong \pi^*\Pic(X^{(n)}) \times \Z[B],
\]
giving a full description of the Picard group of $X^{[n]}$:
\begin{align*}
    \Pic(X^{[n]}) &\cong \Pic(X)\times \Hom(\Alb(X), \Pic^0(X))^{\S_2}\times\Z[B] \\
                &\cong \Pic(X)\times\NS(\Alb(X))\times\Z[B].
\end{align*}
As noted by Lehn \cite[Lemma 3.7]{Lehn_1999} this latter generator 
$B = \frac{E}{2}$ is the same as $-c_1(\O_X^{[n]})$, where $\O_X^{[n]}$ is the tautological vector bundle obtained by pushing 
forward the structure sheaf $\O_{\mathcal{Z}_n}$ of the universal subscheme $\mathcal{Z}_n\subset X^{[n]}\times X$ under the natural 
projection to $X^{[n]}$.

The group of homomorphisms between any complex tori is a free abelian group \cite[Proposition 1.2.2]{BL_complex_abvar} and hence 
is discrete. The two abelian varieties $\Pic^0(X), \Pic^0(X^{[n]})$ have the same dimension 
(see e.g. \cite[Theorem 2.3.14]{göttsche2006hilbert}, \cite{Soergel1993}, \cite{cheah_hilbert_hodge} 
and the fact that the dimension of $\Pic^0$ is the irregularity $\dim H^{0,1}$ - see \cite{hodge-diamond-cutter} to calculate this easily), so that $\Pic^0(X^{[n]})$ lives inside the $\Pic(X)$ factor of $\Pic(X^{[n]})$. This implies 
\[
\NS(X^{[n]}) \cong \NS(X) \times \NS(\Alb(X)) \times \Z[B].
\]

We note that 
while much work on the subject has focused on the case where $X$ is a surface of irregularity zero so that $\Pic(X^{[n]})\cong \Pic(X)\times\Z[B]$, for a surface with $\dim\Alb(X) = h^1(X, \O_X)\neq 0$ the $\NS(\Alb(X))$ factor need not be zero, giving an extra 
term in the formula for $\Pic(X^{[n]})$. 
\newline\newline
We can describe $X^{[2]}$ as either the blowup of $X^{(2)}$ along the ideal sheaf of the diagonal or as the quotient of the blowup 
$Bl_{\Delta}X^2$ of $X^2$ along the diagonal $\Delta$ by the natural involution, see Appendix \ref{Sn_blowup_description} for 
a more detailed discussion. 
These descriptions fit in a natural commutative diagram 
\begin{center}
\begin{tikzcd}
Bl_{\Delta} X^2 \arrow[d, "q"'] \arrow[r, "\hat{\pi}"] & X^2 \arrow[d, "p"] \\
{X^{[2]}} \arrow[r, "\pi"]                     & X^{(2)}           
\end{tikzcd}.
\end{center}
Let us write $F$ for the exceptional divisor of $q:Bl_{\Delta}X^2\rightarrow X^{[2]}$, so that $q^*B = F$ as divisor classes. 
Since $X^2$ is smooth (hence normal) and $\S_2$ is finite, the quotient $X^{(2)}$ is normal \cite[p.5]{mumford1994geometric}. By \cite[Lemma 7.11, 7.12]{debarre2001higher}, the pushforward $\pi_*\O_{X^{[2]}}(mB)$ is 
isomorphic to the trivial sheaf $\O_{A^{(2)}}$ for $m\ge 0$. Thus, if $L$ is a line bundle on $X^{(2)}$ we have by push-pull: 
\begin{formula}\label{pushforward_formula}
    If $m\ge 0$ then
    \[H^0(X^{[2]}, \pi^*L\otimes\O_{A^{[2]}}(mB)) \cong H^0(X^{(2)}, L).\]
\end{formula}
\noindent
The same is true for arbitrary $n$, not just $n=2$, as \cite[Lemma 7.11]{debarre2001higher} applies 
more generally to proper birational morphisms to a normal variety. Since $X^{[2]}$ is the quotient of $Bl_{\Delta}X^2$ by the 
cyclic group $\S_2$, by \cite[Corollary 3.11]{esnault1992lectures} the pushforward $q_*\O_{Bl_{\Delta}X^2}$ 
decomposes as a direct sum of line bundles 
\[
    q_*\O_{Bl_{\Delta}}X^2 \cong \O_{X^{[2]}}\oplus L
\]
where $\O_{A^{[2]}}, L$ are respectively the $1, -1$ eigensheaves of the action of $\S_2$ on $\O_{Bl_{\Delta}X^2}$. By push-pull, 
we conclude that:
\begin{formula}\label{S2_invar_sections_no_pullback}
If $\mathcal{F}$ is a sheaf on $X^{[2]}$, then 
\[
    H^0(X^{[2]}, \mathcal{F}) \cong H^0(Bl_{\Delta}X^2, q^*\mathcal{F})^{\S_2}.
\]
\end{formula}
\noindent
If our sheaf is in fact of the form $\pi^*\mathcal{G}$ 
for a sheaf $\mathcal{G}$ now on $X^{(2)}$, then by commutativity we have the equality 
\[
    H^0(Bl_{\Delta}X^2, q^*\pi^*\mathcal{G}) \cong H^0(Bl_{\Delta}X^2, \hat{\pi}^*p^*\mathcal{G}). 
\]
Since the fibers of $\hat{\pi}$ are connected, we have an isomorphism 
\[
    H^0(Bl_{\Delta}X^2, \hat{\pi}^*p^*\mathcal{G}) \cong H^0(X^2, p^*\mathcal{G})
\]
induced by $\hat{\pi}_*$. The respective involutions on $Bl_{\Delta}X^2, X^2$ commute with $\hat{\pi}$, so that taking 
$\S_2$-invariant sections respects this last isomorphism. The upshot then is that:
\begin{formula}\label{S2_invar_sections}
If $m\ge 0$, then 
\[
    H^0(X^{[2]}, \pi^*\mathcal{G}\otimes \O_{X^{[2]}}(mB)) \cong H^0(X^2, p^*\mathcal{G})^{\mathfrak{S}_2}.
\]
\end{formula}
\noindent
Thus, calculating dimensions of global sections of certain sheaves on $X^{[2]}$ is equivalent to calculating 
invariant global sections on $X^2$. We mention this interpretation as we will eventually calculate the dimension of global sections 
for many line bundles on $X^{[2]}$, while calculating the dimension of invariant sections on $X^2$ can be difficult to do directly.

\subsection{Abelian surfaces}\label{abelian_surface_prelim}\hfill\\
Let $A$ be a complex abelian surface. In the case $X=A$, $\Alb(X) = X$, so that the work of Fogarty in the previous section immediately yields:
\begin{proposition}
For $A$ a complex abelian surface, 
\[
\Pic(A^{[n]}) \cong \Pic(A)\times \NS(A)\times \Z[B],
\] 
where $\NS(A)$ is the Neron-Severi group of $A$. In particular, if $A$ has Picard rank 1, then 
\[
\Pic(A^{[n]}) \cong \Pic(A) \times \Z \times \Z[B].
\]
\end{proposition}
When $A$ is principally polarized, it is a decent exercise in the theory of abelian varieties to prove this directly.
Given a morphism in $\Hom(\Alb(X), \Pic^0(X)) \cong \Hom(A, A^{\vee})$, under the identification 
$A\cong A^{\vee}$ given by the principal polarization we may thus identify $\Hom(\Alb(X), \Pic^0(X))$ with $\Hom(A, A) = \End(A)$. 
Under this identification, the action of swapping the factors on the $\Hom$ side turns out to coincide with the 
Rosati involution on $\End(A)$, and it is well-known that the Rosati-fixed part of $\End(A)$ may be identified with $\NS(A)$. 
However, this proof requires a principal polarization to identify $A$ with $A^{\vee}$, which is insufficiently general for our 
purposes. 

There is a summation morphism $\Sigma: A^{(n)} \rightarrow A$ defined via 
\[
\Sigma(a_1x_1 + ... + a_kx_k) = a_1x_1 + ... + a_kx_k,
\]
where the latter sum uses the group operation on $A$.
Precomposing this map with $\pi: A^{[n]}\rightarrow A^{(n)}$ gives a summation map from $A^{[n]}$ to $A$ which we will also 
denote with $\Sigma$ and make clear the source from context. The fibers of $\Sigma:A^{[n]}\rightarrow A$ are hyperkähler varieties 
of dimension $2(n-1)$ called generalized Kummer varieties, first studied in \cite[Section 7]{Beauville1983VaritsKD}. 
For $n=2$ these fibers are the usual Kummer K3 surfaces. 
\newline

\noindent
Let $\mathcal{P}_A$ denote the Poincaré bundle on $A\times A^{\vee}$. By \cite[p.78]{mumford2008abelian}, we have 
that 
\begin{formula}\label{sum_pullback}
\[
\Sigma^*D = (1\otimes\phi_D)^*\mathcal{P}_A + \pi_1^*D + \pi_2^*D
\]
\end{formula}
\noindent
in $\Pic(A^2)$ for $D$ a divisor on $A$, where $\phi: A\rightarrow A^{\vee}$ is the map $x \mapsto [t_x^*D - D]$. Thus, $\Sigma^*D = (D, \phi_D, 0)$ in $\Pic(A^{[2]}) \cong \Pic(A)\times \Hom(A,A^{\vee})^{\S_2} \times \Z[B]$. With the identification $\Hom(A,A^{\vee})^{\S_2}\cong \NS(A)$ described previously, we may then identify 
\[
    \Sigma^*D = (D,D,0)\in \Pic(A^{[2]})\cong \Pic(A)\times \NS(A) \times \Z[B].
\]
We will use this formula later to perform a convenient change of basis for $\NS(A)$, by replacing the generator $\phi_{\Theta}$ 
with $\Sigma^*\Theta$ in $\Pic(A^{[2]}), \NS(A^{[2]})$ for $\Theta$ a polarization. 

Observe that this formula implies that $\Sigma^*:\Pic^0(A) \rightarrow \Pic^0(A^{[2]})$ is an injective homomorphism of complex 
abelian varieties. These two abelian varieties have the same dimension as noted in Subsection \ref{hilb_prelim},
and hence $\Sigma^*:\Pic^0(A)\rightarrow \Pic^0(A^{[2]})$ is an isomorphism 
since we are working with complex tori over $\C$. This argument holds regardless of the Picard number of $A$. 

\section{Counterexamples to the question of Belmans, Oberdieck, Rennemo}\label{counterexample_section}
\noindent
Here we give a description of some counterexamples to Question \ref{BOR_question} using special abelian varieties of Picard 
rank at least 2. Our constructions give automorphisms on $A^{[n]}$ for different values of $n$ and $\dim A$, where $A$ is an abelian 
variety.

The key result permitting these constructions is the following: 
\begin{proposition}\label{induced_aut_lemma}
If $X$ is either an affine scheme or a projective scheme over an infinite field $k$,
then any $\S_n$-equivariant automorphism $f:X^n\rightarrow X^n$ induces an automorphism on the smoothable locus of $X^{[n]}$.
\end{proposition}

We defer the proof of this to the appendix. In short, one can describe the smoothable locus of $X^{[n]}$ as a blowup of the 
symmetric product $X^{(n)}$ at an ideal sheaf supported along the big diagonal by work of Ekedahl and Skjelnes \cite{ekedahl2014recovering} or a similar description by Rydh and 
Skjelnes \cite{rydh2010intrinsic}. An $\S_n$-equivariant automorphism of $X^n$ descends 
to an automorphism of $X^{(n)}$ by the universal property of quotients by finite groups, and will necessarily restrict to the big 
diagonal pointwise. We prove a stronger version of this last statement in 
Proposition \ref{multiplicity_lemma}. One can show moreover that this automorphism will fix the ideal sheaf in question which we blow up to yield 
the smoothable locus of $X^{[n]}$. Thus, the automorphism of $X^{(n)}$ will lift to an automorphism of the blowup by 
\cite[Corollary II.7.15]{hartshorne1977algebraic}. As discussed in the appendix, we believe the hypotheses of this proposition can 
probably be weakened by using the full generality of the constructions of Ekedahl-Rydh-Skjelnes, but these conditions are already 
significantly more general than what we need in this paper. 

Given a $\S_n$-equivariant automorphism $f:X^n\rightarrow X^n$, write $\bar{f},\tilde{f}$ for the associated automorphisms 
on $X^{(n)}$ and the smoothable locus $X_{sm}^{[n]}$. The map $f\mapsto \tilde{f}$ 
induces a group homomorphism $\Aut(X^n)^{\S_n}\rightarrow \Aut(X_{sm}^{[n]})$ as one can easily check. We always have the commutative 
diagram 
\begin{center}
\begin{tikzcd}
\Aut(X) \arrow[d, hook] \arrow[r, "\text{id}"] & \Aut(X) \arrow[d, hook] \\
\Aut(X^n)^{\S_n} \arrow[r]                     & {\Aut(X_{sm}^{[n]})}        
\end{tikzcd}
\end{center}
where $\Aut(X)\hookrightarrow \Aut(X^n)^{\S_n}$ embeds via the diagonal action (to check that the two maps in question on $X^{[n]}$ 
arising from an element in $\Aut(X)$ agree in the diagram, simply check on the dense open locus of reduced subschemes). 

The map $\Aut(X^n)^{\S_n}\rightarrow \Aut(X_{sm}^{[n]})$ is injective for $n\ge 3$. To see this, suppose $f\in\Aut(X^n)^{\S_n}$ is such 
that $\tilde{f}$ is the identity on $X_{sm}^{[n]}$. It must be then that $\bar{f}$ is the identity on $X^{(n)}$, so that 
the image of a point $(x_1,...,x_n)\in X^n$ under $f$ must be $\sigma(x_1,...,x_n)$ for some $\sigma\in\S_n$. This $\sigma$ must 
be the same for all points in a given irreducible component of $X^n$ by continuity. To elaborate, suppose we fix an irreducible component $Z\subset X^n$. For each $\sigma\in \S_n$ the locus $Z_{f,\sigma}$ of points in $Z$ where $f$ agrees with $\sigma$ is closed, since $X$ is separated whether it is affine or projective. However, the $Z_{f,\sigma}$ cover $Z$ but if all the $Z_{f,\sigma}$ are proper then a finite union of proper closed subsets 
cannot cover $Z$ since $Z$ is irreducible. Thus, $f|_Z = \sigma$ for some $\sigma\in\S_n$.
However, $f|_Z$ is $\S_n$-equivariant, implying that $\sigma\tau = \tau\sigma$ for all $\tau\in\S_n$. As the center of $\S_n$ is 
trivial for $n\ge 3$, $f|_Z$ is the identity. Since this holds for all irreducible components $Z$ of $X^n$, we conclude that 
$f$ is globally trivial. For $n=2$ the $\S_2$-equivariant automorphism $f(x_1, x_2) = (x_2, x_1)$ also induces 
the identity on $X^{[2]}$, so that the fibers of $\Aut(X^2)^{\S_2}\rightarrow \Aut(X^{[2]})$ have size 2 (since $n=2$, specifying the smoothable locus is superfluous). 

For a point $\mathcal{Z} = \sum_{i=1}^k a_ix_i$ in $X^{(n)}$ we say that $\mathcal{Z}$ has multiplicity $\lambda$ where $\lambda$ 
is the partition $(a_1,\ldots,a_k)$ with $a_1\ge\ldots\ge a_k$. We have the same notion for points on $X^n$ and $X^{[n]}$ by looking 
at the multiplicity of the image under the maps $X^n\rightarrow X^{(n)}, X^{[n]}\rightarrow X^{(n)}$. 
We now prove the following important fact about our induced automorphisms: 
\begin{proposition}\label{multiplicity_lemma}
    Suppose $f:X^n\rightarrow X^n$ is an $\S_n$-equivariant automorphism with $X$ as in Proposition \ref{induced_aut_lemma}.  
    The automorphisms $f, \bar{f}, \tilde{f}$ on $X^n, X^{(n)}, X_{sm}^{[n]}$ respectively all preserve multiplicities. 
\end{proposition}
\begin{proof}
    By this we mean that if $\mathcal{Z}$ is an element of either $X^n, X^{(n)}, X_{sm}^{[n]}$ with multiplicity $\lambda$ then 
    $f(\mathcal{Z})$ also has multiplicity $\lambda$. 
    It is immediate by construction that the map $f$ on $X^n$ preserves multiplicities if and only if $\bar{f}$ does on $X^{(n)}$, 
    and since the Hilbert-Chow morphism $\pi:X^{[n]}\rightarrow X^{(n)}$ preserves multiplicities the same equivalence holds between 
    $\bar{f}$ and $\tilde{f}$. Thus, it suffices to prove that any $\S_n$-equivariant automorphism $f:X^n\rightarrow X^n$ preserves 
    multiplicities.
    
    Since the $\S_n$-equivariant automorphisms of $X^n$ are in correspondence with the automorphisms of $X^{(n)}$, this is just a 
    general statement about automorphisms of $X^{(n)}$. For a partition $\lambda = (\lambda_1,...,\lambda_k), \lambda_1\ge\ldots\ge \lambda_k$ of $n$, we write $X^{(n)}_{\lambda}$ for the locus of points in $X^{(n)}$ of multiplicity type $\lambda$. 
    We may also write $X^n_{\lambda}$ for the preimage of $X^{(n)}_{\lambda}$ under the projection. Given 
    two partitions $\lambda=(\lambda_1,\ldots,\lambda_k), \tau=(\tau_1,\ldots,\tau_l)$ we say that $\lambda$ \textit{refines} $\tau$ 
    if $k\ge l$ and there exists a partition $\{1,...,k\} = I_1\cup\ldots\cup I_m$ of the index set of $\lambda$ such that for each 
    $j$, $\sum_{s\in I_j} \lambda_s = \tau_j$. For example, $(2,2,1,1)$ and $(3,1,1,1)$ both refine $(3,3)$ since we may group up 
    elements of these partitions and add them up so as to get $(3,3)$. By contrast, $(4,2)$ and $(3,3)$ do not refine each other. 
    This notion endows the set of partitions of $n$ 
    with the structure of a poset, so that we will write $\lambda\ge\tau$ if $\lambda$ refines $\tau$. 
    
    We may now prove the proposition by inducting up the poset of partitions of $n$ (strictly speaking, we are inducting with 
    respect to a topological sort of this finite poset). To start the induction, we need to prove that 
    $f$ fixes $X_{(n)}^n$ since any partition refines $(n)$ (and this is the only element of the poset for which this is true, so 
    there is a unique base case). We write out $f$ into coordinate 
    functions $f_1,...,f_n:X^n\rightarrow X$, so that at a point of multiplicity $(n)$ we may write $f(x,\ldots, x) = (f_1(x,\ldots,x),
    \ldots, f_n(x,\ldots,x))$. Any permutation of the input $(x,\ldots,x)$ to $f$ does nothing to the left hand side, but will 
    permute the factors $f_i(x,\ldots,x)$ on the right hand side. Since this applies to any permutation, we conclude that 
    $f_i(x,\ldots,x)=f_j(x,\ldots,x)$ for any $i,j$, so that $f(x,\ldots,x)$ is also a point of multiplicity $(n)$. 
    
    Suppose now that $f$ preserves $X_{\tau}^n$ for any $\tau\lneq\lambda$, with $\lambda = (\lambda_1,...,\lambda_k)$. We may 
    without loss of generality consider a point $\vec{x} = (x_1,\ldots,x_1,x_2,\ldots,x_2,\ldots,x_k,\ldots,x_k)\in X^n$ with     multiplicity $\lambda$ (i.e. we assume the coordinates are grouped up). We have $f(\vec{x}) = (f_1(\vec{x}),\ldots,f_n(\vec{x}))$ 
    as before. Using the $\S_n$-equivariance of $f$, we see that the input $\vec{x}$ is unaffected if we permute the first $\lambda_1$ 
    entries, so that the first $\lambda_1$ of the $f_i(\vec{x})$ on the right hand side are all equal. The same is true 
    for the next $\lambda_2$ entries, and so on. To show that the right hand side $(f_1(\vec{x}),\ldots,f_n(\vec{x}))$ has 
    the desired multiplicity $\lambda$, we need to show that the coordinates in different groups corresponding to different $\lambda_i$ 
    are distinct. If not, then necessarily two groups corresponding to different $\lambda_i$ are all equal, so that the right hand 
    side is a point with multiplicity $\tau$ refined by $\lambda$. 
    However, we know that $f$ restricts to an automorphism of $X_{\tau}^n$ by induction, so by invertibility it cannot send a 
    point of $X_{\lambda}^n$ to a point of $X_{\tau}^n$. Thus, we are done. 
    
    As an aside, we require that $f:X^n\rightarrow X^n$ be an automorphism as there exist $\S_n$-equivariant morphisms on $X^n$ 
    which do not preserve multiplicities. As an example, fix a point $p\in X$ - the constant map $f(x_1,...,x_n) = (p,...,p)$ 
    is equivariant but sends any point to a point of multiplicity $(n)$. 
    
    We remark additionally that if $\lambda\ge\tau$ then the closure $\overline{X_{\lambda}^n}$ contains $X_{\tau}^n$. This is quite 
    intuitive. As an example, in the case where $(3,1,1,1)$ refines $(3,3)$, if we have three distinct points of multiplicity 1 
    in $X$ then we can let them run into each other, so that any point of multiplicity type $(3,3)$ is the limit of points of 
    type $(3,1,1,1)$. One can be more formal as an exercise in point-set topology, which we now sketch. 
    On $X^n$, by taking complements we wish to prove that 
    any open set $U$ contained in $(X^n_{\lambda})^c$ is in fact contained in $(X^n_{\tau})^c$. Any such open set $U$ may be written 
    as a union of product open sets, so it suffices to show that if a product open set $U_1\times\ldots\times U_n\subset X^n$ contains 
    a point of multiplicity $\tau$ then it contains a point of multiplicity $\lambda$. Since $X$ is projective over an infinite 
    field these open sets $U_i$ have infinitely many points, so given our point of multiplicity $\tau\in U_1\times\ldots\times U_n$ 
    we can perturb the coordinates to get a point of multiplicity $\lambda$. 
\end{proof}
Note that preserving the multiplicity structure of points in $X^{[n]}$ is a stronger statement than simply fixing the exceptional divisor 
of non-reduced points in $X^{[n]}$.
Thus, to produce a counterexample to the question of \cite{BOR} it suffices to find an 
$\S_n$-equivariant automorphism $f:X^n \rightarrow X^n$ which does not arise from the diagonal embedding $\Aut(X)\hookrightarrow\Aut(X^n)^{\S_n}$. 
For $X = A$ an abelian variety, a group endomorphism $f:A^n \rightarrow A^n$ may be decomposed as a matrix of group endomorphisms of $A$: 
\[
    f = \begin{bmatrix}
        f_{11} & ... & f_{1n} \\
        &...& \\
        f_{n1} & ... & f_{nn}
        \end{bmatrix}
\]
where $f(x_1,...,x_n) = (\sum_{j=1}^n f_{1j}(x_j),...,\sum_{j=1}^nf_{nj}(x_j))$. One can check that this identification transforms composition of morphisms into matrix multiplication using the ring structure of $\End(A)$. For such a matrix 
to be $\S_n$-equivariant it must have all diagonal entries equal and all off-diagonal entries equal, and certainly such a matrix 
is indeed $\S_n$-equivariant. Thus, to produce $\S_n$-equivariant automorphisms of $A^n$ we need only find entries $x,y\in\End(A)$ such 
that the determinant of the matrix 
\[
\begin{bmatrix}
x & y & y & ... & y\\
y & x & y & ... & y\\
&&&...& \\
y & y & y & ... & x
\end{bmatrix}
\]
is a unit in the ring $\End(A)$, as then the matrix (and hence the endomorphism) has an inverse given by the adjoint matrix.
If $y\neq 0$, then the automorphism of $A^n$ does not arise from the diagonal embedding of $\Aut(A)$ (these are automorphisms as a variety) and so the corresponding 
automorphism of $A^{[n]}$ will give a desired counterexample. 
\newline\newline 
We now give some examples of such matrix solutions. 
\subsection{$\dim A = 2, n = 2$ via Pell's equation}\hfill\\
These examples were first published by Sasaki \cite{sasaki2023nonnatural} and independently discovered by the author. 
Let $d\ge 2$ be a positive integer which is not a perfect square. The Pell's equation 
\[
x^2 - dy^2 = 1
\]
has infinitely many integer solutions $(x,y)$. One can construct 
complex abelian surfaces whose endomorphisms are precisely the ring of integers $\O_K$ of a totally real number field - see \cite[Chapter 2.2]{goren2002lectures} and the discussion after equation (2.46).
Let $A$ be a complex abelian surface with $\End(A) = \O_K$ where $K = \Q(\sqrt{d})$, so that in particular $A$ has an endomorphism 
$\sqrt{d}$ such that the self-composition $\sqrt{d}\circ\sqrt{d}$ coincides with the endomorphism $d\in\Z\subset\End(A)$.

Fix an integer solution $(x,y)$ to the Pell's equation $x^2 - dy^2 = 1$ with $x,y\neq 0$. The matrix 
\[
    M = 
    \begin{bmatrix}
    x & y\sqrt{d} \\
    y\sqrt{d} & x
    \end{bmatrix}
    \in \End(A^2)
\]
has determinant $x^2 - dy^2 = 1$ as an element in $\End(A)$, and so defines an automorphism of $A^2$. Since the entries on the 
diagonal and off-diagonal are separately equal to some fixed value, the automorphism is $\S_2$-equivariant, and since the off-diagonal 
entry is nonzero it is not in the image of $\Aut(A)\hookrightarrow \Aut(A^2)^{\S_2}$. By the preceding discussion, we obtain an 
automorphism of $A^{[2]}$ which is not natural, but which nevertheless preserves the diagonal. 

The Picard rank of $A$ is 2, since $\O_K$ is a $\Z$-module of rank 2 so that the Picard rank is at most 2, 
but if it were 1 then we could write every element of $\NS(A)$ as an integer multiple of some generator $\Theta$, so that 
$(\sqrt{d})^*\Theta = k\Theta$ for some integer $k$ and hence $d\Theta = k^2\Theta$, contradicting that $d$ is not a perfect square.
Alternatively, simply invoke \cite[Lemma 1.2]{yoshioka2023}.

\subsection{$\dim A \ge 2, n \ge 2$ via nilpotent endomorphisms}\label{nilpotent_end_counterex}\hfill\\
Suppose $B = A^m$ for $A$ an abelian variety. Any endomorphism $N$ of $B$ whose $m\times m$ matrix is upper triangular with all zeros on the diagonal 
is nilpotent by linear algebra, so that $N^k = 0$ for some $k$. Suppose $N\neq 0$ so that $2\le k\le m$. The determinant 
of the $n\times n$ matrix 
\[
\begin{bmatrix}
    I & N & N & ... & N \\
    N & I & N & ... & N \\
    N & N & I & ... & N \\
    &&&...& \\
    N & N & N & ... & I    
\end{bmatrix}
\]
is $I + N^2T \in \End(B) = \End(A^m)$ where $T$ is some polynomial in $N, I$. To see this, expand the determinant along the first row. The first 
term of this expansion will be $I$ times the determinant of the $(n-1)\times (n-1)$ version of this matrix, and the subsequent terms 
will be $\pm N$ times the determinant of a matrix whose first column is all $N$'s. We thus inductively get either $I$ or terms 
divisible by $N^2$, as desired. 

The term $N^2T$ is upper triangular 
with zeros on the diagonal, so that $I + N^2T$ is an $m\times m$ upper triangular matrix of the form 
\[
I + N^2T = 
\begin{bmatrix}
1 & * & \ldots &* \\
0 & 1 & \ldots &* \\
&&\ldots& \vdots \\
0 & 0 & \ldots &1
\end{bmatrix}
\]
with $1$'s on the main diagonal. This matrix has determinant $1\in\End(A)$ by expanding along the first column successively, 
so that $I+N^2T$ is an invertible element of $\End(B)$ and hence our original $\S_n$-equivariant endomorphism on $B^n$ defined 
by the first matrix is 
an automorphism, and thus induces an unnatural automorphism preserving the exceptional divisor on the smoothable locus of $B^{[n]}$. For 
$B = E\times E$ a product abelian surface where $E$ is an elliptic curve we obtain counterexamples to the question of Belmans-Oberdieck-Rennemo for all $n$. Note that these product abelian varieties have Picard rank greater than 1 coming from pullbacks from the 
individual factors and from maps between the factors. 
\subsection{$\dim A \ge 3, n = 3$ and beyond using higher degree number fields}\label{higher_deg_counterex}\hfill\\
By now the pattern is clear for how to construct counterexamples using abelian varieties - we have to find an abelian variety $A$ 
such that the determinantal equation 
\[
    \det
    \begin{bmatrix}
    x & y & y & ... & y \\
    y & x & y & ... & y \\
    y & y & x & ... & y \\
    &&...&&\\
    y & y & ... & y & x 
    \end{bmatrix}
    \in \End(A)^{\times}
\]
admits solutions over $\End(A)$ with $y\neq 0$. As previously discussed, by \cite[Chapter 2.2]{goren2002lectures} we can try to look for solutions 
in rings of integers of totally real fields. This is not the only possibility, as many other rings can arise as the endomorphisms of abelian varieties.

For one example, take $n=3$, so that the determinant above becomes $x^3 - 3xy^2 + 2y^3$. View $y$ as a fixed coefficient, and set 
the determinant equal to 1, so that we wish to find solutions in $x$ to 
\[
    x^3 - (3y^2)x + (2y^3 - 1) = 0.
\]
The discriminant of this cubic in $x$ is 
\[
    \Delta = - (4(-3y^2)^3 + 27(2y^3 - 1)^2) = 108y^3 - 27.
\]
If $y$ is a positive integer then the discriminant is positive, so that the equation has three distinct real roots. Picking one 
of these solutions $x = \alpha$ gives a degree 3 totally real field $K = \Q(\alpha)$, and hence we may find an abelian variety $A$ admitting multiplication by $\O_K$ (one can take $A$ to be a threefold). The resulting endomorphism on $A^3$ given by the matrix 
\[
    M = 
    \begin{bmatrix}
    \alpha & y & y \\
    y & \alpha & y \\
    y & y & \alpha
    \end{bmatrix}
\]
is therefore invertible, and so we obtain an induced unnatural automorphism on $A^{[3]}$ fixing the exceptional divisor. For a smooth variety $X$ 
it is known that $X^{[3]}$ is also still smooth, and therefore equal to its smoothable component. Whether more solutions 
of these determinantal equations exist for other values of $n$ and with other endomorphism rings of abelian varieties is an 
interesting question to which we do not know the answer.

We can however rule out nontrivial solutions over the integers: 
\begin{proposition}\label{no_mat_int_sol}
Let $M_n$ be the $n\times n$ matrix 
\[
M_n = 
\begin{bmatrix}
    x & y & y & \ldots & y \\
    y & x & y & \ldots & y \\
    &&\ldots&&\\
    y & y & \ldots & y & x 
    \end{bmatrix}.
\]
There are no pairs $(x,y)$ of integers with $y\neq 0$ such that $\det(M_n) = \pm 1$ for $n\ge 3$. 
\end{proposition}
\begin{proof}
The key is to show how to factor this determinant for all $n$. To do this, set $T_n$ to be the matrix 
\[
T_n = 
\begin{bmatrix}
y & y & y & \ldots & y \\
y & x & y & \ldots & y \\
y & y & x & \ldots & y \\
&&&\ldots&\\
y & y & y & \ldots & x
\end{bmatrix}
\]
which consists of all $y$'s in the first row and column and a copy of $M_{n-1}$ in the bottom right hand minor. 
By expanding determinants along the top row and using row operations to transform the resulting minors into copies of $M_{n-1}, T_{n-1}$, a straightforward induction shows that
\begin{formula}\label{Mn_det_formula}
\begin{align*}
    \det(M_n) &= (x-y)^{n-1}(x + (n-1)y) \\
    \det(T_n) &= y(x-y)^{n-1}.
\end{align*}
\end{formula}

We may now prove the proposition. Suppose that $\det(M_n) = (x-y)^{n-1}(x + (n-1)y) = \pm 1$ for $x,y$ integers. It must be then that either $x-y = 1$ or $x-y = -1$. If $x-y = 1$ then $x + (n-1)y = ny + 1$ must be equal to $\pm 1$, so that $ny = 0, -2$, which can't happen if $n\ge 3$ unless $y=0$. Similarly, if $x-y = -1$ then 
$x + (n-1)y = ny - 1 = \pm 1$, implying $ny = 0, 2$, which can't happen unless $y=0$. We conclude the proposition. 
\end{proof}
This gives some evidence that automorphisms $A^{[n]}$ should be natural for all $n$ if $A$ is an abelian surface of Picard rank 1. However, it is not a proof, as it is not immediate that all automorphisms of $A^{[n]}$ arise from $A^{(n)}$. 
\subsection{Action on generalized Kummer fibers}
Suppose $(a_1,...,a_n)\in A^n$ is a vector such that $a_1 + ... + a_n = 0$. Any matrix $M_n$ defining an automorphism on $A^n$ with 
$x,y\in\End(A)$ as before acts on this point as 
\[
    M_n(a_1,\ldots,a_n) = 
    \begin{bmatrix}
    x & y & y & \ldots & y \\
    y & x & y & \ldots & y \\
    &&&\ldots&\\
    y & y & \ldots & y & x 
    \end{bmatrix}
    \begin{bmatrix}
    a_1 \\
    a_2 \\
    \vdots \\
    a_n
    \end{bmatrix}
    = 
    \begin{bmatrix}
    x(a_1) + y(a_2 + \ldots + a_n) \\
    x(a_2) + y(a_1 + a_3 + \ldots + a_n) \\
    \vdots \\
    x(a_n) + y(a_1 + \ldots + a_{n-1}) 
    \end{bmatrix}
    =
    \begin{bmatrix}
    (x-y)(a_1) \\
    (x-y)(a_2) \\
    \vdots \\
    (x-y)(a_n)
    \end{bmatrix}
\]
since $a_i = -(a_1 + \ldots + a_{i-1} + a_{i+1} + \ldots + a_n)$ by hypothesis. Thus, the 
induced automorphism $\tilde{f}$ of the smoothable locus of $A^{[n]}$ restricts to an automorphism of the fiber over $0$ of the summation morphism 
$\Sigma: A^{[n]}\rightarrow A$ preserving the exceptional divisor. \cite[Theorem 3.1]{boissiere2011higher} implies for an abelian surface 
that such an automorphism of the generalized Kummer variety $K_{n-1}(A)\subset A^{[n]}$ must be the restriction of a natural 
automorphism $(t_a\circ g)^{[n]}$ on $A^{[n]}$, where $g$ is a group automorphism of $A$ and $t_a$ is translation by an $n$-torsion point 
if $n\ge 3$. 

However, there is no contradiction. Though 
our $\tilde{f}$ on $A^{[n]}$ is globally not a natural automorphism, its restriction to $K_{n-1}(A)$ agrees with the restriction 
of the natural endomorphism $(x-y)^{[n]}$ per our calculation. We just need to show that $x-y$ is in fact an automorphism. By formula \ref{Mn_det_formula} we know that 
\[
    \det(M_n) = (x-y)^{n-1}(x + (n-1)y)
\]
is a unit in the ring $\End(A)$. 

It is not true for arbitrary rings 
that if the product of two elements is a unit then those elements must have been units. However, it is true for $\End(A)$, since 
if $g\circ h = \id$ in $\End(A)$ then we may lift this to a product of linear maps on some complex vector space, and those 
linear maps must be invertible by taking determinants. Thus, $x-y$ is an automorphism of $A$, so $(x-y)^{[n]}$ is a natural automorphism whose action on $K_{n-1}(A)$ agrees with that of $\tilde{f}$. 

\section{Numerical calculations}\label{numerical_calc_section}
\noindent
In this section, let $A$ be a complex abelian surface of Picard rank 1, whose Neron-Severi group is generated by a polarization $\Theta$ such that $\Theta^2 = 2k$ for some positive integer $k$. 
To prove Theorem \ref{main_theorem}, we will need to calculate various numerical invariants on $A^{[2]}$. The reader who is only interested in 
the main theorem may immediately begin with the proof in the next section and refer back to these results as they are used. 
\subsection{Intersection numbers on $A^{[2]}$}\label{intersection_number_subsection}\hfill\\
We have that 
\[
\NS(A^{[2]}) \cong \Z^3,
\]
with generators $\Theta_{[2]}, \phi_{\Theta}, B$. Formula \ref{sum_pullback} allows us to change our basis to have generators 
\[
\Theta_{[2]}, \Sigma^*\Theta, B
\]
instead. We will calculate all top intersection numbers of these divisors, and for simplicity will write 
\begin{align*}
    x &= \Theta_{[2]} \\
    y &= \Sigma^*\Theta
\end{align*}
so that $\NS(A^{[2]})$ has free generators $x, y, B$. 

Consider the following commutative diagram:
\begin{center}
\begin{tikzcd}
Bl_{\Delta} A^2 \arrow[d, "q"'] \arrow[r, "\hat{\pi}"] & A^2 \arrow[d, "p"] \\
{A^{[2]}} \arrow[r, "\pi"]                     & A^{(2)}           
\end{tikzcd}
\end{center}
where all maps are the obvious ones.
Note that $q, p$ are generically of degree two and $\hat{\pi}, \pi$ are generically of degree 1 with connected fibers. 
We will lift our calculations of intersection 
numbers in $A^{[2]}$ upstairs to $Bl_{\Delta}A^2$. We will respectively denote the exceptional divisors of $A^{[2]}, Bl_{\Delta}A^2$ 
as $E, F$. The normal bundle to the diagonal in the 
self-product of any variety $X$ is well-known to be isomorphic to the tangent bundle $T_X$ (the diagonal $\Delta \subset X^2$ is 
itself isomorphic to $X$), which for $X = A$ an abelian surface is just a trivial rank two bundle $T_A \cong \O_A^{\oplus 2}$. 
Thus, the exceptional divisor $F = \P(N_{\Delta / A^2})$ is simply isomorphic to $A\times \P^1$, with normal bundle in 
$Bl_{\Delta}A^2$ simply given by the pullback of $\O_{\P^1}(-1)$ along the projection $F = A\times \P^1 \rightarrow \P^1$. 
Since $F = q^{-1}(E)$ set-theoretically and $q$ has degree two, we have that $q^*E = 2F$ in Picard. 

Suppose we wish to calculate an intersection number on $A^{[2]}$ of the form 
\[
\int_{A^{[2]}}E^m\cdot \pi^*\alpha
\]
where $m > 0$ and $\alpha$ is a cycle on $A^{(2)}$ of complementary codimension. Since $q^*E = 2F$ and $q$ is generically of degree 2, 
we may instead calculate this as 
\begin{align*}
\int_{A^{[2]}}E^m\cdot \pi^*\alpha &= \frac{1}{2}\int_{Bl_{\Delta}A^2} q^*(E^m\cdot \pi^*\alpha) \\
                    &= \frac{1}{2}\int_{Bl_{\Delta} A^2} (q^*E)^m \cdot \hat{\pi}^*p^*\alpha \\
                    &= 2^{m-1} \int_{Bl_{\Delta} A^2} F^m \cdot \hat{\pi}^*p^*\alpha.
\end{align*}
Consider then the following blowup diagram:
\begin{center}
\begin{tikzcd}
F \cong A\times\P^1 \arrow[r, "j", hook] \arrow[d, "p_1"'] & Bl_{\Delta}A^2 \arrow[d, "\hat{\pi}"] \\
\Delta \cong A \arrow[r, "i", hook]            & A^2                          
\end{tikzcd}.
\end{center}
We note first that $[F] = j_*1$ in the Chow ring $A^*(Bl_{\Delta}A^2)$ and that $j^*F = -H$, where $H$ is the first Chern class of the line bundle $\O_{\P^1}(1)$ on the $\P^1$ factor of $F \cong A \times \P^1$ by standard intersection theory.

By using these facts and the push-pull formula we find for any $m\ge 1$ that as classes in $A^*(Bl_{\Delta}A^2)$:
\begin{align*}
    F^m &= F^{m-1}\cdot j_*1 \\
    &= j_*(j^*F^{m-1}\cdot 1) \\
    &= (-1)^{m-1}j_*(H^{m-1}).
\end{align*}
Thus, for any $m\ge 1$ and $\beta\in A^*(A^2)$ we find that 
\begin{align*}
    F^m \cdot \hat{\pi}^*\beta &= (-1)^{m-1}j_*(H^{m-1})\cdot \hat{\pi}^*\beta \\
        &= (-1)^{m-1}j_*(H^{m-1}\cdot j^*\hat{\pi}^*\beta) \\
        &= (-1)^{m-1}j_*(H^{m-1}\cdot p_1^*i^*\beta). 
\end{align*}
For this last quantity to give a nonzero codimension 4 cycle, we see that we must have $m=2$ and $\codim\beta = 2$. Moreover, since 
$H$ is just a point in each $\P^1$ fiber of $F\cong A\times\P^1$ and the map $p_1:F\rightarrow \Delta\cong A$ is just projection onto the $A$ factor, 
\[
\int_{Bl_{\Delta}A^2} j_*(H^{m-1}\cdot p_1^*i^*\beta)
\]
simply counts the number of points in the intersection $\int_{\Delta}i^*\beta = \int_{A^2}\beta\cdot\Delta$ in $A^{2}$ when $m=2$.

Returning to our initial calculation of $E^m\cdot \pi^*\alpha$ in $A^{[2]}$, by setting $\beta = p^*\alpha$ we find that 
\begin{align*}
\int_{A^{[2]}}E^m\cdot \pi^*\alpha &= 2^{m-1}\int_{Bl_{\Delta} A^2} F^m\cdot \hat{\pi}^*p^*\alpha \\
                    &= (-2)^{m-1}\int_{Bl_{\Delta} A^2} j_*(H^{m-1}\cdot q^*i^*\beta) \\
                    &= -2\int_{A^2} \beta\cdot\Delta
\end{align*}
if $m=2$ to make this intersection number potentially nonzero. 

If instead $m=0$ so that we are calculating the intersection number on $A^{[2]}$ of divisors purely coming from $A^{(2)}$, then since 
$p:A^2\rightarrow A^{(2)}$ is generically of degree 2 and $\pi:A^{[2]}\rightarrow A^{(2)}$ is generically of degree 1 we find that 
\begin{align*}
\int_{A^{[2]}}\pi^*\beta &= \int_{A^{(2)}}\beta \\
                        &= \frac{1}{2}\int_{A^2}p^*\beta,
\end{align*}
allowing us to calculate any top intersection number of divisors on $A^{[2]}$ on $A^2$ instead.  

With our basis $x, y, B$ for $\NS(A^{[2]})$, when lifting to $A^2$ we have that $x$ lifts to $\pi_1^*\Theta + \pi_2^*\Theta$ where 
the $\pi_i$ are the natural projections. 
By our preceding work calculating intersection numbers on $A^{[2]}$, we know 
that any monomial of the form \[\int_{A^{[2]}}x^a y^b E^c\] can only be nonzero if $c = 0, 2$. We now calculate the intersection 
numbers of all these monomials. 
Recall that $\Theta^2 = 2k$, and write $p$ for the class of a point. 

We first calculate 
\begin{align*}
    \int_{A^{[2]}}x^2y^2 &= \frac{1}{2}\int_{A^2}(\pi_1^*\Theta + \pi_2^*\Theta)^2(\Sigma^*\Theta)^2 \\
                    &= 4k^2\int_{A^2}\pi_1^*(p)\Sigma^*(p) + 2k\int_{A^2}(\pi_1^*\Theta)(\pi_2^*\Theta)\Sigma^*(p)
\end{align*}
where we use the symmetry $\pi_1^*(\Theta^2)\Sigma^*(\Theta^2) = \pi_2^*(\Theta^2)\Sigma^*(\Theta^2)$.
Thinking set-theoretically, $\pi_1^*(p)$ simply fixes the first coordinate of a point $(x,y)$ in $A^2$, and $\Sigma^*(p)$ fixes the 
sum of the point, which together determines the second coordinate of $(x,y)$ so that 
\[
\int_{A^2}\pi_1^*(p)\Sigma^*(p) = 1. 
\]
To calculate the second term above, if we fix a representative for $\Theta$ then $(\pi_1^*\Theta)(\pi_2^*\Theta)\Sigma^*(p)$ consists 
of those points $(x,y)$ where $x,y\in\Theta$ and $x + y = p$ for some fixed point $p$. This is equivalent to counting the number of 
points $x\in A$ such that both $x, p-x$ are in $\Theta$. This is simply the self intersection of $\Theta$ with the pullback of $\Theta$ 
under a translation and $(-1)^*$, neither of which change the numerical equivalence class of $\Theta$, giving 
\[
\int_{A^2} (\pi_1^*\Theta)(\pi_2^*\Theta)\Sigma^*(p) = \int_A \Theta^2 = 2k.
\]
Hence, 
\begin{align*}
\int_{A^{[2]}}x^2y^2 &= 4k^2\int_{A^2}\pi_1^*(p)\Sigma^*(p) + 2k\int_{A^2}(\pi_1^*\Theta)(\pi_2^*\Theta)\Sigma^*(p) \\
                    &= 8k^2.
\end{align*}
Any terms involving $(\pi_i^*\Theta)^m = \pi_i^*(\Theta^m)$ and $y^m = (\Sigma^*\Theta)^m = \Sigma^*(\Theta^m)$ vanish if $m > 2$, and thus 
\begin{align*}
\int_{A^{[2]}} x^3y &= \frac{1}{2}\int_{A^2} (\pi_1^*\Theta + \pi_2^*\Theta)^3\Sigma^*(\Theta) \\ 
                &= 6k\int_{A^2} \pi_1^*(p)(\pi_2^*\Theta)\Sigma^*(\Theta).
\end{align*}
The class $\pi_1^*(p)(\pi_2^*\Theta)\Sigma^*(\Theta)$ corresponds to points $(x,y)$ such that both $y, p+y$ are in $\Theta$ for some fixed point $p = x$, which amounts to the intersection of $\Theta$ with the translation of $\Theta$ by $p$ which doesn't change numerical equivalence, and hence has value 
\[
\int_{A^2}\pi_1^*(p)(\pi_2^*\Theta)\Sigma^*(\Theta) = \int_A \Theta^2 = 2k, 
\]
so that 
\begin{align*}
\int_{A^{[2]}} x^3y &= 6k\int_{A^2} \pi_1^*(p)(\pi_2^*\Theta)\Sigma^*(\Theta) \\
    &= 12k^2. 
\end{align*}
Finally, 
\begin{align*}
\int_{A^{[2]}} x^4 &= \frac{1}{2}\int_{A^2}(\pi_1^*\Theta + \pi_2^*\Theta)^4 \\
                    &= 12k^2 \int_{A^2}\pi_1^*(p)\pi_2^*(p) \\
                    &= 12k^2.
\end{align*}

These are all the classes involving no $E$ term that we need to calculate, since any monomial with $y^m$ for $m > 2$ will be zero. 
Using our formula for intersection numbers of the form $E^m\cdot \pi^*\alpha$, we find that 
\begin{align*}
\int_{A^{[2]}} x^2E^2 &= -2\int_{A^2}(\pi_1^*\Theta + \pi_2^*\Theta)^2\cdot\Delta \\
                    &= -8k\int_{A^2}\pi_1(p)\cdot\Delta - 4\int_{A^2}\pi_1^*(\Theta)\pi_2^*(\Theta)\cdot\Delta.
\end{align*}
The first term $\pi_1^*(p)\cdot \Delta$ counts points $(x,y)$ with $x=p$ fixed and $y=x$ giving just one point, and for 
$\pi_1^*(\Theta)\pi_2^*(\Theta)\cdot\Delta$ we consider the commutative diagram(s)
\begin{center}
\begin{tikzcd}
A\cong\Delta \arrow[rd, "{(x,x)\mapsto x}"'] \arrow[r, "i", hook] & A\times A \arrow[d, "{\pi_1, \pi_2}"] \\
                                                                  & A                                    
\end{tikzcd}
\end{center}
so that 
\begin{align*}
    \int_{A^2}\pi_1^*(\Theta)\pi_2^*(\Theta)\cdot\Delta &= \int_A \Theta^2 \\
                                                        &= 2k
\end{align*}
by push-pull with $i$. We thus conclude that 
\begin{align*}
\int_{A^{[2]}}x^2E^2 &=  -8k\int_{A^2}\pi_1(p)\cdot\Delta - 4\int_{A^2}\pi_1^*(\Theta)\pi_2^*(\Theta)\cdot\Delta \\
                &= -16k.
\end{align*}
Similarly, 
\begin{align*}
\int_{A^{[2]}}y^2E^2 &= -2\int_{A^2}(\Sigma^*\Theta)^2\cdot\Delta \\
                    &= -4k\int_{A^2}\Sigma^*(p)\cdot\Delta \\
                    &= -64k.
\end{align*}
We use that if we take $\Sigma^*(p)$ to be represented by the fiber over 0 then the integral
\[
\int_{A^2}\Sigma^*(p)\cdot\Delta
\]
counts the points $(x,y)$ with $y = -x$ and 
$x = y$, corresponding to the 16 2-torsion points of $A$. 

Finally, we calculate the last relevant term
\begin{align*}
\int_{A^{[2]}}xyE^2 &= -2\int_{A^2} (\pi_1^*\Theta + \pi_2^*\Theta)(\Sigma^*\Theta)\cdot\Delta \\
                &= -4\int_{A^2} \pi_1^*(\Theta)\Sigma^*(\Theta)\cdot \Delta.
\end{align*}
This integral counts points $(x,y)$ with $x\in\Theta, x+y\in\Theta,$ and $x=y$, i.e. points $x\in A$ such that $x, 2x\in\Theta$. 
This amounts to calculating the intersection $\Theta\cdot (\cdot 2)^*\Theta$ where $(\cdot 2)^*$ is the pullback under the 
multiplication by 2 map on $A$. We can pick $\Theta$ to be symmetric (we only care about numerical equivalence), 
so that from the formula 
\[
    (\cdot n)^* L \cong L^{\otimes \frac{n(n+1)}{2}}\otimes (-1)^*L^{\otimes \frac{n(n-1)}{2}}
\]
\cite[p.59 Corollary 3]{mumford2008abelian} we get $(\cdot 2)^*\Theta = 4\Theta$ (using symmetry so that $\Theta = (-1)^*\Theta$). 
Thus, we conclude that 
\begin{align*}
\int_{A^{[2]}}xyE^2 &= -4\int_{A^2} \pi_1^*(\Theta)\Sigma^*(\Theta)\cdot \Delta \\
                    &= -4\int_{A} \Theta\cdot (4\Theta) \\
                    &= -32k.
\end{align*}
We summarize this discussion with a table, recalling that $\Theta^2 = 2k$ for the chosen polarization $\Theta$ on $A$ and that $B = \frac{E}{2}$: 
\begin{center}
\begin{tabular}{||c | c||}
\hline
Cycle & $\int_{A^{[2]}}$ \\
\hline
$x^4$ & $12k^2$ \\
\hline
$x^3y$ & $12k^2$ \\
\hline
$x^2y^2$ & $8k^2$ \\
\hline
$x^2B^2$ & $-4k$ \\
\hline
$xyB^2$ & $-8k$ \\
\hline
$y^2B^2$ & $-16k$ \\
\hline
\end{tabular}
\end{center}

\subsection{Wirtinger pullbacks}\hfill\\
Consider the Wirtinger map $\xi: A\times A\rightarrow A\times A$ defined via $\xi(x,y) = (x + y, x - y)$ for $A$ any abelian variety 
with symmetric line bundle $\Theta$. By the see-saw theorem, we have that 
\[
\xi^*(\pi_1^*\Theta\otimes\pi_2^*\Theta) = \pi_1^*\Theta^{\otimes 2} \otimes \pi_2^*\Theta^{\otimes 2}
\]
(see \cite[Section 3 Proposition 1]{mumford_on_equations} for a proof and \cite[p. 336]{mumford_prym}). Direct calculation shows that the diagram
\begin{center}
\begin{tikzcd}
A\times A \arrow[r, "\xi"] \arrow[d, "\pi_1"'] & A\times A \arrow[d, "\Sigma"] \\
A \arrow[r, "\cdot 2"]                         & A                            
\end{tikzcd}
\end{center}
commutes where $\cdot 2$ is the endomorphism $x \mapsto 2\cdot x$ on $A$, so that 
\begin{equation}\label{wirtinger_pullback}
\xi^*\Sigma^*\Theta = \pi_1^*(\cdot 2)^*\Theta.
\end{equation}
Since $\Theta$ is symmetric we conclude that 
\[
\xi^*\Sigma^*\Theta = \pi_1^*\Theta^{\otimes 4},
\]
and hence that 
\begin{equation}\label{wirtinger_pullback_theta}
\xi^*((\pi_1^*\Theta\otimes\pi_2^*\Theta)^{\otimes a}\otimes(\Sigma^*\Theta)^{\otimes b}) \cong 
\pi_1^*\Theta^{\otimes(2a+4b)}\otimes \pi_2^*\Theta^{\otimes 2a}
\end{equation}
on $A^2$. We will use this later to rule out certain coefficients for $g^*E$ in Neron-Severi by considering dimensions of 
global sections. 

\subsection{Dimensions of global sections on the symmetric product}\label{global_sections_calcs}\hfill\\
Intersection numbers alone will not suffice to show what we want about the action of automorphisms of $A^{[2]}$ on $NS(A^{[2]})$. 
If $c \ge 0$, then any line bundle $L$ on $A^{[2]}$ in the Neron-Severi equivalence class $ax + by + cB$ can be written as 
$\pi^*L'\otimes \O_{A^{[2]}}(cB)$ for some line bundle $L$ on $A^{(2)}$ in the equivalence class $ax + by\in NS(A^{(2)})$ where 
$\pi:A^{[2]}\rightarrow A^{(2)}$ is the Hilbert-Chow morphism. We recall by formula \ref{pushforward_formula} that it suffices to calculate the dimensions of global sections of $L'$ on $A^{(2)}$ in this case. 

We give the following nearly-complete formulas when $\Theta$ is a symmetric principal polarization: 
\begin{formula}\label{global_sections_formula}
\[
\dim H^0(A^{(2)}, \Theta_{(2)}^{\otimes k} \otimes \Sigma^*\Theta^{\otimes \ell} \otimes L_0) =
    \begin{cases}
        0, & \text{if } k < 0 \text{ or } k + 2\ell < 0\\
        0, & \text{if } k \ge 0 \text{ and } k + 2\ell = 0 \text{ and } L_0^{\otimes 2}\ncong\O_{A^{(2)}}\\
        \ell^2, & \text{if } k = 0 \text{ and } \ell > 0\\
        \frac{(k^2+1)(k+2\ell)^2}{2}, & \text{if } k > 0 \text{ and } k + 2\ell > 0
    \end{cases}
\]
\end{formula}
\noindent
where $L_0$ is a line bundle in $Pic^0(A^{(2)})$. The coefficients $k, \ell$ index the class in the Neron-Severi group $NS(A^{(2)})$, 
and twisting by arbitrary $L_0$ gives a class in $Pic(A^{(2)})$. 
We will partially address and generalize the remaining cases in the next section.

We note by formula \ref{S2_invar_sections} that these formulas give the dimension of the space of $\S_2$-invariant sections 
of the pullbacks of these line bundles to $A^2$, when it can be difficult to explicitly determine the global sections of these 
pullbacks and the corresponding $\S_2$-action. 

Let us consider each of these cases. 

\subsubsection{Cases where $H^0 = 0$}\hfill\\
Since $\Sigma^*:\Pic^0(A)\rightarrow\Pic^0(A^{(2)})$ is an isomorphism, we may write our line bundle on $A^{(2)}$ as 
\[
    L = \Theta_{(2)}^{\otimes k}\otimes\Sigma^*(\Theta^{\otimes \ell} \otimes L_0)  
\]
for some $L_0\in\Pic^0(A)$. Since the projection $p:A^2\rightarrow A^{(2)}$ is surjective, the pullback $p^*$ induces an injection 
on global sections of locally free sheaves (if $p^*s$ is the zero section for some $s\in H^0(A^{(2)}, V)$ with $V\rightarrow A^{(2)}$ 
a vector bundle, then $s(p(x)) = 0$ for all $x\in A^2$, but since $p$ is surjective this implies $s(y) = 0$ for all $y\in A^{(2)}$ so 
that $s = 0$). Thus, if we can show that the pullback $p^*L$ on $A^2$ has no global sections, then $H^0(A^{(2)}, L) = 0$, as desired. 
Moreover, since the Wirtinger map $\xi:A^2\rightarrow A^2$ considered in the previous subsection is surjective, we can instead try to 
show that $H^0(A^2, \xi^*p^*L) = 0$. By our results in the previous section, this is the same as 
\begin{align*}
    H^0(A^2, \xi^*p^*L) &\cong H^0(A^2, \xi^*((\pi_1^*\Theta\otimes\pi_2^*\Theta)^{\otimes k}\otimes\Sigma^*(\Theta^{\otimes \ell}\otimes L_0))) \\
    &\cong H^0(A^2, \pi_1^*\Theta^{\otimes(2k+4\ell)}\otimes\pi_2^*\Theta^{2k}\otimes \pi_1^*(\cdot 2)^*L_0).
\end{align*}
From \cite{mumford2008abelian} we have that $(-1)^*L_0\cong L_0^{\otimes -1}$, and hence $(\cdot 2)^*L_0 \cong L_0^{\otimes 2}$ by the 
usual formula for $(\cdot n)^*L$, so that this becomes 
\begin{align*}
    H^0(A^2, \xi^*p^*L) &\cong H^0(A^2, \pi_1^*(\Theta^{\otimes(2k+4\ell)}\otimes L_0^{\otimes 2})\otimes \pi_2^*\Theta^{\otimes 2k}) \\
        &\cong H^0(A, \Theta^{\otimes(2k+4\ell)}\otimes L_0^{\otimes 2})\otimes H^0(A, \Theta^{\otimes 2k})
\end{align*}
by the Künneth decomposition for the sheaf cohomology of a box product 
(\cite[\href{https://stacks.math.columbia.edu/tag/0BEC}{Tag 0BEC}]{stacks-project}). This gives us the desired vanishing in all of 
our cases. If $k < 0$, then the second factor $H^0(A, \Theta^{\otimes 2k}) = 0$ automatically. If $k + 2\ell < 0$ then 
$\Theta^{\otimes(2k + 4\ell)}$ has no global sections and Euler characteristic 
\begin{align*}
    \chi(A, \Theta^{\otimes(2k+4\ell)}) &= \frac{(2k+4\ell)^2}{2}\int_A \Theta^2 \\
        &= (2k+4\ell)^2 \\
        &\neq 0,
\end{align*}
so that by \cite[p.150]{mumford2008abelian} the sheaf cohomology of $\Theta^{\otimes(2k+4\ell)}$ is concentrated in a single degree. This 
degree (the \textit{index}) is determined by the first Chern class of each line bundle \cite[p.61]{BL_complex_abvar}, which is an element of the discrete group $H^2(A, \Z)$ and hence is constant within a numerical 
equivalence class in $\Pic(A)$, so that $H^0(A, \Theta^{\otimes (2k+4\ell)}\otimes L_0^{\otimes 2}) = 0$ if $2k+4\ell < 0$. If $2k + 4\ell = 0$, 
then the only line bundle in $Pic^0(A)$ with a global section is $\O_A$, so we again get the desired vanishing if 
$L_0^{\otimes 2}\ncong\O_A$.

\subsubsection{Case where $k = 0, \ell > 0$}\hfill\\
This case is almost immediate. We have that $L = \Sigma^*(\Theta^{\otimes \ell}\otimes L_0)$ for some $L_0\in\Pic^0(A)$. 
We may pushforward under $\Sigma$ to instead calculate 
\begin{align*}
    H^0(A^{(2)}, L) &\cong H^0(A, \Sigma_*L) \\
                    &\cong H^0(A, \Theta^{\otimes \ell}\otimes L_0)
\end{align*}
by push-pull, since $\Sigma$ has connected fibers. Since $\ell > 0$, we have that 
\begin{align*}
    \chi(A, \Theta^{\otimes \ell}\otimes L_0) &= \frac{\ell^2}{2}\int_A \Theta^2 \\
                                        &= \ell^2 \\
                                        &\neq 0
\end{align*}
since $L_0$ is numerically trivial. Since $\ell > 0$, $\Theta^{\otimes \ell}$ is effective and hence has a section, so that by concentration 
of cohomology $\dim H^0(A, \Theta^{\otimes \ell}) = \ell^2$. Since the index of a non-degenerate line bundle is unchanged by twisting 
with $L_0\in\Pic^0(A)$, we are done. 

\subsubsection{Case where $k > 0$ and $k+2\ell > 0$}\label{symm_prod_h0}\hfill\\
This case requires substantially more work than the previous two. 
Define the morphism 
\begin{align*}
    w(x, y)&: A\times A \rightarrow A^{(2)} \\
    w(x, y) &= (x + y, x - y),
\end{align*}
and note that $w$ is invariant if we replace $y$ with $-y$, so that $w$ commutes with the natural $\Z/2\Z$ action 
on the second $A$ factor and hence descends to a morphism 
\begin{align*}
    \mu(x, y)&: A\times (A / \pm 1) \rightarrow A^{(2)} \\
    \mu(x, y) &= (x + y, x - y).
\end{align*}
\begin{proposition}
The following diagram is Cartesian: 
\begin{center}
\begin{tikzcd}
A\times (A/\pm 1) \arrow[r, "\mu"] \arrow[d, "\pi_1"'] & A^{(2)} \arrow[d, "\Sigma"] \\
A \arrow[r, "(\cdot 2)"]                               & A                          
\end{tikzcd}
\end{center}
\end{proposition}
\begin{proof}
A simple calculation shows that the square is commutative. It remains to prove the universal property of pullback squares. Suppose that 
$X$ is a scheme equipped with morphisms $f:X\rightarrow A, g:X\rightarrow A^{(2)}$ such that the diagram 
\begin{center}
\begin{tikzcd}
X \arrow[r, "g"] \arrow[d, "f"'] & A^{(2)} \arrow[d, "\Sigma"] \\
A \arrow[r, "\cdot 2"]           & A                          
\end{tikzcd}
\end{center}
is commutative. We wish to show the existence of a morphism $h:X\rightarrow A\times (A/\pm 1)$ such that the diagram 
\begin{center}
\begin{tikzcd}
X \arrow[rdd, "f"', bend right] \arrow[rrd, "g", bend left] \arrow[rd, "h"] &                                                      &                             \\
                                                                            & A\times (A/\pm 1)\arrow[r, "\mu"] \arrow[d, "\pi_1"'] & A^{(2)} \arrow[d, "\Sigma"] \\
                                                                            & A \arrow[r, "\cdot 2"]                               & A                          
\end{tikzcd}
\end{center}
commutes. We will construct this morphism as $h = h_1\times h_2$ for morphisms $h_1:X\rightarrow A,\, h_2:X\rightarrow A/\pm 1$. 
We unsurprisingly set $h_1 = f$. 
To find $h_2$, consider the anti-diagonal embedding $i: A\hookrightarrow A^2,\, x\mapsto (x, -x)$, which is equivariant with 
respect to $\Z/2\Z$ where the group action is via multiplication by $\pm 1$ on the left and swapping the factors on the right, and hence 
descends to an embedding $\bar{i}: A/\pm 1 \hookrightarrow A^{(2)}$. Given any point $a\in A$ we have a natural translation morphism 
$t_a^{[2]}:A^{(2)}\rightarrow A^{(2)}$ acting pointwise. 

Observe that 
\[
\Sigma\circ t_{-f(x)}^{[2]}(g(x)) = \Sigma(g(x)) - 2f(x) = 2f(x) - 2f(x) = 0
\]
for any $x\in A$, so that $t_{-f(x)}^{[2]}(g(x))$ lies in the image of $\bar{i}$. We may thus define 
\[
h_2(x) := \bar{i}^{-1}(t_{-f(x)}^{[2]}(g(x))) \in A/\pm 1.
\]
The commutative relation $f = \pi_1\circ h$ is immediate. To prove the other relation $g = \mu\circ h$, suppose that we may write 
$g(x) = (a,b) = (b,a)\in A^{(2)}$ for points $a,b\in A$. The translate $t^{[2]}_{-f(x)}(g(x)) = (a-f(x), b-f(x))$ satisfies 
$a-f(x) = -(b-f(x))$ per our calculation, and $h_2(x)$ may be identified with the image of $a-f(x)\in A/\pm 1$. We then calculate 
\begin{align*}
\mu\circ h(x) &= \mu(f(x), h_2(x)) \\
                &= (f(x) + h_2(x), f(x) - h_2(x)) \\
                &= (f(x) + (a-f(x)), f(x) + (b-f(x))) \\
                &= (a, b) \\
                &= g(x).
\end{align*}
To see that this induced morphism $X\rightarrow A\times A/\pm 1$ is unique, simply note that $\mu(x,y) = t_x^{[2]}\bar{i}(y)$.
The universal property of pullbacks thus holds, so that our original diagram was Cartesian, as desired. 
\end{proof}
Set $B_k = \Sigma_*(\Theta_{(2)}^{\otimes k})$. Let $L_0\in\Pic^0(A)$ be arbitrary, so that we wish to calculate the dimension of 
$H^0(A^{(2)}, \Theta_{(2)}^{\otimes k}\otimes \Sigma^*(\Theta^{\otimes \ell}\otimes L_0))$.
By push-pull we wish to determine the dimension of $H^0(A, B_k\otimes (\Theta^{\otimes \ell}\otimes L_0))$. Since the preceding diagram is Cartesian we find that 
\[
    (\cdot 2)^*B_k \cong \pi_{1*}\mu^*\Theta_{(2)}^{\otimes k}.
\]
We wish to show that 
\[
\mu^*\Theta_{(2)}^{\otimes k} \cong \pi_1^*\Theta^{\otimes 2k}\otimes \pi_2^*L_k
\]
for some line bundle $L_k$ on $A/\pm 1$ where $\pi_2: A\times (A/\pm 1)\rightarrow A/\pm 1$ is the projection (since $k > 0$ in this 
section, we hope that the choice of notation $L_k$ will not cause confusion with the fixed element $L_0\in\Pic^0(A)$). Consider the 
twisted line bundle 
\[
M_k = \mu^*\Theta_{(2)}^{\otimes k}\otimes \pi_1^*\Theta^{\otimes -2k}.
\]
It suffices by the seesaw principle \cite[p.54]{mumford2008abelian} to show that $M_k|_{A\times \{y\}} \cong \O_A$ for all $y\in A/\pm 1$. 
This restriction is the same as the pullback $i_{A, y}^*M_k$ under the inclusion map 
$i_{A,y}:A\hookrightarrow A\times (A/\pm 1), x\mapsto (x,y)$. Fix a lift $\hat{y}$ of $y$ from $A / \pm 1$ to $A$: there are 
two such choices. We have a similar inclusion map $i_{A, \hat{y}}:A \hookrightarrow A\times A, x\mapsto (x,\hat{y})$ yielding 
a commutative diagram 
\begin{center}
\begin{tikzcd}
A \arrow[r, "{i_{A,\hat{y}}}", hook] \arrow[d, "\id"] & A\times A \arrow[d, "\pi_1\times q"] \arrow[r, "\xi"] & A\times A \arrow[d, "p"] \\
A \arrow[r, "{i_{A,y}}", hook]                        & A\times (A/\pm 1) \arrow[r, "\mu"]                    & A^{(2)}                 
\end{tikzcd}
\end{center}
where $q, p$ are natural projections and $\xi:A^2\rightarrow A^2$ is the Wirtinger map $\xi(x,y) = (x+y, x-y)$. 
To show that $M_k|_{A\times\{y\}}\cong i_{A,y}^*M_k$ is trivial it suffices to prove that 
\begin{align*}
    \id^*i_{A,y}^*M_k &\cong i_{A,\hat{y}}^*(\pi_1^*\Theta^{\otimes -2k}\otimes (\pi_1\times q)^*\mu^*\Theta_{(2)}^{\otimes k}) \\
        &\cong i_{A,\hat{y}}^*(\pi_1^*\Theta^{\otimes -2k}\otimes \xi^*p^*\Theta_{(2)}^{\otimes k}) \\
        &\cong i_{A,\hat{y}}^*(\pi_1^*\Theta^{\otimes -2k}\otimes \xi^*(\pi_1^*\Theta\otimes\pi_2^*\Theta)^{\otimes k}) \\
        &\cong i_{A,\hat{y}}^*(\pi_1^*\Theta^{\otimes -2k}\otimes (\pi_1^*\Theta\otimes\pi_2^*\Theta)^{\otimes 2k}) \\
        &\cong i_{A,\hat{y}}^*\pi_2^*\Theta^{\otimes 2k} \\
        &\cong \O_A
\end{align*}
is trivial, as we have just done. We conclude that $i_{A,y}^*M_k$ is trivial for all $y\in A/\pm 1$ 
so that 
\begin{align*}
\mu^*\Theta_{(2)}^{\otimes k}\otimes \pi_1^*\Theta^{\otimes -2k} &= M_k \\
    &\cong \pi_2^*L_k
\end{align*}
for some line bundle $L_k$ on $A/\pm 1$ by the seesaw principle. Hence,
\[
\mu^*\Theta_{(2)}^{\otimes k} \cong \pi_1^*\Theta^{\otimes 2k}\otimes \pi_2^*L_k.
\]
We thus derive that 
\begin{align*}
(\cdot 2)^*B_k &\cong \pi_{1*}\mu^*\Theta_{(2)}^{\otimes k} \\
                &\cong \pi_{1*}(\pi_1^*\Theta^{\otimes 2k}\otimes \pi_2^*L_k) \\
                &\cong \Theta^{\otimes 2k}\otimes H^0(A/\pm 1, L_k)
\end{align*}
where by tensoring a sheaf $\mathcal{F}$ with a vector space $V$ we mean the direct sum $\mathcal{F}^{\oplus\dim V}$.
Set $V_k = H^0(A/\pm 1, L_k)$ in what follows. Twisting our previous equality yields  
\[
(\cdot 2)^*(B_k\otimes (\Theta^{\otimes \ell}\otimes L_0)) \cong (\Theta^{\otimes 2k + 4\ell}\otimes L_0^{\otimes 2}) \otimes V_k
\]
since $\Theta$ is symmetric so that $(\cdot 2)^*\Theta \cong \Theta^{\otimes 4}$ and $(\cdot 2)^*L_0 \cong L_0^{\otimes 2}$ since 
$L_0\in\Pic^0(A)$ as discussed in previous cases. 
Set $\mathcal{F}_{k,\ell} = B_k\otimes(\Theta^{\otimes \ell}\otimes L_0)$, the sheaf whose cohomology we wish to determine. 
We will see later (Lemma \ref{sigma_pushforward}) that the higher pushforwards $R^j\Sigma_*\Theta_{(2)}^{\otimes k}$ vanish, so that 
$B_k$ is in fact a vector bundle and hence so is $\mathcal{F}_{k,\ell}.$

To calculate the global sections of this, we may in fact just calculate Euler characteristics: 
\begin{lemma}\label{Fkl_cohomology}
    For all $j > 0$, $H^j(A, \mathcal{F}_{k,\ell}) = 0$. In particular, $\dim H^0(A, \mathcal{F}_{k,\ell}) = \chi(A, \mathcal{F}_{k,\ell})$. 
\end{lemma}
\begin{proof}
    We first claim that the pushforward $(\cdot 2)_*\O_A$ splits as the direct sum of all 2-torsion line bundles: 
    \[
    (\cdot 2)_*\O_A \cong \bigoplus_{L\in \Pic^0(A)[2]} L.
    \]
    This follows by \cite[p.72]{mumford2008abelian} applied to the group $G = A[2]$ of 2-torsion points acting on $A$ by addition. 
    By push-pull we have that 
    \[
    (\cdot 2)_*(\cdot 2)^*\mathcal{F}_{k,\ell} \cong \bigoplus_{L\in Pic^0(A)[2]}L\otimes\mathcal{F}_{k,\ell}. 
    \]

    \noindent
    Since $\cdot 2:A\rightarrow A$ is finite it is affine, so that the pushforward $(\cdot 2)_*$ preserves higher cohomology 
    (combine \cite[\href{https://stacks.math.columbia.edu/tag/01XC}{Lemma 01XC}]{stacks-project} and  
    \cite[\href{https://stacks.math.columbia.edu/tag/01F4}{Lemma 01F4}]{stacks-project}). Thus, 
    \begin{align*}
    H^j(A, (\cdot 2)_*(\cdot 2)^*\mathcal{F}_{k,\ell}) &= H^j(A, (\cdot 2)^*\mathcal{F}_{k,\ell}) \\                                                   &= H^j(A, (\Theta^{\otimes 2k+4\ell}\otimes L_0^{\otimes 2})\otimes V_k) \\
            &= H^j(A, \Theta^{\otimes 2k+4\ell}\otimes L_0^{\otimes 2}) \otimes V_k \\
            &= 0
    \end{align*}
    for $j > 0$ since $2k+4\ell > 0$. Since $\mathcal{F}_{k,\ell}$ embeds into $(\cdot 2)_*(\cdot 2)^*\mathcal{F}_{k,\ell}$ as a direct summand, 
    we conclude that $\mathcal{F}_{k,\ell}$ has vanishing higher cohomology since cohomology respects direct sums.
\end{proof}
We now wish to compute the dimension of $V_k = H^0(A/\pm 1, L_k)$, and will do so by first showing $H^i(A/\pm 1, L_k) = 0$ for $i > 0$ and then calculating 
$\chi(A/\pm 1, L_k)$. 

To show this, note that since $q: A \rightarrow A/\pm 1$ is finite the higher pushforward sheaves $R^i q_*q^*L_k$ vanish giving the equality 
\[
    H^i(A, q^*L_k) = H^i(A/\pm 1, q_*q^*L_k) 
\]
of higher cohomology groups using \cite[\href{https://stacks.math.columbia.edu/tag/01XC}{Lemma 01XC}]{stacks-project} 
and \cite[\href{https://stacks.math.columbia.edu/tag/01F4}{Lemma 01F4}]{stacks-project} as before. 
The canonical adjunction morphism $\O_{A/\pm 1}\rightarrow q_*\O_A$ exhibits $\O_{A/\pm 1}$ as a factor of $q_*\O_A$, so by push-pull 
\begin{align*}
    H^i(A, q^*L_k) &= H^i(A/\pm 1, q_*q^*L_k) \\
                &= H^i(A/\pm 1, L_k \otimes q_*\O_A)
\end{align*}
and hence $H^i(A/\pm 1, L_k)$ is a factor of $H^i(A, q^*L_k)$. In particular, to show $H^i(A/\pm 1, L_k) = 0$ for $i > 0$ it suffices to show that $H^i(A, q^*L_k) = 0$.

To this end, recall the previous diagram 
\begin{center}
\begin{tikzcd}
A\times A \arrow[d, "(\id\times q)"'] \arrow[r, "\xi"] & A\times A \arrow[d, "p"] \\
A\times (A/\pm 1) \arrow[r, "\mu"]                     & A^{(2)}                 
\end{tikzcd}
\end{center}
so that on the one hand 
\begin{align*}
(\id\times q)^*\mu^*\Theta_{(2)}^{\otimes k} &\cong (\id\times q)^*(\pi_1^*\Theta^{\otimes 2k}\otimes \pi_2^*L_k) \\ 
                                            &\cong \pi_1^*\Theta^{\otimes 2k}\otimes \pi_2^*q^*L_k
\end{align*}
where we use $\pi_1,\pi_2$ to denote projections on each of the product spaces $A\times A, A\times (A/\pm 1)$, but on the other hand 
\begin{align*}
(\id\times q)^*\mu^*\Theta_{(2)}^{\otimes k} &\cong \xi^*p^*\Theta_{(2)}^{\otimes k} \\
                                            &\cong \xi^*(\pi_1^*\Theta\otimes\pi_2^*\Theta)^{\otimes k} \\
                                            &\cong \pi_1^*\Theta^{\otimes 2k}\otimes\pi_2^*\Theta^{\otimes 2k},
\end{align*}
so that 
\[
\pi_1^*\Theta^{\otimes 2k}\otimes \pi_2^*q^*L_k \cong \pi_1^*\Theta^{\otimes 2k}\otimes\pi_2^*\Theta^{\otimes 2k}
\]
and hence 
$q^*L_k\cong \Theta^{\otimes 2k}$ since we may twist this equality by $\pi_1^*\Theta^{\otimes -2k}$ to cancel the first factor and 
then pushforward to $A$ under $\pi_2$ (noting that the fibers of $\pi_2$ are connected). We know that $H^i(A, \Theta^{\otimes 2k}) = 0$ 
for $i > 0$ since $k > 0$ by Kodaira vanishing. Thus, $H^i(A, q^*L_k) = 0$, so that $H^i(A/\pm 1, L_k) = 0$ as previously noted. 

Since we have shown that the higher cohomology of $L_k$ vanishes, we now calculate $\dim V_k = \chi(A/\pm 1, L_k)$. To do this, 
we will calculate $\chi(A, B_k) = \chi(A, \Sigma_*\Theta_{(2)}^{\otimes k})$ in two different ways. We recall the following fact: 
\begin{lemma}\label{chi_pushforward}
    If $f: X\rightarrow Y$ is a morphism and $\mathcal{G}$ a sheaf on $X$ such that $R^j f_*\mathcal{G} = 0$ for $j > 0$, then 
    $\chi(X, \mathcal{G}) = \chi(Y, f_*\mathcal{G})$. 
\end{lemma}
\begin{proof}
This is an immediate consequence of \cite[\href{https://stacks.math.columbia.edu/tag/01F4}{Lemma 01F4}]{stacks-project}, as the 
Leray spectral sequence $H^j(Y, R^i f_*\mathcal{G}) \Rightarrow H^{i+j}(X, \mathcal{G})$ will degenerate. 
\end{proof}
In light of this, we have the following: 
\begin{lemma}\label{sigma_pushforward}
For all $j > 0$, $R^j \Sigma_*\Theta_{(2)}^{\otimes k} = 0$, and hence 
\[
\chi(A, B_k) = \chi(A^{(2)}, \Theta_{(2)}^{\otimes k}). 
\]
\end{lemma}
\begin{proof}
Since our previous diagram was Cartesian, we have that 
\begin{align*}
    (\cdot 2)^*R^j\Sigma_*(\Theta_{(2)}^{\otimes k}) &= R^j\pi_{1*}\mu^*\Theta_{(2)}^{\otimes k} \\
                                                    &= R^j\pi_{1*}(\pi_1^*\Theta^{\otimes 2k}\otimes \pi_2^*L_k) \\
                                                    &= \Theta^{\otimes 2k} \otimes R^j\pi_{1*}\pi_2^*L_k \\
                                                    &= \Theta^{\otimes 2k} \otimes H^j(A/\pm 1, L_k) \\
                                                    &= 0
\end{align*}
since we showed that $H^j(A/\pm 1, L_k) = 0$. Since $R^j\Sigma_*(\Theta_{(2)}^{\otimes k})$ pulls back to the zero sheaf under the surjective morphism $\cdot 2:A\rightarrow A$, it must have been the zero sheaf to start with, as desired. 
\end{proof}

\begin{corollary}\label{first_chi_expression}
\[
    \chi(A, B_k) = \frac{k^2(k^2+1)}{2}
\]
for $k > 0$. 
\end{corollary}
\begin{proof}
Since $A^{(2)}$ is the quotient of $A^2$ by the finite group $\Z/2\Z$, the resolution $\pi: A^{[2]}\rightarrow A^{(2)}$ satisfies 
$R^j\pi_*\O_{A^{[2]}} = 0$ for $j > 0$ by \cite{viehweg1977rational}, see also \cite[Theorem 2]{chatzistamatiou2012higher} for the 
same result in positive characteristic. Thus,
\[
    \chi(A^{(2)}, \Theta_{(2)}^{\otimes k}) = \chi(A^{[2]}, \pi^*\Theta_{(2)}^{\otimes k})
\]
by push-pull and Lemma \ref{chi_pushforward}. We may calculate the latter term by \cite[Lemma 5.1]{EGL}, obtaining that 
\begin{align*}
    \chi(A, B_k) &= \chi(A^{(2)}, \Theta_{(2)}^{\otimes k}) \\
                &= \chi(A^{[2]}, \pi^*\Theta_{(2)}^{\otimes k}) \\
                &= \binom{\chi(A, \Theta^{\otimes k}) + 1}{2} \\
                &= \binom{k^2 + 1}{2}, 
\end{align*}
as desired. 
\end{proof}
\begin{corollary}
    \[
    \dim V_k = 2(k^2 + 1).
    \]
\end{corollary}
\begin{proof}
We calculate $\chi(A, B_k)$ another way. 
We have that $\chi(A, B_k) = \frac{1}{16}\chi(A, (\cdot 2)^*B_k)$ by \cite[Corollary 9.12]{vandergeer_abvar} since $\cdot 2: A\rightarrow A$ has degree 16. We may compute this via our Cartesian diagram to find 
\begin{align*}
    \chi(A, B_k) &= \frac{1}{16}\chi(A, (\cdot 2)^*B_k) \\  
                &= \frac{1}{16}\chi(A, \pi_{1*}\mu^*(\Theta_{(2)}^{\otimes k})) \\
                &= \frac{1}{16}\chi(A, \pi_{1*}(\pi_1^*\Theta^{\otimes 2k}\otimes \pi_2^*L_k)) \\
                &= \frac{1}{16}\chi(A, \Theta^{\otimes 2k}\otimes V_k) \\
                &= \frac{\dim V_k}{16}\chi(A, \Theta^{\otimes 2k}) \\
                &= \frac{k^2 \dim V_k}{4}.
\end{align*}
Equating this last expression with $\dfrac{k^2(k^2+1)}{2}$ from Corollary \ref{first_chi_expression} gives the desired equality. 
\end{proof}
We now conclude with our desired result:
\begin{proposition}
If $k>0$ and $2k+4\ell > 0$, then 
\[
    \dim H^0(A^{(2)}, \Theta_{(2)}^{\otimes k}\otimes \Sigma^*(\Theta^{\otimes \ell}\otimes L_0)) = \frac{(k^2 + 1)(k + 2\ell)^2}{2}
\]
for $L_0\in\Pic^0(A)$.
\begin{proof}
    As noted, this is the same as calculating the dimension of $H^0(A, \mathcal{F}_{k,\ell})$, and this is simply 
    $\chi(A, \mathcal{F}_{k,\ell})$ due to the vanishing of higher cohomology. 
    We find that 
    \begin{align*}
    \chi(A, \mathcal{F}_{k,\ell}) &= \frac{1}{16}\chi(A, (\cdot 2)^*\mathcal{F}_{k,\ell}) \\  
                                &= \frac{1}{16}\chi(A, (\cdot 2)^*(B_k\otimes (\Theta^{\otimes \ell}\otimes L_0))) \\
                                &= \frac{1}{16}\chi(A, \pi_{1*}\mu^*\Theta_{(2)}^{\otimes k}\otimes (\cdot 2)^*(\Theta^{\otimes \ell}\otimes L_0)) \\
                                &= \frac{1}{16}\chi(A, (\Theta^{\otimes 2k}\otimes V_k)\otimes \Theta^{\otimes 4\ell}\otimes L_0^{\otimes 2}) \\
                                &= \frac{\dim V_k}{16}\chi(A, \Theta^{\otimes 2k+4\ell}\otimes L_0^{\otimes 2}) \\
                                &= \frac{k^2 + 1}{8}(2k+4\ell)^2 \\
                                &= \frac{(k^2 + 1)(k + 2\ell)^2}{2},
    \end{align*}
    as desired. 
\end{proof}
\end{proposition}

\subsection{Dimensions of global sections on the Hilbert square}\label{global_sections_hilb2}\hfill\\
Our formulas in the previous sections show that if $L_0\in\Pic^0(A^{[2]})$ is not a 2-torsion line bundle, then the line bundle 
$\Theta_{[2]}^{\otimes 2k}\otimes \Sigma^*\Theta^{-\otimes k}\otimes L_0$ on $A^{[2]}$ has no global sections, for $\Theta$ a symmetric 
principal polarization. However, when $L_0$ is 2-torsion we can get nonzero global sections. We can go further and give bounds for 
the dimension of
\[
    H^0(A^{[2]}, \Theta_{[2]}^{\otimes 2k} \otimes \Sigma^*\Theta^{\otimes -k} \otimes \O_{A^{[2]}}(-mB))
\]
where $k,m$ are a positive integers. Note that the pushforward $\pi_*\O_{A^{[2]}}(-mB)$ is not trivial but rather a power of the ideal sheaf 
of the image, which will nontrivially impact the number of global sections. For simplicity we will only consider the case where 
$L_0 \cong \O_{A^{[2]}}$ and $\Theta$ is the symmetric principal polarization given by the vanishing of the Riemann theta function, 
which we will shortly explain more concretely. These assumptions will suffice for our proof. 
One could devise similar formulas for the general case by considering theta functions with characteristics \cite[Exercise 4.11]{BL_complex_abvar}. We only require some of the results of this section to prove our main result, but include the full formulas as we find the calculations interesting. 

Consider the difference map $d:A\times A\rightarrow A$, $d(x,y) = x-y$. It is easy to check by the 
theorem of the square that $d^*L \cong (L\boxtimes L)^{\otimes 2}\otimes\Sigma^*L^{\otimes -1}$ for any symmetric line bundle $L$ on $A$. 
We extend our usual Cartesian blowup square with $d$: 
\begin{center}
\begin{tikzcd}
Bl_{\Delta}A^2 \arrow[d, "q"'] \arrow[r, "\hat{\pi}"] & A^2 \arrow[d, "p"] \arrow[r, "d"] & A \\
{A^{[2]}} \arrow[r, "\pi"]                            & A^{(2)}                           &  
\end{tikzcd}
\end{center}
Set $T_k = \Theta_{(2)}^{\otimes 2k}\otimes\Sigma^*\Theta^{\otimes -k}$ on $A^{(2)}$ for simplicity. By our formula 
for $d^*L$ we see that $p^*T_k \cong d^*\Theta^{\otimes k}$.
We may calculate global sections on $A^{[2]}$ by pulling back 
to $Bl_{\Delta}A^2$ and taking $\mathfrak{S}_2$-invariants by formula \ref{S2_invar_sections_no_pullback}, so that 
\begin{align*}
    H^0(A^{[2]}, \Theta_{[2]}^{\otimes 2k} \otimes \Sigma^*\Theta^{\otimes -k} \otimes \O_{A^{[2]}}(-mB)) &\cong 
        H^0(A^{[2]}, \pi^*T_k\otimes\O_{A^{[2]}}(-mB)) \\
        &\cong H^0(Bl_{\Delta}A^2, q^*(\pi^*T_k\otimes\O_{A^{[2]}}(-mB)))^{\mathfrak{S}_2} \\
        &\cong H^0(Bl_{\Delta}A^2, \hat{\pi}^*p^*T_k\otimes q^*\O_{A^{[2]}}(-mB))^{\mathfrak{S}_2} \\
        &\cong H^0(Bl_{\Delta}A^2, \hat{\pi}^*d^*\Theta^{\otimes k} \otimes \O_{Bl_{\Delta}A^2}(-mF))^{\mathfrak{S}_2}
\end{align*}
where $F$ is the exceptional divisor on $Bl_{\Delta}A^2$.
Since $d(y,x) = -d(x,y)$, we have a commutative diagram 
\begin{center}
\begin{tikzcd}
Bl_{\Delta}A^2 \arrow[r, "\hat{\pi}"] \arrow[d, "\hat{\sigma}"] & A^2 \arrow[r, "d"] \arrow[d, "\sigma"] & A \arrow[d, "-1"] \\
Bl_{\Delta}A^2 \arrow[r, "\hat{\pi}"]                           & A^2 \arrow[r, "d"]                     & A                
\end{tikzcd}
\end{center}
where $\sigma, \hat{\sigma}$ are the usual involutions on $A^2, Bl_{\Delta}A^2$ and $-1$ is multiplication by $-1$ on $A$. Thus, in the 
equality 
\[
    H^0(Bl_{\Delta}A^2, \hat{\pi}^*d^*\Theta^{\otimes k} \otimes \O_{Bl_{\Delta}A^2}(-mF))\cong H^0(A, \Theta^{\otimes k}\otimes(d\circ\hat{\pi})_*\O_{Bl_{\Delta}A^2}(-mF))
\]
we may take $\mathfrak{S}_2$-invariants on both sides, where the action on the left is by swapping factors of $A^2$ and on the right we 
multiply $A$ by $-1$. It remains only to describe the pushforward $(d\circ\hat{\pi})_*\O_{Bl_{\Delta}A^2}(-mF)$. 
Since $\hat{\pi}:Bl_{\Delta}A^2\rightarrow A^2$ is a blowup of smooth varieties, 
\[
    \hat{\pi}_*\O_{Bl_{\Delta}A^2}(-mF) \cong \mathcal{I}_{\Delta}^{\otimes m},
\]
where $\mathcal{I}_{\Delta}$ is the ideal sheaf of the diagonal in $A^2$. The map $d:A^2\rightarrow A$ 
is a surjective map of smooth projective varieties with equidimensional fibers (they are all isomorphic to $A$ by calculation), so 
$d$ is flat by miracle flatness. The inverse image of $0\in A$ is just the diagonal $\Delta\subset A^2$, so by flatness the 
inverse image ideal sheaf $\mathcal{I}_{\Delta}\cong d^{-1}(\mathcal{I}_0)\cdot\O_{A^2}$ coincides with the pullback $d^*\mathcal{I}_0$. 
We conclude that 
\begin{align*}
H^0(A^{[2]}, \Theta_{[2]}^{\otimes 2k} \otimes \Sigma^*\Theta^{\otimes -k} \otimes \O_{A^{[2]}}(-mB)) &\cong
    H^0(Bl_{\Delta}A^2, \hat{\pi}^*d^*\Theta^{\otimes k} \otimes \O_{Bl_{\Delta}A^2}(-mF))^{\mathfrak{S}_2} \\
    &\cong H^0(A, \Theta^{\otimes k}\otimes(d\circ\hat{\pi})_*\O_{Bl_{\Delta}A^2}(-mF))^+ \\
    &\cong H^0(A, \Theta^{\otimes k}\otimes\mathcal{I}_0^{\otimes m})^+
\end{align*}
where we write $+$ for the sections invariant under pullback by $\pm 1$ on the $A$ side. This is the space of even theta functions of weight $k$ which vanish to 
order $m$ at 0. Though we do not have exact formulas, we can estimate the number of these. 

\begin{remark}\label{twisted_difference_formula}
The relevance of $L_0$ being 2-torsion or not in our previous formula now makes sense: 
if $L_0$ is not 2-torsion then $(-1)^*(\Theta\otimes L_0) \cong \Theta\otimes L_0^{\otimes -1} \neq \Theta\otimes L_0$, so that 
we cannot even define the action of $\pm 1$ on $H^0(A, \Theta\otimes L_0)$. If $L_0\in\Pic^0(A)$ is 2-torsion, then 
\begin{align*}
d^*(\Theta^{\otimes k}\otimes L_0) &\cong (\Theta^{\otimes 2k}\otimes L_0^{\otimes 2})\boxtimes(\Theta^{\otimes 2k}\otimes L_0^{\otimes 2})\otimes\Sigma^*(\Theta^{\otimes -k}\otimes L_0^{\otimes -1}) \\
    &\cong (\Theta\boxtimes\Theta)^{\otimes 2k}\otimes\Sigma^*(\Theta^{\otimes -k}\otimes L_0), 
\end{align*}
so that if $L_0$ is an arbitrary 2-torsion line bundle in $\Pic^0(A)$ (equivalently $\Pic^0(A^{[2]})$) then 
\[
H^0(A^{[2]}, \Theta_{[2]}^{\otimes 2k}\otimes\Sigma^*\Theta^{\otimes -k}\otimes\O_{A^{[2]}}(-mB)\otimes\Sigma^*L_0) \cong 
    H^0(A, \Theta^{\otimes k}\otimes \mathcal{I}_0^{\otimes m}\otimes L_0)^+
\]
as well.
\end{remark}
Our work that follows here consists of well-known results and is very similar to the proofs of \cite[Corollary 4.6.6]{BL_complex_abvar} and copies the work in Section 15.6 of the same. 
Suppose that $A = \C^g/\Lambda$ is a principally polarized abelian variety defined by a matrix $\tau$ in the Siegel upper half space, 
such that the principal polarization $\Theta$ has the Riemann theta function
\[
    \theta(z) = \sum_{n\in \Z^g}\exp(2\pi i n^t z + \pi i n^t \tau n)
\] 
as a section (where we view $n, z$ as columns vectors so that the transposes make sense to yield scalars). The weight $k$ theta 
functions $H^0(A, \Theta^{\otimes k})$ have as basis the $k^g$ functions of the form
\[
    \theta_{\epsilon}(z) = \sum_{n\equiv\epsilon\pmod{(k\Z)^g}} \exp\left(2\pi i n^t z + \frac{\pi}{k} i n^t \tau n\right)
\]
where $\epsilon$ is a vector in $(\Z/k\Z)^g$ \cite[equation 3.1]{beauville2013theta}. We see that this sum is unchanged if we replace $z$ with $-z$ and $n$ with $-n$, so that 
$\theta_{\epsilon}(z) = \theta_{-\epsilon}(-z)$ for all $z, \epsilon$. Thus, if $s(z) = \sum_{\epsilon}a_{\epsilon}\theta_{\epsilon}(z)$ 
satisfies $s(z) = s(-z)$, then $a_{\epsilon} = a_{-\epsilon}$ for all $\epsilon$, and hence the space $H^0(A, \Theta^{\otimes k})^+$ 
of even weight $k$ theta functions is spanned by the orbit vectors $\theta_{\epsilon} + \theta_{-\epsilon}$, 
or simply $\theta_{\epsilon}$ if $2\epsilon \equiv 0 \pmod{(k\Z)^g}$. 
It is easy to check that this collection is linearly independent as well, so that to determine 
the dimension of $H^0(A, \Theta^{\otimes k})^+$ we simply need to count the number of these orbit vectors. 
Set $X$ to be the set $(\Z/k\Z)^g$, so that this number of distinct orbit vectors is simply the sum $|X^{fixed}| + \frac{1}{2}|X^{mov}|$, where 
$X^{fixed}$ consists of those $\epsilon$ modulo $(k\Z)^g$ unchanged by multiplying by $-1$, and $X^{mov}$ consists of those $\epsilon$ 
that are not fixed. If $k$ is odd then the only fixed $\epsilon$ is the zero vector, and if $k$ is even then in each coordinate one 
may choose either $0$ or $\frac{k}{2}$ to get a fixed vector, yielding $2^g$ fixed $\epsilon$. Since $X^{mov}$ is the complement of 
$X^{fixed}$ and $|X| = k^g$, we see that 
\begin{formula}\label{even_theta_formula}
\[
\dim H^0(A, \Theta^{\otimes k})^+ =
    \begin{cases}
        \frac{1}{2}(k^g + 2^g), & \text{if } k \equiv 0 \pmod{2} \\
        \frac{1}{2}(k^g + 1), & \text{if } k \equiv 1 \pmod{2}
    \end{cases}.
\]
\end{formula}

We now incorporate the order of vanishing at 0, and return to the case $g = 2$ of abelian surfaces for simplicity. 
Choosing local complex coordinates 
$z_1, z_2$ about 0, if $s(z) = s(-z)$ then all of the terms of odd degree in the Taylor expansion of $s$ about 0 will 
vanish automatically. For $s$ to vanish at 0 up to order $m$ amounts to the terms of degree $\le m-1$ all vanishing in the Taylor 
expansion, so that if $s$ is even then its order of vanishing will automatically be even as well (if it vanishes up to odd order $m$ 
then it automatically vanishes to even order $m+1$). In dimension 2 there are $d+1$ Taylor terms of degree $d$ 
($z_1^d, z_1^{d-1}z_2, ..., z_2^d$). Write $\mathcal{I}_0^{\otimes \ell}$ for the ideal sheaf of functions vanishing to order $\ell$ 
at $0\in A$, so that we wish to estimate the dimension of $H^0(A, \Theta^{\otimes k}\otimes\mathcal{I}_0^{\otimes 2m})^+$ for integers 
$k, m$ (we put $2m$ since the order of vanishing is automatically even for even theta functions as discussed). Since we only care about 
the vanishing of the even Taylor series terms, vanishing to order $2m$ imposes 
\[
    1 + 3 + 5 + ... + (2m - 1) = m^2
\]
linear conditions on the vector space $H^0(A, \Theta^{\otimes k})^+$. Some of these linear conditions may be redundant, giving a 
lower bound 
\[
    \dim H^0(A, \Theta^{\otimes k}\otimes\mathcal{I}_0^{\otimes 2m})^+ \ge \dim H^0(A, \Theta^{\otimes k})^+ - m^2. 
\]
We may plug in our formula \ref{even_theta_formula} to compute the latter term. Putting everything together, we conclude the following:
\begin{proposition}\label{h0_2k_estimate}
For $m \ge 0$, 
\[
    \dim H^0(A^{[2]}, \Theta_{[2]}^{\otimes 2k}\otimes\Sigma^*\Theta^{\otimes -k}\otimes\O_{A^{[2]}}(-(2m-1)B)) =
    \dim H^0(A^{[2]}, \Theta_{[2]}^{\otimes 2k}\otimes\Sigma^*\Theta^{\otimes -k}\otimes\O_{A^{[2]}}(-2mB)),
\]
and 
\[
    \dim H^0(A^{[2]}, \Theta_{[2]}^{\otimes 2k}\otimes\Sigma^*\Theta^{\otimes -k}\otimes\O_{A^{[2]}}(-2mB)) \ge 
        \frac{1}{2}(k^2 + 4) - m^2 
\]
for even $k$ and 
\[
    \dim H^0(A^{[2]}, \Theta_{[2]}^{\otimes 2k}\otimes\Sigma^*\Theta^{\otimes -k}\otimes\O_{A^{[2]}}(-2mB)) \ge 
        \frac{1}{2}(k^2 + 1) - m^2 
\]
for odd $k$, for our given choice of symmetric principal polarization $\Theta$. If $m=0$ this is an equality. 
\end{proposition}

\subsection{Kummer K3 surfaces}\label{kummer_k3}\hfill\\
We will need to study divisors on the Kummer K3 surface associated to $A$.

Let $Km(A)$ denote the Kummer K3 surface associated to $A$, which we may view as the fiber over 0 of the summation morphism 
$\Sigma:A^{[2]}\rightarrow A$, or equivalently as the blowup of the singular surface $A/\pm 1$ at the 16 singular points corresponding 
to the 2-torsion points of $A$. We may define a map 
\[
    \tilde{\mu}:A\times Km(A)\rightarrow A^{[2]}
\]
by sending a pair $(a, \mathcal{Z})$ to the natural translation $t_a^{[2]}(\mathcal{Z})$. 
This map fits into the commutative diagram
\begin{equation}\label{mu_square}
\begin{tikzcd}
A\times Km(A) \arrow[r, "\tilde{\mu}"] \arrow[d, "\pi_1"'] & {A^{[2]}} \arrow[d, "\Sigma"] \\
A \arrow[r, "\cdot 2"]                         & A                            
\end{tikzcd}
\end{equation}
which is Cartesian as noted by \cite[footnote 2]{Beauville1983VaritsKD}. Thus, $\tilde{\mu}$ has degree 16 since $(\cdot 2)$ does.

Since $\tilde{\mu}$ is a map of smooth projective varieties with equidimensional fibers, 
it is flat by miracle flatness. 
Viewing $Km(A)$ as the blowup of $A/\pm 1$ at the 16 singular 2-torsion points gives that $\tilde{\mu}$ is simply the functorial extension 
of our previously-defined $\mu: A\times (A/\pm 1) \rightarrow A^{(2)}$ to the blowup. To elaborate, tracing out the definitions shows that the diagram 
\begin{equation}\label{descent_square}
\begin{tikzcd}
A\times Km(A) \arrow[r, "\tilde{\mu}"] \arrow[d, "\pi_1\times b"'] & {A^{[2]}} \arrow[d, "\pi"] \\
A\times (A/\pm 1) \arrow[r, "\mu"]                               & A^{(2)}                   
\end{tikzcd}
\end{equation}
commutes, by identifying a pair of points $(x,y)\in A^{(2)}$ whose sum $x+y$ is zero with the equivalence class 
$x\sim -x \in A/\pm 1$. The preimage of the diagonal of $A^{(2)}$ under $\mu$ is simply $A\times A[2]\subset A\times A/\pm 1$ (that is, 
consists of any point in the first factor and a singular point of $A/\pm 1$ corresponding to a 2-torsion point of $A$ 
in the second factor), so that by the universal property of blowups there exists a unique morphism $f:A\times Km(A)\rightarrow A^{[2]}$ 
such that the diagram
\begin{center}
\begin{tikzcd}
A\times Km(A) \arrow[r, "f"] \arrow[d, "\pi_1\times b"'] & {A^{[2]}} \arrow[d, "\pi"] \\
A\times (A/\pm 1) \arrow[r, "\mu"]                               & A^{(2)}                   
\end{tikzcd}
\end{center}
commutes, so that this induced $f$ agrees with $\tilde{\mu}$. 

Recall that the second cohomology $H^2(Km(A), \Z)$ contains a factor isomorphic to $H^2(A, \Z)$ with double the original intersection 
form, and this factor is the orthogonal complement to the sublattice of $H^2(Km(A), \Z)$ rationally generated by the classes $E_i$ of the 
16 exceptional divisors \cite[Section 2]{garbagnati2016kummer}. 
Moreover, this decomposition holds when restricting to the Picard lattice of $Km(A)$, and the map 
$\tau: H^2(A, \Z) \hookrightarrow H^2(Km(A), \Z)$ is induced via pullback and pushforward along the commutative blowup diagram
\begin{center}
\begin{tikzcd}
{Bl_{A[2]}A} \arrow[d, "\hat{b}"'] \arrow[r, "\hat{q}"] & Km(A) \arrow[d, "b"] \\
A \arrow[r, "q"]                 & A/\pm 1             
\end{tikzcd}
\end{center}
\cite[Section VIII.5]{barth2015compact}. When $A$ has Picard rank 1 with ample generator $\Theta$, the Picard lattice 
is of rank 17, containing an ample class we denote $H = \tau(\Theta)$ such that $H^2 = 2\Theta^2$.

We wish to calculate the pullback under $\tilde{\mu}$ of an arbitrary line bundle on $A^{[2]}$.
\begin{proposition}\label{kummer_k3_pullback}
Set $L$ = $\Theta_{[2]}^{\otimes k}\otimes\Sigma^*\Theta^{\otimes \ell}\otimes\O_{A^{[2]}}(mB)\otimes \Sigma^*L_0$
for integers $k,\ell,m$ and $L_0\in\Pic^0(A)$. 

Then
\[
\tilde{\mu}^*L \cong (\Theta^{\otimes 2k+4\ell}\otimes L_0^{\otimes 2})\boxtimes \O_{Km(A)}\left(kH + \frac{m}{2}\sum_{i=1}^{16}E_i\right).
\]
\end{proposition}
\begin{proof}
Using the notation and work from Subsection \ref{symm_prod_h0}, we know that 
\begin{align*}
    \tilde{\mu}^*L &\cong (\pi_1\times b)^*\mu^*(\Theta_{(2)}^{\otimes k}\otimes \Sigma^*\Theta^{\otimes \ell}\otimes \Sigma^*L_0)\otimes\tilde{\mu}^*\O_{A^{[2]}}(mB) \\
        &\cong ((\Theta^{\otimes 2k+4\ell}\otimes L_0^{\otimes 2})\boxtimes b^*L_k)\otimes \tilde{\mu}^*\O_{A^{[2]}}(mB).
\end{align*}
We thus need to calculate $b^*L_k$ and $\tilde{\mu}^*\O_{A^{[2]}}(mB)$. 
For any divisor $D$ on $A/\pm 1$ we see that $\tau(q^*D) = 2b^*D$ since 
\begin{align*}
\tau(q^*D) &= \hat{q}_*(\hat{b}^*q^*)(D) \\
    &= \hat{q}_*(\hat{q}^*b^*)(D) \\
    &= 2b^*D
\end{align*}
as $\hat{q}:Bl_{A[2]}A\rightarrow Km(A)$ is generically of degree two. 
We may thus calculate $b^*L_k$ by looking at the pullback to $A$ and dividing 
by two. From the original commutative diagram 
\begin{center}
\begin{tikzcd}
A\times A \arrow[r, "\xi"] \arrow[d, "\pi_1\times q"'] & A\times A \arrow[d, "p"] \\
A\times (A/\pm 1) \arrow[r, "\mu"]                       & A^{(2)}                 
\end{tikzcd}
\end{center}
where $\xi:A\times A\rightarrow A\times A$ is the Wirtinger map we see that $q^*L_k$ is just the second factor of 
$\xi^*(\Theta\boxtimes\Theta)^{\otimes k} = \Theta^{\otimes 2k}\boxtimes \Theta^{\otimes 2k}$ on $A\times A$, so that 
\[
b^*L_k \cong \tau(\Theta^{\otimes k}) = \O_{Km(A)}(kH).
\] 

To calculate $\tilde{\mu}^*\O_{A^{[2]}}(mB)$, we need 
to determine the preimage of the exceptional divisor $E = 2B$ under $\tilde{\mu}$. Since any natural automorphism preserves the 
multiplicity structure of a subscheme, we see that $\tilde{\mu}(a,\mathcal{Z}) = t_a^{[2]}(\mathcal{Z})$ lies in $E$ if and only if 
$\mathcal{Z}$ is a subscheme supported at a single point with multiplicity 2 (with no condition on the point $a$ in the first factor), 
so that $\tilde{\mu}^*E$ consists of the 16 exceptional divisors $\sum_{i=1}^{16} E_i \in \Pic(A\times Km(A))$, 
and hence $\tilde{\mu}^*B = \frac{1}{2}\sum_{i=1}^{16}E_i$ (since $\tilde{\mu}$ is flat, we may identify pullbacks of line bundles with preimages of divisors). We thus conclude that 
\[
    \tilde{\mu}^*\O_{A^{[2]}}(mB) \cong \O_A\boxtimes\O_{Km(A)}\left(\frac{m}{2} + \sum_{i=1}^{16}E_i\right). 
\]
Putting everything together, we find that 
\begin{align*}
\tilde{\mu}^*L &\cong ((\Theta^{\otimes 2k+4\ell}\otimes L_0^{\otimes 2})\boxtimes b^*L_k)\otimes\tilde{\mu}^*\O_{A^{[2]}}(mB) \\
            &\cong (\Theta^{\otimes 2k+4\ell}\otimes L_0^{\otimes 2})\boxtimes \O_{Km(A)}\left(kH + \frac{m}{2}\sum_{i=1}^{16}E_i\right),
\end{align*}
as desired. 
\end{proof}\noindent

\subsubsection{Automorphisms of Jacobian Kummer K3 surfaces}\hfill\\
The preceding discussion did not require $A$ to have a specific polarization type, and can clearly be extended 
to abelian surfaces of higher Picard rank. We now suppose that $A$ is principally polarized, so that $Km(A)$ is a \textit{Jacobian} 
Kummer surface. Equivalently, $Km(A)$ is the minimal resolution of a quartic surface in $\P^3$ with 16 nodes. Given a principally 
polarized abelian surface $(A,\Theta)$, the image of linear series $|2\Theta|$ is precisely this singular quartic surface in $\P^3$. 

The automorphisms of Jacobian Kummer surfaces have been studied since the 19th century, and completely classified in the generic 
Picard rank 17 case by Kondo \cite{kondo1997automorphism}. There, Kondo shows that $\Aut(Km(A))$ is generated by various classically 
known automorphisms, along with 192 automorphisms of infinite order constructed by Keum \cite{keum1997automorphisms}. One of these 
classical automorphisms is the \textit{switch} involution $\sigma$ on $Km(A)$, which arises from the fact that the singular 
quartic Kummer surface in $\P^3$ is birational to its projective dual surface in $\P^3$. Since K3 surfaces are minimal, this birational 
involution lifts to an isomorphism of $Km(A)$. The following is known:
\begin{lemma}\label{switch_action_lemma}
    \[\sigma^*H = 3H - \sum_{i=1}^{16}E_i\] 
    and 
\[
    \sigma^*\left(\frac{1}{2}\sum_{i=1}^{16}E_i\right) = 4H - \frac{3}{2}\sum_{i=1}^{16}E_i
\]
in the Picard lattice of $Km(A)$. 
\end{lemma}\noindent
To see this, refer to \cite[Table (5.2)]{keum1997automorphisms}, where the $E_i$ are written as $N_{\alpha}$ (``nodes'') for 
$\alpha\in(\Z/2\Z)^4$. In their notation, this will yield that 
\[\sigma^*\left(\frac{1}{2}\sum_{\alpha}N_{\alpha}\right) = \frac{1}{2}\sum_{\alpha}T_{\alpha}\]
where the $T_{\alpha}$ are rational curves in $Km(A)$ called the \textit{tropes}. To evaluate this sum of tropes, use section (1.8) 
of the same paper to write the tropes in terms of $H$ and the $N_{\alpha}$. 

\section{Proof of main theorem}
\subsection{Descent under the summation morphism and proof strategy}\label{descent_strategy}\hfill\\
We return now to the proof of Theorem \ref{main_theorem}. 
\begin{proposition}\label{summation_descent}
Given a morphism $g:A^{[n]}\rightarrow A^{[n]}$ where $A$ is a complex abelian surface, there exists a morphism $\bar{g}:A\rightarrow A$ 
of algebraic varieties such that the diagram 
\begin{center}
\begin{tikzcd}
{A^{[n]}} \arrow[r, "g"] \arrow[d, "\Sigma"'] & {A^{[n]}} \arrow[d, "\Sigma"] \\
A \arrow[r, "\bar{g}"]                        & A                            
\end{tikzcd}
\end{center}
commutes.
\end{proposition}
\begin{proof}
Consider the composition $\Sigma\circ g:A^{[n]}\rightarrow A$. 
Generalized Kummer varieties are simply connected and hence have trivial Albanese variety (as then $H^{1,0} = 0$).
The image of any generalized Kummer fiber of $\Sigma$ under this map $\Sigma\circ g$ must be a single point, and hence 
$g$ sends fibers of $\Sigma$ to fibers. This gives a set-theoretic map $\bar{g}$ on $A$ which commutes with $g$ under $\Sigma$, which we must 
show is in fact a morphism of varieties.  This map $\bar{g}$ is well-defined 
and satisfies the preceding commutative diagram by construction. To show that this function 
is a morphism, we need to show that it is continuous with respect to the Zariski topology on $A$ and that it pulls back regular 
functions on open subsets of $A$ to regular functions. 

We first observe that $\Sigma:A^{[n]}\rightarrow A$ is flat by miracle flatness \cite[Exercise III.10.9]{hartshorne1977algebraic}, 
since $A^{[n]}$ and $A$ are smooth and $\Sigma$ has equidimensional fibers. 

By \cite[\href{https://stacks.math.columbia.edu/tag/01UA}{Lemma 01UA}]{stacks-project} $\Sigma$ is thus an open map. Since 
$\Sigma$ is surjective we have that $\Sigma(\Sigma^{-1}(X)) = X$ for any set $X\subset A$, so that for any open set $U\subset A$ we have that 
\begin{align*}
    \bar{g}^{-1}(U) &= \Sigma(\Sigma^{-1}(\bar{g}^{-1}(U))) \\
                    &= \Sigma(g^{-1}(\Sigma^{-1}(U))) 
\end{align*}
which is open since $\Sigma$ is an open mapping, and $g, \Sigma$ are continuous, so that $\bar{g}$ is continuous. 

Let us now prove the second condition that $\bar{g}$ pulls back local regular functions to local regular functions. Let 
$r\in \O_A(V)$ be a local regular function for some open set $V$ in $A$, and write $U = \bar{g}^{-1}(V)$ for convenience. We wish to show that the composition $(r\circ \bar{g})|_U:U\rightarrow\C$ lies in 
$\O_A(U)$. If $U$ does not contain the point 0, fix a length $n-1$ subscheme $\mathcal{Z}_0$ of $A$ supported entirely at 0.
There then exists a well-defined local section 
$s|_U : U\rightarrow \Sigma^{-1}(U)$ by simply defining 
\[
s|_U(p) = \mathcal{Z}_0 + (p) \in A^{[n]},
\]
so that $\Sigma\circ s|_U$ is the identity. We thus find that 
\begin{align*}
    (r\circ\bar{g})|_U &= r\circ\bar{g}\circ\Sigma\circ s|_U \\
                &= r\circ\Sigma\circ g\circ s|_U
\end{align*}
which is regular since $\Sigma, g, s|_U$ are all morphisms. If $U$ does contain the point 0, then the same argument gives 
that the pullback of $r$ to $U - \{0\}$ is regular, which extends uniquely to a regular function on $U$ by algebraic Hartogs since 
$A$ is two-dimensional. This unique extension to a regular function on $U$ is just $r\circ \bar{g}$, so that in either case 
$r$ pulls back to a regular function as desired, proving that $\bar{g}$ is a morphism. 
\end{proof}
The assignment $g\mapsto \bar{g}$ is clearly functorial by considering the following commutative diagram: 
\begin{center}
\begin{tikzcd}
{A^{[n]}} \arrow[r, "g"] \arrow[d, "\Sigma"'] & {A^{[n]}} \arrow[r, "h"] \arrow[d, "\Sigma"'] & {A^{[n]}} \arrow[d, "\Sigma"] \\
A \arrow[r, "\bar{g}"]                        & A \arrow[r, "\bar{h}"]                        & A                          
\end{tikzcd}
\end{center}
where $g,h:A^{[n]}\rightarrow A^{[n]}$ are arbitrary morphisms. It then follows that if $g:A^{[n]}\rightarrow A^{[n]}$ 
is an automorphism then so is $\bar{g}$, giving a group homomorphism which we denote by 
\begin{align*}
    &\ind_A^{\Sigma} : \Aut(A^{[n]}) \rightarrow \Aut(A) \\ 
    &\ind_A^{\Sigma}(g) = \bar{g}
\end{align*}
(to be read as ``the automorphism on $A$ induced by/under $\Sigma$''). 
Considering the natural automorphism map $(-)^{[n]}:\Aut(A)\rightarrow \Aut(A^{[n]})$, it is not true that 
$\ind_A^{\Sigma}\circ (-)^{[n]}$ is the identity, since an automorphism of $A$ as a variety may include a nontrivial translation term. 
A simple calculation shows that this composition is the identity when restricted to group automorphisms of $A$, however. Moreover, this gives a surjective map from $\Aut(A^{[n]})$ to $\Aut(A)$. 
\begin{proposition}
The map \[\ind_A^{\Sigma}:\Aut(A^{[n]})\rightarrow \Aut(A)\] is surjective. 
\end{proposition}
\begin{proof}
Let $g\in \Aut(A)$ be given, and decompose it as $g = t_a \circ h$, where $h$ is a group automorphism and $t_a$ is the translation 
by $a\in A$. Let $b\in A$ be a point satisfying $nb = a$ - there are $n^4$ possible choices for $b$. Fix a point $p\in A$ and let 
$\mathcal{Z}\in A^{[n]}$ be an arbitrary point in the preimage $\Sigma^{-1}(p)$, and write $\mathcal{Z} = \sum_{i=1}^k a_i(x_i)$ for 
the support of $\mathcal{Z}$ with multiplicity. Our argument only cares about multiplicities rather than finer scheme structure, 
so we will abuse notation and identify $\mathcal{Z}$ with this formal sum.
We wish to show that $\ind_A^{\Sigma}((t_b\circ h)^{[n]}) = g$. We calculate 
\begin{align*}
    \Sigma((t_b\circ h)^{[n]}\left(\sum_{i=1}^k a_i(x_i)\right)) &= \Sigma\left(\sum_{i=1}^k a_i(t_b\circ h)(x_i)\right) \\
                                        &= \Sigma\left(\sum_{i=1}^k a_i(h(x_i) + b)\right) \\
                                        &= \sum_{i=1}^k a_ih(x_i) + nb \\
                                        &= h\left(\sum_{i=1}^k a_ix_i\right) + a \\
                                        &= h(p) + a \\
                                        &= (t_a \circ h)(p) \\
                                        &= g(p)
\end{align*}
where we hope it is clear where $\Sigma$ refers to a formal sum of points versus addition on $A$,
so that $\ind_A^{\Sigma}((t_b \circ h)^{[n]}) = g$ and hence $\ind_A^{\Sigma}$ is surjective, as desired. 
\end{proof}
The calculation in the preceding proposition shows that $\ind_A^{\Sigma}$ is not injective however, since any of the $n^4$ possible 
points $b\in A$ satisfying $nb = a$ would yield the same $g = \ind_A^{\Sigma}((t_b \circ h)^{[n]})$ but different initial 
automorphisms $(t_b \circ h)^{[n]} \in \Aut(A^{[n]})$. For the case $n=2$ we will show that this fully describes the fibers of 
$\ind_A^{\Sigma}: \Aut(A^{[2]}) \rightarrow \Aut(A)$ for certain polarization types by proving that all automorphisms of $A^{[2]}$ are natural in those cases. 
\newline\newline
For the rest of the paper, let $A$ denote a complex abelian surface of Picard rank 1, so that $\End(A) = \Z$ with group automorphisms 
$\Z/2\Z$, and hence its full automorphism group $\Aut(A)$ is a disjoint union of two complex tori of dimension 2. We fix a symmetric 
polarization $\Theta$ satisfying $(-1)^*\Theta = \Theta$ which generates the Neron-Severi group $\NS(A) 
\cong \Z$, such that either $\Theta^2 = 2$ (the principal polarization case) or $\Theta^2 = (2\ell)^2$ for $\ell$ a positive integer. 
Note that $\Theta^2$ is necessarily even, so that this covers all polarizations with self-intersection a perfect square. 

Boissière proved that the automorphism groups of $X$ and $X^{[n]}$ have the same dimension \cite[Corollaire 1]{boissière_2012} 
and hence the same identity component $\Aut^0$ by considering the connected zero-dimensional quotient of $\Aut^0(X^{[n]})$ by the 
inclusion of $\Aut^0(X)$ under the map $(-)^{[n]}$. 

Let us fix an arbitrary automorphism $g:A^{[2]}\rightarrow A^{[2]}$ and decompose the induced automorphism 
on $A$ as $\ind_A^{\Sigma}(g) = t_a\circ h$ where $h(0) = 0.$ We will first calculate various intersection 
numbers on $A^{[2]}$ to show that $g$ necessarily fixes the class of the exceptional divisor $E$ in the Neron-Severi group of 
$A^{[2]}$. We will then lift this to show that $g$ fixes the class of $E$ in the Picard group and then as an actual subvariety. This 
will allow us to descend $g$ to an automorphism of $A^{(2)}$, which we can then lift to an automorphism of $A^2$ by results of 
Belmans, Oberdieck, and Rennemo. Calculations on $A^2$ will show that $g$ must be natural. 

\subsection{Fixing the exceptional divisor in Neron-Severi}\label{fix_E_NS}\hfill\\
This is the most involved portion of the proof. 

Suppose that $g^*B = ax + by + cB$ and that $g^*x = dx + ey + fB$ in the Neron-Severi group of $A^{[2]}$, where $a,b,c,d,e,f$ are 
integers. We wish to show that $g^*B = B$ in Neron-Severi, i.e. that $a = b = 0, c = 1$.

We have the diagram 
\begin{center}
\begin{tikzcd}
{A^{[2]}} \arrow[d, "\Sigma"'] \arrow[r, "g"] & {A^{[2]}} \arrow[d, "\Sigma"] \\
A \arrow[r, "t_a\circ h"]                      & A                             
\end{tikzcd}.
\end{center}
Since $A$ has Picard rank 1, $h$ is multiplication by $\pm 1$, so that $h^*\Theta = \Theta$ regardless and hence $(t_a \circ h)^*\Theta = \Theta$ 
in Neron-Severi. We conclude upon taking pullbacks that $g^*\Sigma^*\Theta = \Sigma^*\Theta$ 
and hence $g^*y = y$. 
Since $g^*$ acts as an invertible linear transformation on the lattice $\NS(A^{[2]}) \cong \Z^3$, we may write it as an 
invertible integer matrix with respect to the ordered basis $x,y,B$ as 
\[
g^* = 
\begin{bmatrix} 
    d & 0 & a \\
    e & 1 & b \\
    f & 0 & c
\end{bmatrix},
\]
which must necessarily have determinant $\pm 1$. Recall our calculated intersection numbers at the end of Subsection \ref{intersection_number_subsection}. 
Since $g^*$ must preserve intersection numbers as it is an automorphism, using our calculations in the previous section and our 
expressions for $g^*B$, we have the equality
\begin{align*}
    -16k &= \int_{A^{[2]}}y^2B^2 \\
        &= \int_{A^{[2]}}g^*(y^2B^2) \\
        &= \int_{A^{[2]}}y^2 g^*(B)^2 \\
        &= \int_{A^{[2]}} y^2(a^2x^2+b^2y^2 + c^2B^2 +2(abxy + acxB + bcyB)) \\
        &= 8k^2a^2 - 16kc^2
\end{align*}
so that $ka^2 - 2c^2 = -2$, which implies that $c^2\neq 0$.
We note next that since $B^3$ is numerically trivial, so is $g^*(B^3)$, so that 
\begin{align*}
    0 &= \int_{A^{[2]}}g^*(B)^3 B \\
        &= \int_{A^{[2]}} (a^3x^3 + b^3y^3 + c^3B^3 + 3(a^2bx^2y + a^2cx^2B + ab^2xy^2 + ac^2xB^2 + b^2cy^2B) + 6abcxyB) B \\
        &= -12ck(a + 2b)^2
\end{align*}
where we simplify the expression by recalling that a monomial $x^ay^bB^c$ can only be nonzero if $c = 0, 2$. 
Since $c, k\neq 0$, we conclude that $a+2b = 0$. We also have that
\begin{align*}
    8k^2 &= \int_{A^{[2]}} x^2y^2 \\
    &= \int_{A^{[2]}} g^*(x)^2y^2 \\
    &= \int_{A^{[2]}} (d^2x^2 + e^2y^2 + f^2B^2 + 2(de xy + df xB + ef yB))y^2 \\
    &= 8k^2d^2 - 16kf^2,
\end{align*}
so that $kd^2 - 2f^2 = k$. We finally have the equality 
\begin{align*}
    12k^2 &= \int_{A^{[2]}} x^4 \\
    &= \int_{A^{[2]}} (dx + ey + fB)^4 \\
    &= 12k^2d^4 + 48k^2 d^3e - 24k d^2f^2 + 48k^2 d^2e^2 - 96k def^2 - 96k e^2f^2
\end{align*}
where we preemptively cancelled numerically trivial terms in expanding the multinomial, so that 
\[
kd^4 + 4kd^3e - 2d^2f^2 + 4kd^2e^2 - 8def^2 - 8e^2f^2 = k. 
\]
Substituting in the $2f^2 = kd^2 - k = k(d^2 - 1)$, we find that 
\[
k = k(d+2e)^2
\]
so that $d+2e = \pm 1$, or $d = \pm 1 - 2e$.  
\newline\newline
Substituting in the relations $a = -2b$ and $d = \pm 1 - 2e$ to our matrix for the action of $g^*$ on $\NS(A^{[2]})$, we obtain
\[
g^* = 
\begin{bmatrix}
    \pm 1 - 2e & 0 & -2b \\
    e & 1 & b \\
    f & 0 & c
\end{bmatrix}\in \GL(\NS(A^{[2]})).
\]
\subsubsection{Perfect square polarization case}\hfill\\
It is now almost immediate to conclude this section in the case that $\Theta^2$ is an even perfect 
square. Suppose that $k = 2\ell^2$, then 
\begin{align*}
-2 &= ka^2 - 2c^2 \\
    &= 2\ell^2a^2 - 2c^2 \\
    &= 2(\ell a-c)(\ell a+c)
\end{align*}
so that $(c-a\ell)(c+a\ell) = 1$. Since $c-a\ell, c+a\ell$ are integers, it must be the case that they are 
either both $1$ or both $-1$. In the first case, by adding together $c+a\ell=1$ and $c-a\ell=1$ we 
find that $c=1$, whence $al=0$ so that $a=0$ and thus $b=0$ since $a=-2b$, giving 
$(a,b,c)=(0,0,1)$ as desired. In the second case where $c+a\ell = -1, c-a\ell = -1$ then by adding 
together the equations we find that $c=-1$, and as before that $a=b=0$ so that $(a,b,c)=(0,0,-1)$, 
which can't happen as $g^*B$ must be effective. 

\subsubsection{Principal polarization case}\hfill\\
We assume for the rest of this section that $\Theta^2 = 2$, and that the symmetric principal polarization $\Theta$ has the 
Riemann theta function as a section as described in Subsection \ref{global_sections_hilb2}.
By formula \ref{global_sections_formula}, we must have that $d$ and $a$ are nonnegative, as $h^0(A^{[2]}, \Theta_{[2)]}) = h^0(A^{[2]}, \O_{A^{[2]}}(B)) = 1$, but if $d < 0$ then $h^0(A^{[2]}, L) = 0$ for all line bundles $L$ in the 
numerical equivalence class $g^*x = dx + ey + fB$ and similarly for $g^*B$. Moreover, we must have $k+2\ell \ge 0$ in order for a line bundle in the numerical 
equivalence class $kx + \ell y + mB$ to have global sections (regardless of $m$, as if $m\ge 0$ then we can pushforward to $A^{(2)}$ 
to remove the $\O_{A^{[2]}}(mB)$ factor and if $m < 0$ then we only decrease global sections). We have shown that $d = \pm 1 - 2e$ via intersection numbers, so that by plugging in $k = d$ and $\ell = e$ we deduce that $k + 2\ell = (\pm 1 - 2e) + 2e = \pm 1$, so that necessarily $d = 1-2e$.

Our matrix is thus of the form  
\[
g^* = 
\begin{bmatrix}
1 - 2e & 0 & -2b \\
e & 1 & b \\
f & 0 & c
\end{bmatrix}
\]
where $d = 1-2e$ and $a = -2b$ satisfy the relations $d^2 - 2f^2 = 1$ and $a^2 - 2c^2 = -2$. We will now describe the set of 
possible $a,c,d,f$ such that the matrix 
\[
\begin{bmatrix}
d & a \\
f & c
\end{bmatrix}
\]
has determinant $\pm 1$ (the determinant of this $2\times 2$ matrix is the determinant of the matrix for $g^*$). 

Write $\mathcal{P}_k$ for the set 
\[
\mathcal{P}_k := \{(x,y)\in \Z^2 \,:\, x^2 - 2y^2 = k\}
\]
of solutions to the generalized Pell's equation $x^2-2y^2 = k$. Thus, for us $(d,f) \in \mathcal{P}_1$ and 
$(a,c)\in \mathcal{P}_{-2}$. It is easy to check that the function $(x,y) \mapsto (2y, x)$ is a map from $\mathcal{P}_1$ to 
$\mathcal{P}_{-2}$. Conversely, given integers satisfying $x^2-2y^2 = -2$, by reducing modulo 2 we see that $x$ must be even 
and hence the inverse function $(x,y) \mapsto (y, \frac{x}{2})$ is a well-defined map from $\mathcal{P}_{-2}$ to $\mathcal{P}_1$, 
so that these two sets are in bijection. 
\begin{proposition}\label{pell_sol_matrix}
If integers $d,f$ are fixed such that $d^2 - 2f^2 = 1$, then the only possible $a,c$ such that the matrix 
\[
\begin{bmatrix}
d & a \\
f & c
\end{bmatrix}
\]
has determinant 1 is the pair $a = 2f, c = d$ obtained from our map $\mathcal{P}_1 \rightarrow \mathcal{P}_{-2}$, and the only 
option for $a, c$ to obtain determinant $-1$ is the same with the right column flipped in sign.
\end{proposition}
\begin{proof}
Suppose that $d, f$ are fixed 
and that $dc - af = 1$. We assume that $f\neq 0$, as if $f = 0$ then necessarily $d=c=1$ to have determinant 1 
and thus $a = 0$ since $a^2 - 2c^2 = -2$, whence $e = b = 0$ so that we have the identity matrix as desired. 
We thus find that $a = \dfrac{dc-1}{f}$, so that from the relation $a^2 - 2c^2 = -2$ we find that 
\begin{align*}
    -2 &= a^2 - 2c^2 \\
        &= \left(\frac{dc-1}{f}\right)^2 - 2c^2 \\
\end{align*}
so 
\[
    -2f^2 = (d^2c^2 - 2dc + 1) - 2c^2f^2 
\]
or 
\[
(d^2 - 2f^2)c^2 - 2dc + (1 + 2f^2) = 0,
\]
i.e.
\[
c^2 - 2dc + (1 + 2f^2) = 0.
\]
By the quadratic formula, we find that 
\begin{align*}
    c &= \frac{2d \pm \sqrt{4d^2 - 4(1 + 2f^2)}}{2} \\
        &= d \pm \sqrt{d^2 - 2f^2 - 1} \\
        &= d.
\end{align*}
From the resulting matrix 
\[
\begin{bmatrix}
    d & a \\
    f & d
\end{bmatrix}
\]
we see that $d^2 - af = 1$ or $a = \dfrac{d^2 - 1}{f} = \dfrac{2f^2}{f} = 2f$, so that our matrix is of the form 
\[
\begin{bmatrix}
    d & 2f \\
    f & d
\end{bmatrix}
\]
in order to have determinant 1. The statement for determinant $-1$ and fixed $d,f$ follows. 
\end{proof}
We now can split up the argument according to whether the determinant is $1$ or $-1$. 
\newline
\noindent{\bf Case (I): }{\it Determinant $dc-af=1$}\hfill\\
Per proposition \ref{pell_sol_matrix}, we have that 
\[
    \begin{bmatrix}
    d & a \\
    f & c 
    \end{bmatrix} = 
    \begin{bmatrix}
    d & 2f \\
    f & d
    \end{bmatrix}
\]
where $d = 1-2e$ and $a = 2f = -2b$. Since $g^*x,\, g^*B$ are effective, we must have that 
$d \ge 0$ and $a = 2f \ge 0$ by formula \ref{global_sections_formula}. Since $d = 1-2e \ge 0$, 
we have that $e \le 0$. Suppose $e < 0$, 
so that we may apply the $k > 0, k+2\ell > 0$ case of formula \ref{global_sections_formula} to 
calculate $\dim H^0(A^{[2]}, g^*x) = \dim H^0(A^{[2]}, dx + ey + fB) = \dim H^0(A^{(2)}, dx + ey)$ 
since $f \ge 0$ and $d + 2e = (1-2e) + 2e = 1.$ We find that 
\begin{align*}
    \dim H^0(A^{[2]}, g^*\Theta_{[2]}) &= \frac{((1-2e)^2 + 1)((1-2e) + 2e)^2}{2} \\
        &= \frac{2 - 4e + 4e^2}{2} \\
        &= (1 - e)^2 + e^2
\end{align*}
since our formula applies to any line bundle in the same component $dx + ey + fB$ of the 
Neron-Severi group. If $e < 0$ then this last expression is at least 5, contradicting 
that $H^0(A^{[2]}, \Theta_{[2]})$ is 1-dimensional. Thus, $e = 0$, so that $d = 1$ and thus 
$f = 0$ since $d^2 - 2f^2 = 1$, and thus $a=b=0$ and $c = d = 1$, giving that $g^*B = B$ in 
this case, as desired. 
\newline

\noindent{\bf Case (II): }{\it Determinant $dc-af=-1$}\hfill\\
This case is more involved, and requires all of the various numerical invariants we calculated. 
Per proposition \ref{pell_sol_matrix}, we have that 
\[
    \begin{bmatrix}
    d & a \\
    f & c 
    \end{bmatrix} = 
    \begin{bmatrix}
    d & -2f \\
    f & -d
    \end{bmatrix}
\]
with $d = 1-2e \ge 0$ and $a = -2b \ge 0$ as before, so that $b = f \le 0$ and $e \le 0$. If $e = 0$ 
then from tracing out our equalities we see that $g^*B = ax + by + cB = -B$, a contradiction since 
$g^*B$ must be effective. Suppose now that $e < 0$, so that $d\ge 3$. For ease of reading we will flip the sign on $f$, so that $f$ is a positive integer and the solutions to our Pell's equation with the correct signs are $(d,-f)$. The first few solutions of the 
Pell's equation $d^2 - 2f^2 = 1$ with correct signs are 
\[
    (d,-f) = (3,-2), (17, -12), (99,-70), (577, -408), \ldots
\]
We can partition these solutions into two sets $\{(3,-2)\},$ and $\{(17,-12), (99,-70), (577,-408),\ldots\}$. 
Each set will require a different argument to rule out.
\newline\newline
\textit{Subcase (i)}: $(d,-f) = (3,-2)$\hfill\\
This case will correspond to $(a,b,c) = (4, -2, -3)$. We wish to show that no line 
bundle in the numerical equivalence class $4x - 2y - 3B$ has a section, contradicting that $g^*B$ must be effective. 
Any line bundle in this equivalence class can be written $\Theta_{[2]}^{\otimes 4}\otimes\Sigma^*\Theta^{\otimes -2}\otimes\O_{A^{[2]}}(-3B)\otimes L_0$ for some $L_0\in\Pic^0(A^{[2]})$. 
By formula \ref{global_sections_formula} case 2 we see that we must have $L_0$ a 2-torsion line bundle, else this line bundle automatically has 
no sections. If $L_0$ is 2-torsion, then by our remark in Subsection \ref{global_sections_hilb2} we may 
identify 
\[
    H^0(A^{[2]}, \Theta_{[2]}^{\otimes 4}\otimes\Sigma^*\Theta^{\otimes -2}\otimes\O_{A^{[2]}}(-3B)\otimes L_0) 
        \cong H^0(A, (\Theta^{\otimes 2}\otimes L_0)\otimes \mathcal{I}_0^{\otimes 3})^+
\]
via the difference map, where $+$ denotes invariance under $(-1)^*$, and $\mathcal{I}_0^{\otimes 3}$ denotes sections vanishing 
to order at least 3 at $0\in A$. As in Subsection \ref{global_sections_hilb2}, if we take an even section of the line bundle 
$\Theta^{\otimes 2}\otimes L_0$ the odd degree terms of its Taylor expansion about the point $0\in A$ will automatically vanish, 
so that if it vanishes to order 3 at 0 it automatically vanishes to order 4. 

The Seshadri constant of $\Theta$ on $A$ is $\frac{4}{3}$ - the 
upper bound follows by \cite[Theorem A.1(a)]{bauer1998seshadri}, and the lower bound follows by \cite{nakamaye1996seshadri} 
(note that since $A$ has Picard rank 1, it cannot be a product of two elliptic curves else we would have two independent classes 
in $\NS(A)$). It follows by definition \cite[Chapter 15.3]{BL_complex_abvar} that 
\[
    \frac{\Theta\cdot C}{\mult_0(C)} \ge \frac{4}{3}
\]
for any curve $C$ on $A$ numerically equivalent to $k\Theta$, 
where $\mult_0(C)$ is the multiplicity of $C$ at the point $0\in A$. Rearranging this and using that $\Theta\cdot (k\Theta) = 2k$, we find that 
\[
    \mult_0(C) \le \frac{3k}{2}. 
\]
In particular, for $k=2$ we see that any section of $\Theta^{\otimes 2}\otimes L_0$ vanishes to order at most 3 at 0 since 
twisting by $L_0$ does not change numerical equivalence. We needed a section vanishing to order 4, so that no line bundle in the 
numerical equivalence class $4x - 2y - 3B$ on $A^{[2]}$ has sections, ruling out this case as desired. 
\newline\newline
\textit{Subcase (ii)}: $(d,-f) = (17,-12), (99,-70), \ldots$\hfill\\
The trick in this case is to pull back to $A\times Km(A)$ under $\tilde{\mu}^*$, and then apply the switch involution $\sigma^*$ to rewrite the coefficients of the pulled-back line bundle. This will allow us to calculate that the pullback has too many global sections.

Recall that $d = 1-2e$. 
By Proposition \ref{kummer_k3_pullback}, if 
$\tilde{\mu}:A\times Km(A)\rightarrow A^{[2]}$ is the morphism defined in Subsection \ref{kummer_k3} 
then 
\[
\tilde{\mu}^*(\Theta_{[2]}^{\otimes d}\otimes\Sigma^*\Theta^{\otimes e}\otimes\O_{A^{[2]}}(-fB)) = \Theta^{\otimes 2}\boxtimes\O_{Km(A)}\left(dH - \frac{f}{2}\sum_{i=1}^{16}E_i\right)
\]
on $A\times Km(A)$, where $H$ is the ample generator on $Km(A)$ satisfying 
$H^2 = 2\Theta^2 = 4$ and the $E_i$ are the 16 nodes obtained by blowing up the singular points of $A/\pm 1$. 
Suppose $(d_0,f_0)$ is a solution to the Pell's equation $d^2-2f^2 = 1$ with $d_0, f_0 > 0$, and let $\sigma:Km(A)\rightarrow Km(A)$ be the switch involution. 
By Lemma \ref{switch_action_lemma}, 
\begin{align*}
\sigma^*\left(d_0H + \frac{f_0}{2}\sum_{i=1}^{16}E_i\right) &= d_0(3H-\sum_{i=1}^{16}E_i) + f_0(4H - \frac{3}{2}\sum_{i=1}^{16}E_i) \\
    &= (3d_0+4f_0)H - (\frac{3}{2}f_0 + d_0)\sum_{i=1}^{16}E_i.
\end{align*}
Note that if we double the coefficient of $\sum_{i=1}^{16}E_i$ and plug these coefficients into the Pell's equation $d^2-2f^2 = 1$ we find that 
\[
    (3d_0+4f_0)^2 - 2(3f_0+2d_0)^2 = d_0^2 - 2f_0^2 = 1,
\]
so that these coefficients yield a solution to the same Pell's equation with the sign flipped on $f$. 

For instance,
\begin{align*}
\sigma^*\left(3H + \sum_{i=1}^{16}E_i\right) &= 17H - 6\sum_{i=1}^{16}E_i \\
\sigma^*\left(17H + 6\sum_{i=1}^{16}E_i\right) &= 99H - 35\sum_{i=1}^{16}E_i \\
\sigma^*\left(99H + 35\sum_{i=1}^{16}E_i\right) &= 577H - 204\sum_{i=1}^{16}E_i,
\end{align*}
and in general we can rewrite the pullback to $Km(A)$ under $\tilde{\mu}^*$ of a solution $(d,-f)$ to the Pell's equation $d^2-2f^2 = 1$ with $d, f > 0$ as the pullback under $\sigma^*$ of the previous solution with the $f$ coefficient made positive. 

Fix $(d_1,-f_1)$ such that $d_1^2 - 2f_1^2 = 1$ with $d_1, f_0 > 0$ with $d_1 \ge 17$, and let $(d_0, f_0)$ be the previous solution of the Pell's equation with $d_0, f_0>0$, so that $3d_0 + 4f_0 = d_1, 3f_0 + 2d_0 = f_1$. Since $d_1 \ge 17$, we have $d_0 \ge 3, f_0 \ge 2$. 

Since $\frac{f_0}{2}\sum_{i=1}^{16}E_i$ is effective, we have the 
ideal sequence 
\[
0\rightarrow \O_{Km(A)}(d_0H)\rightarrow\O_{Km(A)}\left(d_0H+\frac{f_0}{2}\sum_{i=1}^{16}E_i\right)\rightarrow \O_{Km(A)}\left(d_0H+\frac{f_0}{2}\sum_{i=1}^{16}E_i\middle)\right|_{\frac{f_0}{2}\bigcup_{i=1}^{16}E_i}\rightarrow 0.
\]
Since $\left(d_0H+\frac{f_0}{2}\sum_{i=1}^{16}E_i\right)\cdot E_i = \frac{f_0}{2}(E_i\cdot E_i) = -f_0$, we see that the sheaf $\O_{Km(A)}\left(d_0H+\frac{f_0}{2}\sum_{i=1}^{16}E_i\right)$ 
restricts to the line bundle $\O_{\P^1}(-f_0)$ on each of the 16 disjoint rational curves $E_i$. Thus, this last sheaf in our exact sequence has no global sections, so 
\[
H^0(Km(A), \O_{Km(A)}\left(d_0H+\frac{f_0}{2}\sum_{i=1}^{16}E_i\right))\cong H^0(Km(A), \O_{Km(A)}(d_0H)) 
\]
by taking the exact sequence in cohomology. 

By Riemann-Roch for K3 surfaces, we find that 
\[
    \chi(Km(A), d_0H) = \frac{1}{2}(d_0H)^2 + 2 = \frac{1}{2}(d_0^2\cdot 4) + 2 = 2(d_0^2 + 1).
\]
Since the divisor $d_0H$ has positive self intersection and zero intersection with all 
the $E_i$, it is big and nef and hence the higher cohomology of $\O_{Km(A)}(d_0H)$ vanishes by 
Kawamata-Viehweg. We conclude that $\dim H^0(Km(A), \O_{Km(A)}(d_0H+\frac{f_0}{2}\sum_{i=1}^{16}E_i)) = \dim H^0(Km(A), \O_{Km(A)}(d_0H)) = 2(d_0^2 + 1)$, and hence (letting $d_1 = 1-2e_1$).
\begin{align*}
\dim H^0(A\times Km(A)&, \tilde{\mu}^*(\Theta_{[2]}^{\otimes d_1}\otimes\Sigma^*\Theta^{\otimes e_1}\otimes\O_{A^{[2]}}(-f_1B))) \\
    &= \dim H^0(A\times Km(A), \Theta^{\otimes 2}\boxtimes\O_{Km(A)}\left(d_1H - \frac{f_1}{2}\sum_{i=1}^{16}E_i\right)) \\
    &= \dim H^0(A, \Theta^{\otimes 2})\otimes H^0(Km(A), \sigma^*\O_{Km(A)}\left(d_0H + \frac{f_0}{2}\sum_{i=1}^{16}E_i\right)) \\
    &= 4\cdot 2(d_0^2 + 1) \\
    &= 8(d_0^2 + 1)
\end{align*}
since $\sigma$ is an automorphism and hence preserves global sections. 
To relate this back to the original line bundle on $A^{[2]}$, 
recall from Subsection \ref{kummer_k3} that the diagram 
\begin{center}
\begin{tikzcd}
A\times Km(A) \arrow[r, "\tilde{\mu}"] \arrow[d, "\pi_1"'] & {A^{[2]}} \arrow[d, "\Sigma"] \\
A \arrow[r, "\cdot 2"]                                     & A                            
\end{tikzcd}
\end{center}
is Cartesian, so that the pushforward
\begin{align*}
    \tilde{\mu}_*\O_{A\times Km(A)} &\cong \Sigma^*(\cdot 2)_*\O_A \\
    &\cong \bigoplus_{L\in\Pic^0(A^{[2]})[2]}L
\end{align*}
is just the direct sum of the 16 2-torsion line bundles on $A^{[2]}$ by \cite[p.72]{mumford2008abelian} as in the proof of Lemma \ref{Fkl_cohomology}. 

We conclude by push-pull that  
\[
\sum_{L\in\Pic^0(A^{[2]})[2]}\dim H^0(A^{[2]}, \Theta_{[2]}^{\otimes d}\otimes\Sigma^*\Theta^{\otimes e}\otimes\O_{A^{[2]}}(-fB)\otimes L) = 8(d_0^2 + 1) \ge 80
\]
since $d_0 \ge 3$.
Since there are 16 line bundles appearing in this sum, there must be a term with $h^0$ at least 5. Since the map $L\mapsto L_{[2]}$ induces an isomorphism $\Pic^0(A)\rightarrow\Pic^0(A^{[2]})$, 
every line bundle in the numerical equivalence class of $\Theta_{[2]}$ is of the form $(\Theta\otimes L_0)_{[2]}$. Every such line bundle has exactly one section since 
\[
    \dim H^0(A\times A, (\Theta\otimes L_0)\boxtimes (\Theta\otimes L_0)) = 1.
\]
We have a contradiction then - every line bundle in the equivalence class of $x = \Theta_{[2]}$ has one section, but there are line bundles with at least 5 sections in the equivalence class of $g^*x = d_1x + e_1y - f_1B$.
We conclude that we can rule out all cases for $(d,-f)$ starting at $(17,-12)$. 

We note that if we remove the negative $B$ coefficient and apply formula \ref{global_sections_formula} to the numerical 
equivalence class $17x - 8y$ we obtain a $\frac{17^2+1}{2} = 145$-dimensional space of sections, so it is plausible for there to be 
line bundles in the class of the first solution $17x-8y-12B$ with 5 sections. 
\newline\newline
We conclude that we can rule out the determinant $dc-af = -1$ case, and hence that the only possibility for the action of $g^*$ on $\NS(A^{[2]})$ 
is as the identity. 
\subsection{Fixing the exceptional divisor in Picard and as a subvariety}\label{fix_E_Pic_subvar}\hfill\\
That $g^*B = B$ in the Neron-Severi group means that the difference $D = g^*B - B$ lies in $\Pic^0(A^{[2]})$. Since $\Sigma^*:\Pic^0(A)\rightarrow \Pic^0(A^{[2]})$ is surjective we may 
write $g^*B = B + \Sigma^*D'$ for some $D' \in \Pic^0(A)$. 

Since $\pi_*\O_{A^{[2]}}(B) \cong \O_{A^{(2)}}$ (Section \ref{hilb_prelim}) we have that 
\[
    h^0(A^{[2]}, \O_{A^{[2]}}(B)) = h^0(A^{(2)}, \O_{A^{(2)}}) = 1.
\]
As pushforward preserves dimension of 
global sections, we find that $1 = h^0(A, \Sigma_*g^*\O_{A^{[2]}}(B)) = h^0(A, \Sigma_*\O_{A^{[2]}}(B + \Sigma^*D')) = h^0(A, \O_A(D'))$, since the map $\Sigma$ factors through the Hilbert-Chow morphism.
Since $D'$ lies in $\Pic^0(A)$ and has a global section it must be trivial, and hence we conclude that $g^*B = B$ in Picard. 

Since $g^*B = B$ in $\Pic(A^{[2]})$, if we write $E$ for the actual subvariety of nonreduced subschemes in $A^{[2]}$ then $g(E)$ 
is another divisor within the linear equivalence class of $E$ since as a divisor class $E = 2B$. Since we know that $h^0(A^{[2]}, \O_{A^{[2]}}(E)) = 1$ (via pushforward under $\pi_*$ as for $B$), we conclude 
that $g(E) = E$ and hence $g$ restricts to an automorphism of $E \subset A^{[2]}$. 

\subsection{Lifting to the Cartesian product and proving naturality}\label{A2_lift_naturality}\hfill\\
The restriction of the Hilbert-Chow morphism $\pi: E\subset A^{[2]}\rightarrow \Delta\subset A^{(2)}$ has $\P^1$-fibers (see the proof of \cite[Corollary 3]{iarrobino1972punctual}). 
As $\Delta\cong A$ and there are no nontrivial maps from the projective space $\P^1$ to the abelian variety $A$, the restriction of 
$g$ to $E$ must necessary send $\P^1$-fibers of $\pi$ to $\P^1$-fibers, as the image of a $\P^1$-fiber under $\pi\circ g$ must be a 
single point. By \cite[Proposition 7]{BOR}, $g$ descends to an automorphism $\bar{g}$ of $A^{(2)}$. 
By \cite[Proposition 9, 12]{BOR}, this automorphism lifts to an automorphism $\hat{g}$ of the Cartesian product $A^2$ such that we have the natural commutative diagram 
\begin{center}
\begin{tikzcd}
A^2 \arrow[d, "p"'] \arrow[r, "\hat{g}"] & A^2 \arrow[d, "p"] \\
A^{(2)} \arrow[r, "\bar{g}"]             & A^{(2)}           
\end{tikzcd}.
\end{center}

Since $\hat{g}$ is an automorphism of the complex torus $A^2$, we may write 
\[
\hat{g} = t_{(x_0, y_0)}\circ \hat{h}
\]
for some points $x_0, y_0\in A$ where $\hat{h}(0,0) = 0$ and $\hat{h}$ is invertible. Since $\hat{h}$ is a group endomorphism of the product $A\times A$ we may 
write 
\[
    \hat{h} = 
    \begin{bmatrix}
    \hat{h}_{11} & \hat{h}_{12} \\
    \hat{h}_{21} & \hat{h}_{22} 
    \end{bmatrix}
\]
for some endomorphisms $\hat{h}_{ij}:A\rightarrow A$ per our discussion in Section \ref{counterexample_section}. Since $\End(A)\cong\Z$, $\hat{h}$ is an invertible $2\times 2$ integer matrix with determinant $\pm 1$.  

Since $\hat{g}$ descends to the automorphism $\bar{g}$ under the covering $p:A^2\rightarrow A^{(2)}$, we must have the 
$\mathfrak{S}_2$-equivariance condition 
\[
    (\hat{h}_{11}(x) + \hat{h}_{12}(y) + x_0, \hat{h}_{21}(x) + \hat{h}_{22}(y) + y_0) =
    (\hat{h}_{22}(x) + \hat{h}_{21}(y) + y_0, \hat{h}_{12}(x) + \hat{h}_{11}(y) + x_0)
\]
for all $x,y\in A$. Setting $x = y = 0$ gives $x_0 = y_0$, and then setting $y = 0$ and using that $x_0 = y_0$ gives that 
$\hat{h}_{11} = \hat{h}_{22}, \hat{h}_{21} = \hat{h}_{12}$. Let us write $\hat{h}_1$ for $\hat{h}_{11} = \hat{h}_{22}$ and 
$\hat{h}_2$ for $\hat{h}_{12} = \hat{h}_{21}$, so that our matrix can be rewritten as 
\[
    \hat{h} = 
    \begin{bmatrix}
    \hat{h}_1 & \hat{h}_2 \\
    \hat{h}_2 & \hat{h}_1
    \end{bmatrix}.
\]
Since this matrix has determinant $\pm 1$ and the $\hat{h}_i$ are integers, we find that either $\hat{h}_1=\pm 1, \hat{h}_2 = 0$, or 
$\hat{h}_1 = 0, \hat{h}_2 = \pm 1$ since we may factor the Diophantine equation $\hat{h}_1^2 - \hat{h}_2^2 = \pm 1$ as 
$(\hat{h}_1-\hat{h}_2)(\hat{h}_1+\hat{h}_2) = \pm 1$. 

Thus, $\hat{g} = t_{(x_0, x_0)}\circ \hat{h}$ where we have the choices 
\[
\hat{h} = 
\begin{bmatrix}
\pm 1 & 0 \\
0 & \pm 1
\end{bmatrix}, 
\begin{bmatrix}
0 & \pm 1\\
\pm 1 & 0 
\end{bmatrix}
\]
(in all cases the nonzero entries are equal). In the first case, we have 
\[
    \hat{g}(x,y) = (x_0 \pm x, x_0 \pm y), 
\]
so that the corresponding automorphism $\bar{g}$ on $A^{(2)}$ is the natural automorphism $(t_{x_0}\circ \pm 1)^{[2]}$. In the second case, 
\[
    \hat{g}(x,y) = (x_0 \pm y, x_0 \pm x), 
\]
which again yields $\bar{g} = (t_{x_0}\circ \pm 1)^{[2]}$. Since $\bar{g}$ commutes with $g:A^{[2]}\rightarrow A^{[2]}$ under the Hilbert-Chow morphism and $A^{(2)}-\Delta \cong A^{[2]}-E$ is a dense open set in $A^{[2]}$, $g$ must be the natural automorphism $(t_{x_0}\circ \pm 1)^{[2]}$. 
\newline\newline
We conclude that all automorphisms of $A^{[2]}$ are natural, as desired. 

\section{Open directions}\label{open_directions}
As we note in Question \ref{open_question}, it remains open whether all automorphisms of $A^{[2]}$ are natural for arbitrary polarization types. The only issue with extending our proof is Section \ref{fix_E_NS} where we use numerical invariants to show that the exceptional divisor is fixed by an automorphism of $A^{[2]}$ in the Neron-Severi group. Our calculation of the intersection numbers in $A^{[2]}$ in Subsection \ref{intersection_number_subsection} allowed for any polarization type. Moreover, our formulas for the dimensions of global sections of line bundles in Subsection \ref{global_sections_calcs} can be extended to non-principal polarizations just by inserting the appropriate intersection numbers, and the work of Subsection \ref{global_sections_hilb2} in estimating theta functions that vanish at 0 to prescribed multiplicity can also be carried out for arbitrary polarizations. 

The main work of showing that the exceptional divisor is fixed in Neron-Severi came down to the determinant $-1$ case, where we ruled out the various nontrivial solutions of the corresponding Pell's equation using different techniques. Our investigation suggests that the arguments for the first case still work for arbitrary polarizations. As a reminder, the arguments for this case followed from the Seshadri constants of abelian surfaces and estimates of theta functions with prescribed orders of vanishing. However, the second case crucially relied on using a particular automorphism of Jacobian Kummer surfaces to rewrite a divisor class as a sum of effective classes with positive coefficients. One would need to understand the automorphisms of Kummer surfaces with arbitrary polarizations or use an alternative argument to handle the second case. 

Our investigation shows that for arbitrary polarizations one may use a different argument using the estimates from Subsection \ref{global_sections_hilb2} to rule out all nontrivial solutions of the corresponding Pell's equation starting from the third solution, but that this doesn't work for the second solution. 

In any case, one may ask whether automorphisms of $A^{[n]}$ are natural for abelian surfaces $A$ of Picard rank 1 for $n\ge 3$. We believe that it would be difficult to simply generalize our argument here for higher $n$, and that perhaps another idea is necessary. Calculating top intersection numbers would grow more painful as $n$ increases (though is possible in principle by using the recursions in \cite{EGL}), and the corresponding Diophantine equations in the Neron-Severi group would grow more complex. Even if one could calculate out the Diophantine equations for higher $n$, it is unclear whether the analogous numerical invariants to those calculated in Section \ref{numerical_calc_section} are both tractable and sufficient to rule out nontrivial solutions of these Diophantine equations. 

\section{Appendix}
\subsection{Description of $S^{[n]}$ as a blowup}\label{Sn_blowup_description}\hfill\\
Let $S$ be a smooth complex projective surface, though much of this discussion holds in greater generality. 
Instead of working intrinsically with $S^{[n]}$, one may wish to describe $S^{[n]}$ in terms of the simpler Cartesian or symmetric 
products $S^n, S^{(n)}$. One may attempt to either blow up first and then identify points under the action of $\S_n$, 
or one may choose to first identify points and then blow up. The former approach is commonly taken; in \cite{nakajima1999lectures}
Nakajima describes the Hilbert scheme of two points $S^{[2]}$ as the quotient of the blowup $Bl_{\Delta}S^2$ of the diagonal by the 
induced $\S_2$ action: 
\[
    S^{[2]} \cong Bl_{\Delta}S^2 / \S_2.
\]
For higher $n$, Tikhomirov \cite{tikhomirov1994standard} showed how to construct an iterated blowup diagram 
\begin{center}
\begin{tikzcd}
          & \widetilde{S^n} \arrow[ld, "\pi"'] \arrow[d, "\sigma"] \\
{S^{[n]}} & S^n                                               
\end{tikzcd}
\end{center}
relating the Cartesian product and the Hilbert scheme of points. We have simplified some notation as we do not fully describe the 
details of the construction here. The space $\widetilde{S^n}$ is not one single blowup but rather an iterated sequence of blowups 
of pullbacks of incidence schemes. While the map $\pi:\widetilde{S^n}\rightarrow S^{[n]}$ is $\S_n$-equivariant and generically 
finite of degree $n!$, for $n\ge 3$ it has some positive dimensional fibers, so that this construction does not realize 
$S^{[n]}$ as the geometric quotient of $\widetilde{S^n}$ by the finite group $\S_n$. Hence, we cannot immediately descend an 
$\S_n$-equivariant automorphism of $\widetilde{S^n}$ down to an automorphism of $S^{[n]}$ as we might want (were $S^{[n]}$ a quotient 
of $\widetilde{S^n}$ by the finite group $\S_n$, then we could descend by the universal property of such a quotient). To say more 
about why this map has positive dimensional fibers, we recall that Tikhomirov uses the morphism $\varphi_n$ in the diagram 
\begin{center}
\begin{tikzcd}
          & {Bl_{\mathcal{Z}_{n-1}}(S\times S^{[n-1]})} \arrow[d] \arrow[ld, "\varphi_n"'] \\
{S^{[n]}} & {S\times S^{[n-1]}}                                                           
\end{tikzcd}
\end{center}
first defined by Ellingsrud (see \cite{ellingsrud1998intersection} for a good exposition) which resolves the rational map 
$S\times S^{[n-1]}\dashrightarrow S^{[n]}$ of adding a point of $S$ to a subscheme of length $n-1$ (which is only defined when the 
added point does not intersect the support of the subscheme). In the diagram, $\mathcal{Z}_{n-1}$ is the universal subscheme 
in $S\times S^{[n-1]}$. As described in \cite{ellingsrud1998intersection}, this blowup 
$Bl_{\mathcal{Z}_{n-1}}(S\times S^{[n-1]})$ may be identified with the nested Hilbert scheme $S^{[n-1, n]}$, such that $\varphi_n$ is 
then just the projection from $S^{[n-1,n]}$ to $S^{[n]}$. For $n\ge 3$ this projection can have positive dimensional fibers over 
non-reduced points. As an example for $S=\A^2$, consider the fiber of $\phi_3:(\A^2)^{[2,3]}\rightarrow (\A^2)^{[3]}$ over the 
degree 3 point defined by the ideal $I = (x^2, xy, y^2)$. Adding any nonzero linear polynomial $ax+by$ will give a colength 2 
ideal containing $I$, so that the fiber is positive dimensional here. 

If we could instead realize $S^{[n]}$ directly as a blowup of $S^{(n)}$, then from the universal property of blowups \cite[Corollary II.7.15]{hartshorne1977algebraic} we could lift automorphisms of $S^{(n)}$ fixing the ideal sheaf being blown 
up to automorphisms of $S^{[n]}$. Such descriptions exist in varying levels of generality. Before stating these descriptions, 
we illustrate the situation for $S=\A^2$ and $n=2$. 

Set $S = \Spec(\C[x,y])$ and set $R = \C[x_1,y_1,x_2,y_2]$, so that $S^2 = \Spec(R)$ and $S^{(2)} = \Spec(R^{\S_2})$ where the $\S_2$ 
action swaps $x_1$ with $x_2$ and $y_1$ with $y_2$. Set $p:S^2\rightarrow S^{(2)}$ to be the standard projection map, so that the dual 
map $p^{\#}:R^{\S_2}\rightarrow R$ is just the inclusion map. 
Given an ideal $I\subset R^{\S_2}$, to compare the blowup of $S^{(2)}$ at $I$ 
with a blowup of $S^2$, we need to understand the inverse image ideal sheaf $p^{-1}I\cdot \O_{S^2}$, which in the affine case 
just corresponds to the ideal of $R$ generated by the image $p^{\#}(I)$. While the ideal of the diagonal in $S^2$ is generated 
as 
\[
    I_{\Delta, S^2} = (x_1-x_2, y_1-y_2),
\]
these generators are not elements in $R^{\S_2}$ as they are alternating rather than symmetric polynomials for this $\S_2$-action. 
However, the square of this ideal 
\[
    I_{\Delta, S^{(2)}} = ((x_1-x_2)^2, (x_1-x_2)(y_1-y_2), (y_1-y_2)^2)
\]
does define an ideal in $R^{\S_2}$ whose vanishing yields the diagonal. Since taking the tensor power of a coherent sheaf of ideals 
does not affect the corresponding blowup \cite[Exercise II.7.11(a)]{hartshorne1977algebraic}, we obtain a blowup diagram 
relating $Bl_{I_{\Delta, S^{(2)}}}S^{(2)}$ and $Bl_{I_{\Delta, S^2}}S^2$ which one can use to prove that 
$Bl_{I_{\Delta, S^{(2)}}}S^{(2)}\cong Bl_{I_{\Delta, S^2}}S^2/\S_2 \cong S^{[2]}$. 

Haiman generalized this to all $n$, proving the following result (\cite[Proposition 2.6]{haiman1998t} and subsequent discussion): 
\begin{proposition}
Let $k$ be an algebraically closed field of characteristic zero or greater than $n$. 
Set $k[X,Y] = k[x_1,y_1,...,x_n,y_n]$ with the $\S_n$-action $\sigma(x_i) = x_{\sigma(i)}, \sigma(y_j) = y_{\sigma(j)}$. Let 
$A$ be the subset of polynomials $f$ such that $\sigma\cdot f = (-1)^{|\sigma|}f$, where $|\sigma| = 0,1$ is the sign of the 
permutation, and set $A^i$ to be the vector space spanned by products $f_1\cdot...\cdot f_i$, with the $f_j\in A$. $A^2$ is thus 
a $k[X,Y]^{\S_n}$ ideal, and the Hilbert scheme of points $(\A_k^2)^{[n]}$ is isomorphic to the blowup 
$\Proj(k[X,Y]^{\S_n}\oplus A^2\oplus A^4\oplus ...)$ of the symmetric product $(\A_k^2)^{(n)}$ at the ideal sheaf $A^2$. 
\end{proposition}

This description has been significantly generalized by Ekedahl and Skjelnes \cite{ekedahl2014recovering} and similarly by Rydh and 
Skjelnes \cite{rydh2010intrinsic}. For $\dim X \ge 3$ the Hilbert scheme $X^{[n]}$ is no longer necessarily the closure of the 
locus of reduced subschemes as is the case for dimensions 1 and 2 (we call this closure the \textit{smoothable locus}). 
It is the smoothable locus which they show is the blowup 
of the symmetric product of $X$ at a particular ideal sheaf for an arbitrary scheme $X$ (in fact, they show this even more generally 
for $X\rightarrow S$ a family of algebraic spaces, but we do not need that level of generality). We will use the formalism of 
Rydh and Skjelnes. 

Given $R$ an $A$-algebra, we write $\T_A^n R$ for the $n$-fold tensor product of $R$ over $A$ with the obvious $\S_n$-action 
and $\TS_A^n R$ for the subalgebra of symmetric tensors in $\T_A^n R$. Given $x_1,...,x_n\in R$, we define the alternator map 
\[
    \alpha(x_1,...,x_n) := \sum_{\sigma\in \S_n}(-1)^{|\sigma|}x_{\sigma(1)}\otimes ... \otimes x_{\sigma(n)}
\]
which induces a $\TS_A^n R$-linear map from $\T_A^n R$ to $\T_A^n R$. The product map 
\[
    \alpha\times\alpha: \T_A^n R \otimes_{\TS_A^n R} \T_A^n R \rightarrow \T_A^n R
\]
produces $\S_n$-invariant tensors, as one can check that $\alpha(\sigma(x)) = (-1)^{|\sigma|}\alpha(x)$, so that any element 
$\alpha(x)\alpha(y)$ is invariant. The image of $\alpha\times\alpha$ defines an ideal in the ring $\TS_A^n R$, which we call the 
\textit{canonical ideal} $I_R$. By \cite[Proposition 6.8]{rydh2010intrinsic} these canonical ideals on product affine schemes can be 
glued together to yield an ideal sheaf $\mathcal{I}$ on $X^{(n)}$ for a scheme $X$, and by Corollary 6.18 the blowup of $X^{(n)}$ 
in $\mathcal{I}$ is isomorphic to the smoothable locus of $X^{[n]}$. 
We can then prove the key Proposition \ref{induced_aut_lemma}, whose statement we recall: 
\newline\newline
\textbf{Proposition \ref{induced_aut_lemma}.}
\textit{If $X$ is either an affine scheme or a projective scheme over an infinite field $k$,
then any $\S_n$-equivariant automorphism $f:X^n\rightarrow X^n$ induces an automorphism on the smoothable locus of $X^{[n]}$.}
\newline\newline
{\it Proof.} Since $f$ is $\S_n$-equivariant, it descends to an automorphism $\bar{f}$ on $X^{(n)} := X^n / \S_n$ by the universal property of 
quotients by finite groups. On affine opens the associated ring homomorphism $\bar{f}^{\#}$ is just the restriction of $f^{\#}$ to 
$\S_n$-invariant functions, which works by the equivariance of $f$. 
Since $\bar{f}$ is an automorphism, it is flat, so the inverse image ideal sheaf 
$\bar{f}^{-1}(\mathcal{I})\cdot \O_{X^{(n)}}$ is the same as the pullback $\bar{f}^*\mathcal{I}$. By \cite[Corollary II.7.15]{hartshorne1977algebraic}, it suffices to show that $\bar{f}^*\mathcal{I} = \mathcal{I}$ to get an automorphism on the blowup of $X^{(n)}$ 
along $\mathcal{I}$, which as discussed is the smoothable locus of $X^{[n]}$.  

We wish to find an affine cover $\{U_i = \Spec(R_i)\}$ of $X$ such that 
$X^n$ is covered by the product affine schemes $U_i^n = \Spec(\T^n R)$, and similarly that $X^{(n)}$ is covered by affine schemes 
$U_i^{(n)} = \Spec(\TS^n R)$. If $X$ is affine this is immediate as we may just take the single product affine $X$ itself. 
If $X$ is projective, embed $X$ as a closed subscheme of some projective space $\P_k^n$. Since $k$ is infinite, for 
any finite collection of points in $X$ we can find a hyperplane in $\P_k^n$ which does not intersect any of the points. Hence, 
since the complement in $X$ of a hyperplane section is affine, we find covers of $X^n, X^{(n)}$ by product affines. 

Note that since $X^{(n)}$ has the quotient topology of $X^n$ under $\S_n$, open sets in $X^{(n)}$ correspond to $\S_n$-invariant 
open sets upstairs in $X^n$. We first wish to show that pullback under $f$ does in fact send local sections of $\mathcal{I}$ to 
local sections of $\mathcal{I}$ - i.e. that $\bar{f}^{\#}$ defines a module homomorphism $\mathcal{I}(U)\rightarrow \mathcal{I}(\bar{f}^{-1}(U))$ 
for each open set $U\subset X^{(n)}$. It is easy enough to show this for the product affines $\{U_i^{(n)}\}$. 
Let $s\in\mathcal{I}(U^{(n)})$ for some affine $U\subset X$ be given, so that $s$ is of the form 
$s = \sum_j r_j\alpha(x_j)\alpha(y_j)$
for elements $r_j\in \O_{X^{(n)}}(U^{(n)}), x_j, y_j\in \O_{X^n}(U^n)$ by definition of $\mathcal{I}$ (it is straightforward to 
describe elements of $\mathcal{I}$ over product affines, but not necessarily for general open sets). 
Since $\bar{f}^{-1}(U^{(n)})$ is open in $X^{(n)}$, its lift 
upstairs $f^{-1}(U^n)$ is $\S_n$-invariant, and hence we may speak of the action $\sigma^{\#}$ on $\O_{X^n}(f^{-1}(U^n))$ given by a 
permutation $\sigma$. For $x\in \O_{X^n}(U^n)$, we directly calculate that 
\begingroup
\allowdisplaybreaks
\begin{align*}
    f^{\#}(\alpha(x)) &= f^{\#}\left(\sum_{\sigma} (-1)^{|\sigma|} \sigma^{\#}(x)\right) \\
        &= \sum_{\sigma}(-1)^{|\sigma|} f^{\#}\sigma^{\#}(x) \\
        &= \sum_{\sigma}(-1)^{|\sigma|} \sigma^{\#}f^{\#}(x) \\
        &= \alpha(f^{\#}(x))
\end{align*}
where we use that $f$ is $\S_n$-equivariant to commute $\sigma^{\#}, f^{\#}$. Thus, for $s$ as before we directly calculate that 
\begin{align*}
    \bar{f}^{\#}(s) &= \bar{f}^{\#}\left(\sum_j r_j\alpha(x_j)\alpha(y_j)\right) \\
        &= \sum_j \bar{f}^{\#}(r_j)f^{\#}(\alpha(x_j))f^{\#}(\alpha(y_j)) \\
        &= \sum_j \bar{f}^{\#}(r_j)\alpha(f^{\#}(x_j))\alpha(f^{\#}(y_j))
\end{align*}
\endgroup
where we swap between $\bar{f}$ and $f$ depending on whether we are pulling back invariant sections from $X^{(n)}$ or sections coming 
upstairs from $X^n$. If $X$ is affine then this formula immediately shows that $\bar{f}^{\#}$ is an automorphism of $\mathcal{I}$. 
We thus suppose for the rest of the proof that $X$ is projective over an infinite field $k$. 

We first check that this last term 
\[
\sum_j \bar{f}^{\#}(r_j)\alpha(f^{\#}(x_j))\alpha(f^{\#}(y_j))
\]is indeed an element of $\mathcal{I}(\bar{f}^{-1}(U))$. The open 
set $\bar{f}^{-1}(U)$ is covered by the intersections $\bar{f}^{-1}(U)\cap U_i^{(n)}$ with product affines, and the restriction 
map $\O_{X^{(n)}}(U_i^{(n)}) \twoheadrightarrow \O_{X^{(n)}}(\bar{f}^{-1}(U)\cap U_i^{(n)})$ is a surjection since it is dual to 
an injective map of affine schemes (the intersection of two affine opens being affine). This surjection restricts to a surjection of 
the corresponding ideal sheaves, so that every element of $\mathcal{I}(\bar{f}^{-1}(U)\cap U_i^{(n)})$ can be written in the 
standard form $\sum_j r_j \alpha(x_j)\alpha(y_j)$, since restriction maps commute with $\alpha$. Thus, 
$\bar{f}^{\#}(s) = \sum_j \bar{f}^{\#}(r_j)\alpha(f^{\#}(x_j))\alpha(f^{\#}(y_j))$ restricts to an element of 
$\mathcal{I}(\bar{f}^{-1}(U)\cap U_i^{(n)})$ for each $i$, and these restrictions necessarily glue since they arise from a section 
$\bar{f}^{\#}(s)$ on $\bar{f}^{-1}(U)$, so $\bar{f}^{\#}(s)$ lives in $\mathcal{I}(\bar{f}^{-1}(U))$ as desired. 

As we have just shown that $\bar{f}^{\#}$ sends sections of $\mathcal{I}(U_i^{(n)})$ to $\mathcal{I}(f^{-1}(U_i^{(n)}))$ for product 
affines $U_i^{(n)}$, and these product affines cover $X^{(n)}$, it follows that $\bar{f}^{\#}$ does the same for arbitrary open sets 
$U$. To elaborate, let $s$ be a section of $\mathcal{I}(U)\subset \O_{X^{(n)}}(U)$ for $U\subset X^{(n)}$ open. Similar to our 
preceding argument, the intersections $U_i^{(n)}\cap U$ cover $U$, and each element of $\mathcal{I}(U_i^{(n)}\cap U)$ is the 
restriction of an element of $\mathcal{I}(U_i^{(n)})$. Write $s_i$ for the restriction of $s$ to $\mathcal{I}(U_i^{(n)}\cap U)$, and 
let $\tilde{s}_i$ be an element of $\mathcal{I}(U_i^{(n)})$ which restricts to $s_i$. We have just shown that 
$\bar{f}^{\#}(\tilde{s}_i)$ is an element of $\mathcal{I}(f^{-1}(U_i^{(n)}))$, so that $\bar{f}^{\#}(s_i)$ is an element of 
$\mathcal{I}(\bar{f}^{-1}(U_i^{(n)})\cap \bar{f}^{-1}(U))$ since restriction commutes with $\bar{f}^{\#}$. Since the 
$\bar{f}^{\#}(s_i)$ glue to $\bar{f}^{\#}(s)$, we see that $\bar{f}^{\#}(s)$ lives in $\mathcal{I}(f^{-1}(U))$. 

Thus, we have that pullback under $\bar{f}$ maps the ideal sheaf $\mathcal{I}$ to itself. Since $\bar{f}$ is an automorphism this 
map from $\mathcal{I}$ to itself must necessarily be injective - we have only to show that it is surjective as well. This follows by Nakayama's lemma since $\mathcal{I}$ is coherent and $X$ is projective (hence proper), see \cite[Chapter 4 Exercise 1]{friedman2012algebraic}.

We conclude that $f^*$ induces an automorphism of $\mathcal{I}$, and hence there exists an isomorphism between the blowup of 
$X^{(n)}$ at $\mathcal{I}$ and the blowup at $\bar{f}^*\mathcal{I} = \mathcal{I}$ commuting with $\bar{f}$, i.e. an automorphism 
of the smoothable locus of $X^{[n]}$. \qed
\newline\par
We do not claim that our hypotheses are the most general ones 
possible - we imagine it is possible to prove the same result for more general $X$ using the full generality of the constructions 
in \cite{ekedahl2014recovering}, \cite{rydh2010intrinsic}.


\begin{thebibliography}{1}

\bibitem [BHV]{barth2015compact}
W. Barth, K. Hulek, C. Peters, A. Van de Ven, {\it Compact complex surfaces}, Springer-Verlag Berlin Heidelberg (2004)

\bibitem [BSz]{bauer1998seshadri}
T. Bauer, T. Szemberg, Appendix to {\it Seshadri constants and periods of polarized abelian varieties}, Math Ann. 312 (1998), 607 -- 623

\bibitem [Bea1]{Beauville1983VaritsKD}
A. Beauville, {\it Vari{\'e}t{\'e}s K{\"a}hleriennes dont la premi{\`e}re classe de Chern est nulle}, Journal of Differential Geometry 18 (1983), 755 -- 782

\bibitem [Bea2]{Beauville_unnatK3}
A. Beauville, {\it Some remarks on K\"ahler manifolds with $\mathrm{c}_1 = 0$}, Progress in Mathematics 39 (1983), 1 -- 26

\bibitem [Bea3]{beauville2013theta}
A. Beauville, {\it Theta functions, old and new}, Open Problems and Surveys of Contemporary Mathematics 6 (2013), 99 -- 131

\bibitem [Bel]{hodge-diamond-cutter}
P. Belmans, {\it Hodge diamond cutter}, \url{https://github.com/pbelmans/hodge-diamond-cutter}

\bibitem [BD]{biswas2015automorphisms}
I. Biswas, A. Dhillon, {\it Automorphisms of the Quot schemes associated to compact Riemann surfaces}, International Mathematics Research Notices 6 (2015) 1445 -- 1460

\bibitem [BL]{BL_complex_abvar}
C. Birkenhake, H. Lange, {\it Complex abelian varieties}, Springer Berlin, Heidelberg (2004)

\bibitem [BG]{biswas_gomez}
I. Biswas, T. G\'{o}mez, {\it Automorphisms of a symmetric product of a curve (with an appendix by Najmuddin Fakhruddin)}, Documenta Mathematica 22 (2017), 1181-1192

\bibitem [Bo]{boissière_2012}
S. Boissi{\`e}re, {\it Automorphismes naturels de l'espace de Douady de points sur une surface}, Canadian Journal of Mathematics 1 (2012), 3 -- 23

\bibitem [BCNS]{boissiere2016automorphism}
S. Boissi{\`e}re, A. Cattaneo, M. Nieper-Wi{\ss}kirchen, A. Sarti, {\it The automorphism group of the Hilbert scheme of two points on a generic projective K3 surface}, K3 surfaces and their moduli (2016), 1 -- 15

\bibitem [BNS]{boissiere2011higher}
S. Boissi{\`e}re, M. Nieper-Wi{\ss}kirchen, A. Sarti, {\it Higher dimensional Enriques varieties and automorphisms of generalized Kummer varieties}, Journal de math{\'e}matiques pures et appliqu{\'e}es 95.5 (2011), 553 -- 563 

\bibitem [BSa]{boissiere_sarti}
S. Boissi{\`e}re, A. Sarti, {\it A note on automorphisms and birational transformations of holomorphic symplectic manifolds}, 
Proceedings of the American Mathematical Society 140.12 (2012), 4053 -- 4062

\bibitem [BOR]{BOR}
P. Belmans, G. Oberdieck, J. Rennemo, {\it Automorphisms of Hilbert schemes of points on surfaces}, Transactions of the American Mathematical Society 373.9 (2020), 6139 -- 6156

\bibitem [CC]{catanese_ciliberto_symmell}
F. Catanese, C. Ciliberto, {\it Symmetric products of elliptic curves and surfaces of general type with $p_g = q = 1$}, 
J. Algebraic Geom. 2 (1993), 389 -- 411

\bibitem [Ca]{cattaneo2019automorphisms}
A. Cattaneo, {\it Automorphisms of Hilbert schemes of points on a generic projective K3 surface}, Mathematische Nachrichten 292.10 (2019), 2137 -- 2152

\bibitem [CR]{chatzistamatiou2012higher}
A. Chatzistamatiou, K. R{\"u}lling, {\it Higher direct images of the structure sheaf in positive characteristic}, Algebra \& Number Theory 5.6 (2012), 693 -- 775

\bibitem [Ch]{cheah_hilbert_hodge}
J. Cheah, {\it On the cohomology of the Hilbert scheme of points}, J. Algebraic Geom. 5 (1996), 479 -- 511

\bibitem [CS1]{ciliberto_sernesi_middle_deg}
C. Ciliberto, E. Sernesi, {\it Singularities of the theta divisor and congruences of planes}, J. Algebraic Geom 1.2 (1992), 231 -- 250

\bibitem [CS2]{ciliberto_sernesi_high_deg}
C. Ciliberto, E. Sernesi, {\it On the symmetric products of a curve}, Archiv der Mathematik 61.3 (1993), 285 -- 290

\bibitem [CF]{conrad_abvar_notes}
B. Conrad, lecture notes by T. Feng, {\it Abelian varieties}, \url{http://virtualmath1.stanford.edu/~conrad/249CS15Page/handouts/abvarnotes.pdf} (2015)

\bibitem [De]{debarre2001higher}
O. Debarre, {\it Higher-dimensional algebraic geometry}, Springer New York (2001)

\bibitem [EGL]{EGL}
G. Ellingsrud, L. G\"ottsche, M. Lehn, {\it On the cobordism class of the Hilbert scheme of a surface}, J. Algebraic Geom. 10 (2001), 81 -- 100 

\bibitem [EGM]{vandergeer_abvar}
B. Edixhoven, G. van der Geer, B. Moonen, {\it Abelian varieties}, \url{https://gerard.vdgeer.net/AV.pdf}

\bibitem [ESk]{ekedahl2014recovering}
T. Ekedahl, R. Skjelnes, {\it Recovering the good component of the Hilbert scheme}, Annals of Mathematics (2014), 805 -- 841

\bibitem [ESt]{ellingsrud1998intersection}
G. Ellingsrud, S. Str{\o}mme, {\it An intersection number for the punctual Hilbert scheme of a surface}, Transactions of the American Mathematical Society 350.6 (1998), 2547 -- 2552

\bibitem [EV]{esnault1992lectures}
H. Esnault, E. Viehweg, {\it Lectures on vanishing theorems}, Birkhäuser Basel (1992)

\bibitem [Fo1]{fogarty1}
J. Fogarty, {\it Algebraic families on an algebraic surface}, American Journal of Mathematics 90.2 (1968), 511 -- 521

\bibitem [Fo2]{fogarty2}
J. Fogarty, {\it Algebraic families on an algebraic surface, II, the Picard scheme of the punctual Hilbert scheme}, American Journal of Mathematics 95.3 (1973), 660 -- 687

\bibitem [Fo3]{fogarty_sym}
J. Fogarty, {\it Line bundles on quasi-symmetric powers of varieties}, Journal of Algebra 44.1 (1977), 169 -- 180

\bibitem [Fr]{friedman2012algebraic}
R. Friedman, {\it Algebraic surfaces and holomorphic vector bundles}, Springer New York (1998)

\bibitem [Fu]{fulton_intersection}
W. Fulton, {\it Intersection theory}, Springer New York (1998)

\bibitem [GSa]{garbagnati2016kummer}
A. Garbagnati, A. Sarti, {\it {Kummer surfaces and K3 surfaces with $(\Z/2\Z)^4$ symplectic action}}, The Rocky Mountain Journal of Mathematics 46.4 (2016), 1141 -- 1205

\bibitem [Go]{goren2002lectures}
E. Z. Goren, {\it Lectures on Hilbert modular varieties and modular forms}, American Mathematical Soc. (2002)

\bibitem [G{\"o}]{göttsche2006hilbert}
L. G{\"o}ttsche, {\it Hilbert schemes of zero-dimensional subschemes of smooth varieties}, Springer Berlin, Heidelberg (1994)

\bibitem [GSo] {Soergel1993}
L. G{\"o}ttsche, W. Soergel, {\it Perverse sheaves and the cohomology of Hilbert schemes of smooth algebraic surface}, Mathematische Annalen 296.2 (1993), 235 -- 245

\bibitem [Hai]{haiman1998t}
M. Haiman, {\it $t, q$-Catalan numbers and the Hilbert scheme}, Discrete Mathematics 193.1 (1998), 201 -- 224

\bibitem [Har]{hartshorne1977algebraic}
R. Hartshorne, {\it Algebraic geometry}, Springer New York (1977)

\bibitem [Hay]{hayashi_2018}
T. Hayashi, {\it Universal covering Calabi–Yau manifolds of the Hilbert schemes of $n$ points of Enriques surfaces}, Asian Journal of Mathematics 21.6 (2018), 1099 -- 1120

\bibitem [HN]{hayashida_nishi}
T. Hayashida, M. Nishi, {\it Existence of curves of genus two on a product of two elliptic curves}, Journal of the Mathematical Society of Japan 17.1 (1965), 1 -- 16

\bibitem[Ia1]{iarrobino1972reducibility}
A. Iarrobino, {\it Reducibility of the families of 0-dimensional schemes on a variety}, Inventiones Mathematicae 15 (1972), 72 -- 77

\bibitem [Ia2]{iarrobino1972punctual}
A. Iarrobino, {\it Punctual Hilbert schemes}, Bulletin of the American Mathematical Society 78.5 (1972), 819 -- 823

\bibitem [Ke]{keum1997automorphisms}
J. H. Keum, {\it Automorphisms of Jacobian Kummer surfaces}, Compositio Mathematica 107.3 (1997), 269 -- 288

\bibitem [Ko1]{kondo1997automorphism}
S. Kondo, {\it The automorphism group of a generic Jacobian Kummer surface}, J. Algebraic Geometry 7 (1998), 589 -- 609

\bibitem [Ko2]{kondo_pictrivial_k3}
S. Kondo, {\it Automorphisms of algebraic K3 surfaces which act trivially on Picard groups}, Journal of the Mathematical Society of Japan 44.1 (1992), 75 -- 98

\bibitem [La]{lang1983complex}
S. Lang, {\it Complex multiplication}, Springer New York (1983)

\bibitem [Lee]{lee2024automorphisms}
K. Lee, {\it Automorphisms of Hilbert schemes of Cayley's K3 surfaces}, preprint, \texttt{arXiv:2403.07399} (2024)

\bibitem [Leh]{Lehn_1999}
M. Lehn, {\it Chern classes of tautological sheaves on Hilbert schemes of points on surfaces}, Inventiones Mathematicae 136.1 (1999), 157 -- 207

\bibitem [Ma]{markman2006integral}
E. Markman, {\it Integral generators for the cohomology ring of moduli spaces of sheaves over Poisson surfaces}, Advances in Mathematics 208.2 (2007), 622 -- 646

\bibitem [Mart]{martens_torelli}
H. Martens, {\it An extended Torelli theorem}, American Journal of Mathematics 87.2 (1965), 257 -- 261

\bibitem [Mo]{math1984k3}
D. Morrison, {\it On K3 surfaces with large Picard number}, Inventiones Mathematicae 75 (1984), 105 -- 121

\bibitem [MFK]{mumford1994geometric}
 D. Mumford, J. Fogarty, F. Kirwan, {\it Geometric invariant theory}, Springer Berlin, Heidelberg (1994)

\bibitem [MMR]{mumford2008abelian}
D. Mumford {\it Abelian varieties}, Tata Inst. Fundam. Res. Stud. Math. 5 (1970)

\bibitem [Mu1]{mumford_on_equations}
D. Mumford, {\it On the equations defining abelian varieties. I}, Inventiones Mathematicae 1.4 (1966), 287 -- 354

\bibitem [Mu2]{mumford_prym}
D. Mumford, {\it Prym varieties I}, Contributions to Analysis (1974), 325 -- 350

\bibitem [HNa]{nakajima1999lectures}
H. Nakajima, {\it Lectures on Hilbert schemes of points on surfaces}, American Mathematical Soc. (1999)

\bibitem [MNa] {nakamaye1996seshadri}
M. Nakamaye, {\it Seshadri constants on abelian varieties}, American Journal of Mathematics 118.3 (1996), 621 -- 635

\bibitem[MPo]{munoz_porras}
J. Muñoz Porras, {\it On the structure of the birational Abel morphism}, Mathematische Annalen 281 (1988), 1 -- 6

\bibitem [Ni]{nikulin1975kummer}
V. Nikulin, {\it On Kummer surfaces}, Mathematics of the USSR-Izvestiya 9.2 (1975), 261 -- 275

\bibitem [Ra]{ran_martens}
Z. Ran, {\it On a theorem of Martens}, Rend. Sem. Mat. Torino 44 (1986), 287 -- 291

\bibitem [RS]{rydh2010intrinsic}
D. Rydh, R. Skjelnes, {\it An intrinsic construction of the principal component of the Hilbert scheme}, Journal of the London Mathematical Society 82.2 (2010), 459 -- 481

\bibitem [Sa]{sasaki2023nonnatural}
Y. Sasaki, {\it Nonnatural automorphisms of the Hilbert scheme of two points of some simple abelian variety}, preprint, \texttt{arXiv:2311.17452} (2023)

\bibitem [MS]{MS_pushforward_isogeny}
Sasha (\url{https://math.stackexchange.com/users/335589/sasha}), Mathematics Stack Exchange \url{https://math.stackexchange.com/q/2516457}

\bibitem [St]{stacks-project}
The Stacks project, \url{https://stacks.math.columbia.edu}

\bibitem [Ti]{tikhomirov1994standard}
A. S. Tikhomirov, {\it Standard bundles on a Hilbert scheme of points on a surface}, Algebraic Geometry and its Applications (1994), 183 -- 203

\bibitem [To]{Totaro_2020}
B. Totaro, {\it The integral cohomology of the Hilbert scheme of points on a surface}, Forum of Mathematics, Sigma 8 (2020), 1 -- 6

\bibitem [Vi]{viehweg1977rational}
E. Viehweg, {\it Rational singularities of higher dimensional schemes}, Proceedings of the American Mathematical Society 63.1 (1977), 6 -- 8

\bibitem [We]{weil1957beweis}
A. Weil, {\it Zum Beweis des Torellischen Satzes}, Nachrichten der Akademie der Wissenschaften in G{\"o}ttingen 2 (1957)

\bibitem [Yo]{yoshioka2023}
K. Yoshioka, {\it Birational automorphism groups of a generalized Kummer manifold for an abelian surface with Picard number 1}, Manuscripta Mathematica 173.1 (2024), 727 -- 751

\end{thebibliography}
\end{document}